\documentclass[reqno]{amsart}

\usepackage{amssymb,amsmath,amsthm,xcolor,enumerate,comment,hyperref,cleveref}
\usepackage[normalem]{ulem}
\usepackage[utf8]{inputenc}

\usepackage{mathrsfs}

\usepackage{tikz-cd}
\usetikzlibrary{positioning}

\usepackage{amsaddr}

\allowdisplaybreaks

\newtheorem{thrm}{Theorem}[section]
\newtheorem{cor}[thrm]{Corollary}
\newtheorem{lem}[thrm]{Lemma}
\newtheorem{prop}[thrm]{Proposition}

\theoremstyle{definition}
\newtheorem{defn}[thrm]{Definition}
\newtheorem{exm}[thrm]{Example}

\newtheorem{rem}[thrm]{Remark}

\crefrangeformat{equation}{#3(#1)#4--#5(#2)#6}

\crefname{thrm}{Theorem}{Theorems}
\crefname{lem}{Lemma}{Lemmas}
\crefname{cor}{Corollary}{Corollaries}
\crefname{prop}{Proposition}{Propositions}
\crefname{defn}{Definition}{Definitions}
\crefname{exm}{Example}{Examples}
\crefname{rem}{Remark}{Remarks}
\crefname{section}{Section}{Sections}
\crefname{equation}{\unskip}{\unskip}
\crefname{enumi}{\unskip}{\unskip}
\crefname{subsection}{Subsection}{Subsections}

\makeatletter
\newcommand{\mylabel}[2]{#2\def\@currentlabel{#2}\label{#1}}
\makeatother

\renewcommand{\iff}{\Leftrightarrow}

\newcommand{\Mod}{\textbf{-}\mathbf{Mod}}
\newcommand{\rMod}{\mathbf{Mod}\textbf{-}}

\newcommand{\impl}{\Rightarrow}

\newcommand{\id}{\mathrm{id}}

\newcommand{\cI}{\mathcal{I}}
\newcommand{\cS}[1]{\mathcal S{(#1)}}

\newcommand{\cF}{\mathcal F}

\newcommand{\bbN}{\mathbb N}
\newcommand{\bbZ}{\mathbb Z}

\newcommand{\N}{\mathcal N}

\newcommand{\dom}[1]{\operatorname{\mathrm{dom}}{#1}}
\newcommand{\ran}[1]{\operatorname{\mathrm{ran}}{#1}}

\newcommand{\End}{\mathrm{End}}

\newcommand{\Hom}{\mathrm{Hom}}

\newcommand{\Ext}{\mathrm{Ext}}

\newcommand{\bd}{\mathbf{d}}
\newcommand{\br}{\mathbf{r}}

\DeclareMathOperator{\im}{im}

\newcommand{\m}{{}^{-1}}

\newcommand{\0}{\theta}

\newcommand{\af}{\alpha}
\newcommand{\bt}{\beta}
\newcommand{\lb}{\lambda}
\newcommand{\Lb}{\Lambda}
\newcommand{\gm}{\gamma}
\newcommand{\D}{\mathcal D}
\newcommand{\vf}{\varphi}
\newcommand{\sg}{\sigma}

\newcommand{\dl}{\delta}

\newcommand{\s}{\sigma}

\newcommand{\tl}{\tilde}

\newcommand{\sst}{\subseteq}
\newcommand{\ol}{\overline}

\newcommand{\kpar}[1]{K_{\mathrm{par}}(#1)}

\newcommand{\la}{\cdot}

\newcommand{\emp}{(\phantom{x})}

\newcommand{\cb}{\partial}
\newcommand{\arr}[1]{\overset{#1}{\rightarrow}}
\newcommand{\larr}[1]{\overset{#1}{\longrightarrow}}

\newcommand{\ot}{\otimes} 

\newcommand{\cG}{\mathcal{G}}

\newcommand{\sG}{\mathscr{G}}

\newcommand{\LX}{{\mathcal L} (X)}
\newcommand{\LXs}{{\mathcal L} (X_s)}
\newcommand{\LXsm}{{\mathcal L} (X_{s^{-1}})}

\begin{document}

\title[(Co)Homology of crossed products inverse monoid actions]{Homology and cohomology of crossed products by inverse monoid actions and Steinberg algebras}
	
	%
	\author{Mikhailo Dokuchaev$^a$}
	\address{$^a$Instituto de Matem\'atica e Estat\'istica, Universidade de S\~ao Paulo,  Rua do Mat\~ao, 1010, S\~ao Paulo, SP,  CEP: 05508--090, Brazil,\\ E-mail: \texttt{dokucha@gmail.com}}
	\thanks{Corresponding author: Mikhailo Dokuchaev (\texttt{dokucha@gmail.com})}
	\author{Mykola Khrypchenko$^b$}
	\address{$^b$Departamento de Matem\'atica, Universidade Federal de Santa Catarina, Campus Reitor Jo\~ao David Ferreira Lima, Florian\'opolis, SC,  CEP: 88040--900, Brazil,\\ E-mail: \texttt{nskhripchenko@gmail.com}}
	%
	\author{Juan Jacobo Sim\'on$^c$}
	\address{$^c$Departamento de Matem\'{a}ticas, Universidad de Murcia, 30071 Murcia, Spain,\\ E-mail: \texttt{jsimon@um.es}}
	%
	\begin{abstract} Given a unital action  $\0 $ of an inverse monoid $S$ on an algebra $A$ over a filed $K$ we produce (co)homology spectral sequences which converge to  the Hochschild (co)homology  of the crossed product $A\rtimes_\0 S$ with values in a bimodule over  $A\rtimes_\0 S$. The spectral sequences involve a new kind of  (co)homology of the  inverse monoid $S,$ which is based on $KS$-modules.
		The spectral sequences take especially nice form, when $(A\rtimes_\0 S)^e $ is flat as a left (homology case) or right (cohomology case) $A^e$-module, involving also the 
		Hochschild (co)homology of $A.$ Same nice spectral sequences are also obtained if $K$ is a commutative ring, over which $A$ is projective, and $S$ is   $E$-unitary. We apply our results to the Steinberg algebra $A_K(\sG)$ over a field $K$ of  an ample groupoid $\sG,$ whose unit space $\sG ^{(0)}$ is compact.
		In the homology case our spectral sequence collapses on the $p$-axis, resulting in an isomorphism between the Hochschild homology of   $A_K(\sG)$ with values in an
		$A_K(\sG)$-bimodule $M$ and the homology of the inverse semigroup of the compact open bisections of  $\sG$ with values in the coinvariant quotient of $M.$
	 \end{abstract}

	\maketitle
	\tableofcontents

\section*{Introduction} 

Spectral sequences, invented originally by Leray~\cite{Leray46a,Leray46b}, have been actively used to ``approximate'' the (co)homology of an object (e.g., a group, an algebra, a Hopf algebra, a Lie algebra, etc.) by some related sequence of cohomologies. Lyndon~\cite{Lyndon48} studied the cohomology of the direct product of groups and thus calculated $H^n(G,\bbZ)$, where $G$ is a finite abelian group. He also showed that his technique generalizes to the case of arbitrary group extensions, although the final result is not ``entirely definitive'', namely, it states that $H^n(G,K)$ is isomorphic to a quotient of $H^n(A,H^0(B,K))$, where $B$ is a normal subgroup of $G$ and $A=G/B$. Hochschild and Serre~\cite{Hochschild-Serre53} managed to express the relation between the cohomologies of $G$, $K\trianglelefteq G$ and $G/K$ in terms of a spectral sequence with $E_2^{p,q} \cong H^p(G/K, H^q(K,X))$ converging to $H^{p+q}(G,X)$. In fact, they constructed two spectral sequences: the first one was inspired by the Cartan-Leray~\cite{Cartan48,Cartan-Leray49}
spectral sequence and by a previous Serre's work~\cite{Serre50}, while the second one was built directly from a filtration of the standard cochain complex of $G$. Evens~\cite{Evens75} proved that, whenever $G$ is finite, these two spectral sequences are isomorphic. Beyl~\cite{Beyl81} generalized and completed this result by showing that both sequences are also isomorphic to a Grothendieck spectral sequence~\cite{Grothendieck57}. Barnes~\cite{Barnes85} developed a general theory of spectral sequence constructors and their comparison and applied it to prove the equivalence of several known spectral sequences associated to a Hopf algebra extension. The group and Lie algebra extension spectral sequences were thus treated as special cases. 

Given a unital associative algebra $A$ and a group $G$ acting on $A$ by unit preserving automorphisms, Nistor~\cite{Nistor90} showed that there is a spectral sequence with $E^2_{p,q}=H^p(G,H^q(A))$ converging to the cohomology of a quotient of the standard complex used to calculate $H^{p+q}(A\rtimes G)$. Sanada~\cite{Sanada93} considered twisted crossed products $\Lb\rtimes_\0 G$, where $\Lb$ is a commutative $R$-algebra, which is assumed to be a finitely generated projective module over $R$, $G$ is a finite group acting by $R$-automorphisms on $\Lb$ in a way that $\Lb^G=R$ and $\0$ is a normalized twisting with values in the group of units of $\Lb$. For any $(\Lb\rtimes_\0 G)^e$-module $A$, he constructed a spectral sequence satisfying $E^2_{p,q}=H^p(G,H^q(\Lb,A))$ and converging to $H^{p+q}(\Lb\rtimes_\0 G,A)$. For an arbitrary unital algebra $A$ over a field $k$ and a group $G$ acting on $A$ by automorphisms, Guichardet~\cite{Guichardet01} generalized the ``direct method'' by Hochschild--Serre~\cite{Hochschild-Serre53} to construct a filtration of the standard complex that calculates the Hochschild cohomology of $A\rtimes G$ with values in an $(A\rtimes G)$-bimodule $X$. The corresponding spectral sequence has $E^2_{p,q}=H^p(G,H^q(A,X))$ and converges to $H^{p+q}(A\rtimes G,X)$.

There are generalizations of the notion of a crossed product to \textit{partial group actions}~\cite{E-1,McClanahan95,DE}. They have appeared in the literature under different names depending on the context: partial crossed product~\cite{DE}, partial skew group ring~\cite{Fl}, partial semidirect product~\cite{Ara-Buss24} or partial smash product~\cite{AAR}. Given a partial action of a group $G$ on an algebra $A$, in order to describe the Hochschild cohomology of $A\rtimes G$ with values in an $(A\rtimes G)$-bimodule $M$, Alvares, Alves and Redondo~\cite{AAR} introduced a cohomology of $G$ with values in a $\kpar G$-module, where $\kpar G$ is the partial group algebra~\cite{DEP} of $G$ over a field $K.$ They called such a cohomology a \textit{partial group cohomology} of $G$\footnote{Observe that this cohomology differs from the one studied in~\cite{DK}.}, denoted it $H^\bullet_{par}(G,-)$ and  showed that there exists a Grothendieck spectral sequence
\begin{align}\label{spec-seq-AAR}
	E^2_{p,q}=H^q_{par}(G,H^p(A,M))\impl H^{p+q}(A\rtimes G,M).
\end{align}
In~\cite{DJ1}, Dokuchaev and Jerez generalized this result to a class of \textit{twisted partial actions} of $G$. Moreover, they also produced an analogous homological spectral sequence having the partial homology~\cite{ADK} of $G$ as one of the ingredients. On the other hand, in~\cite{DJ2}, the spectral sequence~\cref{spec-seq-AAR} and its homological analog were extended to \textit{partial smash products} coming from symmetric partial actions of cocommutative Hopf algebras. The \textit{partial (co)homology of Hopf algebras}, involved in the corresponding spectral sequence, was also introduced in~\cite{DJ2}.

In~\cite{Sieben97}, Sieben showed that the (partial) crossed product coming from a partial action $\af$ of a group on a $C^*$-algebra $A$ admitting a covariant representation $(\pi,u,H)$ is in fact isomorphic to the crossed product by an action of a certain inverse semigroup $S$ associated to $(\pi,u,H)$. Exel and Vieira~\cite{EV} defined the crossed product by an inverse semigroup action without using covariant representations and constructed an isomorphism between the crossed product by a partial group action of $G$ on $A$ and the crossed product by an action of the Exel's monoid~\cite{E1} $\cS G$ of $G$ on $A$. Thus, the crossed product by a partial group action turned out to be a particular case of the crossed product by an inverse monoid action. 

The purpose of this paper is to show that there is a Grothendieck spectral sequence converging to the (co)homology of the crossed product by a unital inverse monoid action  and apply them to Steinberg algebras. It is achieved by invoking the Grothendieck  theorems \cite[Theorems 10.47, 10.48]{Rotman} to  suitable pairs of functors, but some preparatory work is needed to define these functors and to show that they satisfy the required conditions and give the desired spectral sequences. 

In \cref{sec-comp-act}, we study the relationship between actions of an inverse monoid $S$ and partial actions of its maximum group image $\cG(S)$. We show in \cref{tl-0-part-act} that any compatible unital action $\0$ of $S$ on an algebra $A$ induces a partial action $\tl\0$ of $\cG(S)$ on $A$. Then, in \cref{prop:SurjectivePhi}, we construct a surjective homomorphism from the crossed product $A\rtimes_\0 S$ onto the  skew group algebra $A\rtimes_{\tl\0} \cG(S)$, which is an isomorphism, whenever $S$ is $E$-unitary.

In \cref{sec-(co)hom}, we  introduce a new (co)homology of inverse monoids whose idea comes from \cite{AAR,ADK}. For practical reasons, we construct projective resolutions of the trivial $KS$-module $KE(S)$ in the category of left (resp. right) $KS$-modules and give explicit formulas for $H^n(S,A)$ (resp. $H_n(S,A)$) coming from these resolutions.

\cref{sec:Homol} is the main ``homological'' part of our work. Given a unital partial action of an inverse monoid $S$ on a $K$-algebra $A$, where $K$ is a commutative ring, we define the right exact functors
\begin{align*}
	F_1(-) &:= A \otimes_{A^e}-: (A\rtimes_\0 S)^e\textbf{-Mod} \to KS\textbf{-Mod},\\
	F_2(-) &:= KE(S) \otimes_{KS}-: KS \textbf{-Mod} \to K\textbf{-Mod}
\end{align*}
and show that their composition $F_2F_1$ is naturally isomorphic to the functor
\begin{align*}
	F(-) := (A\rtimes_\0 S) \otimes_{(A\rtimes_\0 S)^e}-: (A\rtimes_\0 S)^e\textbf{-Mod} \to K\textbf{-Mod},
\end{align*} 
whose left-derived functor gives the Hochschild homology of $A\rtimes_\0 S$. Then, whenever $K$ is a \textit{field}, we show in \cref{teo:HLH} that $F_1$ and $F_2$ satisfy the conditions of~\cite[Theorem 10.48]{Rotman}, so for any $A\rtimes_\0 S$-bimodule $M$ there exists a first quadrant homology spectral sequence
$$
	E^2_{p,q} = H_p(S, (L_q F_1)M  ) \Rightarrow H_{p+q}(A\rtimes_\0 S, M). 
$$
In particular, if $(A\rtimes_\0 S)^e $ is flat as a left $A^e$-module, then there exists  a spectral sequence of the form
\begin{align}\label{E^2_pg=H_p(S_H_q(A_M))=>H_(p+q)(A-rt-S_M)}
	E^2_{p,q} = H_p(S, H_q (A, M)) \Rightarrow H_{p+q}(A\rtimes_\0 S, M). 
\end{align}
More specifically, if $A$ is separable over $K$, then we obtain an isomorphism
$$
H_n(S, M/[A, M]) \cong H_{n}(A\rtimes_\0 S, M).
$$
On the other hand, if $S$ is \textit{$E$-unitary}, then the assumption that $K$ is a field can be replaced by the weaker one requiring flatness of $A$ over $K$. It turns out by \cref{lem:FlatnessInE-unitaryCase} that $(A\rtimes_\0 S)^e $ is flat as a left (right) $A^e$-module in this case, so in \cref{teo:E-UnitaryHomolSpectralSeq} we again come to the spectral sequence \cref{E^2_pg=H_p(S_H_q(A_M))=>H_(p+q)(A-rt-S_M)}, but under these new assumptions on $A$ and $S$.

In \cref{sec-cohom}, we obtain cohomological analogs of the results of \cref{sec:Homol}. In the same setting as in \cref{sec:Homol}, we define the left exact functors
\begin{align*}
	T_1&:=\Hom_{A^e}(A,-): (A \rtimes_\0 S)^e\Mod\to KS\Mod,\\
	T_2&:=\Hom_{KS}(KE(S),-):KS\Mod\to K\Mod,
\end{align*}
whose composition $T_2T_1$ is proved to be naturally isomorphic to the functor
\begin{align*}
	T:=\Hom_{(A \rtimes_\0 S)^e}(A \rtimes_\0 S,-) : (A \rtimes_\0 S)^e\Mod\to K\Mod,
\end{align*}
which determines the Hochschild cohomology of $A \rtimes_\0 S$ with values in an $A \rtimes_\0 S$-bimodule. Assuming that $K$ is a field and  applying~\cite[Theorem 10.47]{Rotman}, we show in \cref{E^pq_2=>H^(p+q)} that for any $A\rtimes_\0 S$-bimodule $M$ there exists a third quadrant cohomology spectral sequence
\begin{align*}
	E^{p,q}_2=H^p(S,(R^qT_1)M)\impl H^{p+q}(A\rtimes_\0 S,M).
\end{align*} 
It follows that, whenever $(A\rtimes_\0 S)^e$ is flat as a right $A^e$-module, there is a spectral sequence of the form
\begin{align}\label{E^pq_2=H^p(S_H^q(A_M))=>H^(p+q)(A-rt-S_M)}
	E^{p,q}_2=H^p(S,H^q(A,M))\impl H^{p+q}(A\rtimes_\0 S,M).
\end{align}  
In particular, if $A$ is separable over $K,$ we come to an isomorphism 
\[
         H^n(S, M^A) \cong H^{n}(A\rtimes_\0 S, M), 
    \] where $M^A$ is the $K$-submodule of invariants of $M.$ 
The same sequence \cref{E^pq_2=H^p(S_H^q(A_M))=>H^(p+q)(A-rt-S_M)} is obtained in \cref{teo:EunitaryCohomSpectralSeq}, under the hypotheses that $S$ is $E$-unitary and $A$ is flat over a commutative ring $K$.

  In Section~\ref{sec:Steinberg} we apply our results to Steinberg algebras.     Introduced independently by Steinberg in 
\cite{St2010},  and by Clark,
 Farthing, Sims and Tomforde in \cite{ClarkFarthSimsTomf2014}, these algebras   draw much attention of algebraists and analysts. Being an algebraic counterpart of Renault’s \cite{Renault} groupoid C*-algebras, they include important classes  of algebras, such as group algebras,  Leavitt path algebras and the  Kumjian–Pask algebras \cite{ClarkFarthSimsTomf2014},  ultragraph Leavitt path algebras \cite{CastroGoncalvesWyk2011,HazratNam2023} and the crossed products by topological  partial actions  of groups on totally disconnected  locally compact Hausdorff spaces  \cite{BeuterGon2018}.

 Let $A_K(\sG)$ be the Steinberg algebra over a field $K$ of  an ample groupoid $\sG,$
  whose unit space $\sG ^{(0)}$ is compact. It is well-known that $A_K(\sG)$ is the crossed product by a unital action of the inverse monoid
  ${\mathcal S}^{a}(\sG )$ of the compact open bisections of $\sG$ on the algebra 
  ${\mathcal L}(\sG ^{(0)})$ of the locally constant functions $\sG ^{(0)}\to K.$ We show in Theorem~\ref{teo:SteinbergHomology} that  in this case
  the spectral sequence \eqref{E^2_pg=H_p(S_H_q(A_M))=>H_(p+q)(A-rt-S_M)}  is applicable and, moreover, it collapses on the $p$-axis, resulting in an isomorphism 
    \begin{equation*}
	H_{n}(A_K(\sG), M) \cong  H_n({\mathcal S}^{a}(\sG ), M/[{\mathcal L}( \sG ^{(0)}), M]),
	\end{equation*} where $M$ is an $ A_K(\sG)$-bimodule.  Our cohomology result, with $A_K(\sG)$ and $M$ as above, is the third quadrant spectral sequence 
		\begin{align*}
		E^{p,q}_2=H^p({\mathcal S}^{a}(\sG ),H^q({\mathcal L}( \sG ^{(0)}) ,M))\impl H^{p+q}( A_K(\sG),M),
	\end{align*} given in  Theorem~\ref{teo:SteinbergCoHomology}.
	

	\section{Preliminaries}\label{sec:Prelim}
	
	Recall that a semigroup $S$ is called {\it inverse}, if for any $s\in S$ there exists a unique $s\m\in S$ (the {\it inverse of} $s$), such that $ss\m s=s$ and $s\m ss\m=s\m$. Each inverse semigroup admits the \textit{natural partial order} defined by $s\le t\iff s=et$ for some idempotent $e$ (or, equivalently, $s=tf$ for some idempotent $f$). A highly important example of an inverse semigroup is  the \textit{symmetric inverse monoid}  $\cI(X)$ of  a set 
	$X,$ which consists of all bijections between the subsets of $X,$ including the empty one 
	$\emptyset \to \emptyset .$ The operation on $\cI(X)$ is given by the composition of bijections on the largest possible domain, and the natural partial order is induced by the restriction of functions.
	
	 Given an inverse semigroup $S$, we denote by $\sg$ the \textit{minimum group congruence} on $S$, which is defined by setting $(s,t)\in\sg$ if and only if there exists $u\le s,t$.  The notation $\cG(S)$ will stand for  the group $S/\sg ,$ called the {\it maximum group image} of $S$.	   An inverse semigroup $S$ is said to be {\it $E$-unitary}, whenever $(e,s)\in\sigma$ and $e\in E(S)$ imply that $s\in E(S)$ (equivalently, $e\le s\impl s\in E(S)$), where $E(S)$ denotes the subsemigroup of all idempotents of $S.$ It is well-known that  $E$-unitary inverse semigroups can also be characterized by the property that $(s,t)\in\sigma\iff s\m t,st\m\in E(S)$ (see~\cite[Theorem 2.4.6]{Lawson}).  
	   
	   By the semigroup algebra $K S$ of an inverse semigroup $S$ over a commutative unital ring $K$ we shall mean the free $K$-module with the free basis $S,$ endowed with the multiplication induced by that of $S.$
	   
	   \underline{In all what follows}, in general, $K$ will be a commutative (associative)  unital ring and $A$ a unital  associative algebra over $K.$ In some subsections, however, $K$ will be assumed to be a field.
	   
	    Recall from~\cite{E6} that a \textit{partial action} $\theta$ of a group $G$ on $A$ is a collection of algebra isomorphisms $\theta_x: D_{x^{-1}}\to D_x$, where $D_x$ is an ideal of $A$, $x\in G$, such that,
	   for all $x,y \in G,$
	\begin{enumerate}[(i)]
		\item $D_1=A$ and $\theta_1=\id_A$; 
		\item $\theta_x(D_{x^{-1}}\cap D_y)=D_x\cap D_{xy}$;
		\item $\theta_x\circ\theta_y=\theta_{xy}$ on $D_{y^{-1}}\cap D_{y^{-1}x^{-1}}$.
	\end{enumerate}
If each ideal $D_x$ is a unital algebra, i.e. $D_x$ is generated by an  idempotent of $A$ which is central in $A,$  we  say that $\theta $ is a {\it unital} partial action. Then obviously $D_x\cap D_y=D_x D_y.$ 

Replacing in the above definition $A$ by a semigroup $T$ we obtain the concept of a partial action of a  group $G$ on $T.$
	
	Given a partial action $\0$ of $G$ on $A$, recall from~\cite{DE} that the \textit{skew group algebra} $A\rtimes_\0 G$ is the $K$-space $\bigoplus_{x\in G} D_x\dl_x$ endowed with a $K$-algebra structure via $a\dl_x\cdot b\dl_y=\0_x(\0\m_x(a)b)\dl_{xy}$ (here $\dl_x$ is just a symbol). Observe that the algebra $A\rtimes_\0 G$ is associative, whenever the ideals $D_x$ are idempotent (see \cite[Corollary 3.2]{DE}). There is an (injective) algebra homomorphism $A\to A\rtimes_\0 G$, $a\mapsto a\dl_1$, giving an $A$-bimodule structure on $A\rtimes_\0 G$ in the case where $A\rtimes_\0 G$ is associative.   
	
	   Let $S$ be an inverse semigroup. Recall from~\cite{EV} that an \textit{action} of $S$ on $A$ is a homomorphism $\0:S\to\cI(A)$, $s\mapsto\0_s$,  such that $\dom\0_s$ and $\ran\0_s$ are ideals of $A$ and $\0_s$ is an algebra isomorphism for all $s\in S$. 
	   
	     \underline{In all what follows} we assume that $S$ is  an inverse  monoid with the unity element denoted by $1_S$ (or simply by $1$) and $\0$ is a \textit{unital} action of $S$ on $A$, i.e. $\0(1_S)=\id_A$ and   $\ran\0_s=1_sA,$ for some   central idempotent $1_s$ of $A.$ Observe that $1_{1_S}=1_A$, the unity element of $A$. 
	    
	    We recall from \cite{EV} the following properties of $\0 $ for all $s,t\in S:$
\begin{align*}
	1_{st}A &\subseteq 1_s A \;\;\; \text{and} \;\;\; 
	1_s 1_{st} = 1_{st};
	\end{align*}
	\begin{align} 
	s\leq t &\impl 1_{s}A \subseteq 1_t A
	\;\;\; \text{and} \;\;\;  1_s 1_t = 1_s;\label{eq:ideals natural order}\\
	 \0 _s (1_{s\m}1_t) &= 1_s 1_{st} = 1_{st}.\label{theta on units}
\end{align}

Obviously,  $\0_e$ is the identity map $1_e A \to 1_e A,$ for each idempotent $e\in S,$ and, since $\0 _s \circ \0\m_s = \0 _s \circ \0_{s\m}  =  \0_{s s\m},$ it follows  that
\begin{align}\label{1_(ss-inv)=1_s}
	1_{s s\m} = 1_s
\end{align}
for all $s\in S$, so that $\0 _{s s\m}$ is the identity map $1_s A \to 1_s A.$ 

We also have that 
\begin{equation*} \0 _s (1_{s\m t}) = 1_s 1_{t}.
\end{equation*} for all $s,t\in S.$ Indeed, by \eqref{theta on units}, 
$ 1_{s\m t} = \0 _{s\m} (1_s 1_t),  $ and applying $\0_s$ to both sides of the latter equality, we obtain
$ \0 _s(1_{s\m t}) = \0 _{ss\m} (1_s 1_t) =1_s 1_t, $
as desired.

In addition, observe that the fact that $\0$ is a homomorphism of semigroups implies that
\begin{equation}\label{eq:thetaOnProductOfIdempotents}1_{ef}= 1_ e 1 _f
\end{equation} for all $e,f \in E(S).$

\section{Compatible actions of inverse semigroups and partial actions of groups}\label{sec-comp-act}

\subsection{From compatible actions of $S$ to partial actions of $\cG(S)$}


We give the next:

\begin{defn}
	Let $\0$ be a (unital) action of $S$ on $A$. We say that $\0$ is \textit{compatible}, if for all $(s,t)\in\sg$ the partial bijections $\0_s$ and $\0_t$ agree on the intersection of their domains, i.e.
	\begin{align*}
		\0_s|_{\dom\0_s\cap\dom\0_t}=\0_t|_{\dom\0_s\cap\dom\0_t}.
	\end{align*}
\end{defn}

By \cite[Proposition 1.2.1 (2)]{Lawson} the latter is equivalent to $\0_{s\m t}$ and $\0_{st\m}$ being idempotents. In particular, any action of an $E$-unitary semigroup $S$ is compatible, as $s\m t,st\m\in E(S)$ for all $(s,t)\in\sg$.

Assume that $\0$ is compatible. For any $g\in \cG(S)$ define 
\begin{align*}
	\D_g=\sum_{s\in g} 1_sA.
\end{align*} 
Then $\D_{g\m}=\sum_{s\in g\m} 1_sA=\sum_{s\in g} 1_{s\m}A$. Given $a=\sum_{s\in g}a_{s\m}\in \D_{g\m}$ with $a_{s\m}\in 1_{s\m}A$, set
\begin{align}\label{tl-0_g(a)}
	\tl\0_g(a)=\sum_{s\in g}\0_s(a_{s\m}).
\end{align}

\begin{lem}\label{tl-0_g-bigection}
	Equality \cref{tl-0_g(a)} defines a bijection from $\D_{g\m}$ to $\D_g$ whose inverse is $\tl\0_{g\m}$.
\end{lem}
\begin{proof}
In order to show that $\tl\0_g$ is well-defined, assume that $\sum_{s\in g}a_{s\m}=0$ with $a_{s\m}\in 1_{s\m}A$ for all $s\in g$ and $\{s\in g\mid a_{s\m}\ne 0\}$ finite.  We are going to prove by induction on the cardinality of $\{s\in g\mid a_{s\m}\ne 0\}$ that $\sum_{s\in g}\0_s(a_{s\m})=0$. The base case (cardinality $1$) is trivial. For the induction step fix $s\in g$ and write 
	\begin{align*}
		0=(1_A-1_{s\m})\sum_{t\in g}a_{t\m}=\sum_{t\in g}(1_A-1_{s\m})a_{t\m},
	\end{align*}
	where $(1_A-1_{s\m})a_{t\m}\in 1_{t\m}A$ and
	\begin{align*}
		|\{t\in g\mid (1_A-1_{s\m})a_{t\m}\ne 0\}|<|\{s\in g\mid a_{s\m}\ne 0\}|,
	\end{align*}
	because $(1-1_{s\m})a_{s\m}=0$. By the induction hypothesis 
	\begin{align*}
		0 &=\sum_{t\in g}\0_t\left((1_A-1_{s\m})a_{t\m}\right)=\sum_{t\in g,t\ne s}\0_t\left((1_A-1_{s\m})a_{t\m}\right)\\
		&=\sum_{t\in g,t\ne s}\0_t\left(a_{t\m}\right)-\sum_{t\in g,t\ne s}\0_t\left(1_{s\m}a_{t\m}\right),
	\end{align*}
	so using compatibility, we have
	\begin{align*}
		\sum_{t\in g,t\ne s}\0_t\left(a_{t\m}\right)&=\sum_{t\in g,t\ne s}\0_t\left(1_{s\m}a_{t\m}\right)=\sum_{t\in g,t\ne s}\0_s\left(1_{s\m}a_{t\m}\right)\\
		&=\0_s\left(1_{s\m}\sum_{t\in g,t\ne s}a_{t\m}\right)=\0_s\left(-1_{s\m}a_{s\m}\right)=-\0_s(a_{s\m}).
	\end{align*}
	Thus, $\sum_{s\in g}\0_s(a_{s\m})=0$, so $\tl\0_g$ is well-defined. It is easy to see that $\tl\0_g(a)\in \D_g$ and that the map $\tl\0_{g\m}:\D_g\to\D_{g\m}$ is inverse to $\tl\0_g$.
		\end{proof}

In order to prove that $\{\tl\0_g\}_{g\in \cG(S)}$ is a partial action of $\cG(S)$ on the algebra $A$, we need some preparation. The following fact should be known. Since we did not find a reference for it, we include a proof for the sake of completeness.

\begin{lem}\label{R=sum-orth-ideals}
	Let $R$  be a (non-necessarily unital) ring and $e_1,\dots,e_n$ central idempotents of $R$ such that $R=\sum_{i=1}^n Re_i$. Then  $R$ is unital with 
	\begin{align}\label{unit of R}
	1_R= \sum _i e_i - \sum_{i<j} e_i e_j + \sum _{i<j<k}e_i e_j e_k +\dots + (-1)^{n+1} e_1 e_2 \cdots e_n,
	\end{align} and
	\begin{align}\label{R=Re_1+...+R.prod(1-e_i)}
R=Re_1\oplus R(1_R-e_1)e_2
 \oplus R(1_R-e_1)(1_R-e_2)e_3
 \oplus\dots\oplus R\prod_{i=1}^{n-1}(1_R-e_i)e_n.
	\end{align} 
\end{lem}
\begin{proof}
	 We use induction on $n$ to show \eqref{unit of R}. For $n=1$ there is nothing to prove. Suppose that \eqref{unit of R} holds for 
	all rings  $R'$ admitting a decomposition $R'=\sum_{i=1}^{n-1} R'e_i$ with some central idempotents $e_i \in R'$.  Now consider our ring  $R=\sum_{i=1}^{n} Re_i$.
	Then by induction the ideal $R'=\sum_{i=1}^{n-1} Re_i = \sum_{i=1}^{n-1} R' e_i  $ is a unital subring with the unity element
	\begin{align}\label{unity of R'}
	  1_{R'}=\sum _{1\leq i \leq n-1} e_i - 
	\sum_{1\leq i<j \leq n-1} e_i e_j + \sum _{1\leq i<j<k\leq n-1}e_i e_j e_k +\dots  + (-1)^{n} e_1 e_2 \cdots e_{n-1}.\end{align} Hence $R'= R 1_{R'}$ and, since  $R= R 1_{R'} + Re_{n+1},$ it is readily seen that 
	the element $ 1_{R'} + e_{n+1} -1_{R'} e_{n+1}$ is the unity element of $R.$  Then using \eqref{unity of R'}  we obtain that
	 $ 1_{R'} + e_{n+1} -1_{R'} e_{n+1}$ equals \eqref{unit of R}, as desired.
	
	 Finally, \cref{R=Re_1+...+R.prod(1-e_i)} directly follows from the following decomposition of $1_R$ into a sum of pairwise orthogonal central idempotents:
	 \begin{align}\label{expresssion for 1_R}
1_R=e_1 + (1_R-e_1)e_2
  + (1_R-e_1)(1_R-e_2)e_3
 + \dots + \prod_{i=1}^{n-1}(1_R-e_i)e_n,
	\end{align} which is readily verified by removing the parentheses in the right hand side of \eqref{expresssion for 1_R}.
	 \end{proof}

\begin{cor}\label{a=sum-a_k-with-a_k-in-R_ek}
	Under the hypotheses of \cref{R=sum-orth-ideals} each $a\in R$ is uniquely written as $a=\sum_{k=1}^n a_k$, where $a_k=a\prod_{i=1}^{k-1}(1_R-e_i)e_k\in Re_k$ for all $1\le k\le n$.
\end{cor}

\begin{lem}\label{tl-0_g-circ-tl-0_h-le-tl-0_gh}
	For all $g,h\in \cG(S)$ one has $\tl\0_g\circ\tl\0_h\le \tl\0_{gh}$.
\end{lem}
\begin{proof}
	Let $a\in\D_{h\m}\cap\D_{h\m g\m}$. We first show that $\tl\0_h(a)\in\D_{g\m}$. We have
	\begin{align*}
		a=\sum_{i=1}^m a_{s\m_i}=\sum_{j=1}^n b_{t\m_j},
	\end{align*}
	where $s_i\in gh$, $t_j\in h$, $a_{s\m_i}\in 1_{s\m_i}A$ and $b_{t\m_j}\in 1_{t\m_j}A$. Applying \cref{a=sum-a_k-with-a_k-in-R_ek} to the ideal $I=\sum_{i=1}^m 1_{s\m_i}A=\sum_{i=1}^m 1_{s\m_i}I$, we  obtain
	\begin{align*}
		a=\sum_{k=1}^m\sum_{j=1}^n b_{t\m_j}f_k,
	\end{align*}
	where, for all $1\le k\le m$,
	\begin{align*}
		f_k=\prod_{i=1}^{k-1}\left(1_I-1_{s\m_i}\right)1_{s\m_k}=\prod_{i=1}^{k-1}\left(1_{s\m_k}-1_{s\m_i}1_{s\m_k}\right)=\prod_{i=1}^{k-1}\left(1_A-1_{s\m_i}\right)1_{s\m_k}.
	\end{align*}
	Observe, using \eqref{theta on units}, that for all $1\le k\le m$ and $1\le j\le n$, 
	\begin{align*}
		\0_{t_j}\left(1_{t\m_j}f_k\right)&=\prod_{i=1}^{k-1}\0_{t_j}\left(1_{t\m_j}\left(1_A-1_{s\m_i}\right)1_{s\m_k}\right)\\
		&=\prod_{i=1}^{k-1}\0_{t_j}\left(1_{t\m_j}1_{s\m_k}-1_{t\m_j}1_{s\m_i}\cdot 1_{t\m_j}1_{s\m_k}\right)\\
		&=\prod_{i=1}^{k-1}\left(1_{t_js\m_k}-1_{t_js\m_i}\cdot 1_{t_js\m_k}\right)=\prod_{i=1}^{k-1}\left(1_A -1_{t_js\m_i}\right)1_{t_js\m_k}.
	\end{align*}
	Then, for all $1\le k\le m$ and $1\le j\le n$,
	\begin{align}\label{eq:theta t_j}
		\0_{t_j}\left(b_{t\m_j}f_k\right)=\0_{t_j}\left(b_{t\m_j}\right)\0_{t_j}\left(1_{t\m_j}f_k\right)=\0_{t_j}\left(b_{t\m_j}\right)\prod_{i=1}^{k-1}\left(1_A - 1_{t_js\m_i}\right)1_{t_js\m_k},
	\end{align}
	which belongs to $\D_{g\m}$ because $t_js\m_k\in h(gh)\m=g\m$. It follows that
	\begin{align*}
		\tl\0_h(a)=\sum_{k=1}^m\sum_{j=1}^n\0_{t_j}\left(b_{t\m_j}f_k\right)\in\D_{g\m}.
	\end{align*}
Thus, $\tl\0_h(\D_{h\m}\cap\D_{h\m g\m})\sst  \D_{h}\cap\D_{g\m}$. It is then easy to prove that $\tl\0_h(\D_{h\m}\cap\D_{h\m g\m})= \D_{h}\cap\D_{g\m},$ using the fact that $\tl\0 \m _h = \tl\0_{h\m}.$ In particular,
\begin{align*}
	\dom(\tl\0_g\circ\tl\0_h)=\tl\0_{h\m}(\D_h\cap\D_{g\m})=\D_{h\m}\cap\D_{h\m g\m}.
\end{align*} 

It remains to show that $\tl\0_g(\tl\0_h(a))=\tl\0_{gh}(a)$. Indeed, in view of \eqref{eq:theta t_j}, we see that
	\begin{align*}
		\tl\0_g(\tl\0_h(a))&=\sum_{k=1}^m\sum_{j=1}^n\0_{s_kt\m_j}\left(\0_{t_j}\left(b_{t\m_j}f_k\right)\right)=\sum_{k=1}^m\sum_{j=1}^n\0_{s_k}\left(\0_{t\m_jt_j}\left(b_{t\m_j}f_k\right)\right)\\
		&=\sum_{k=1}^m\sum_{j=1}^n\0_{s_k}\left(1_{t\m_j}b_{t\m_j}f_k\right)=\sum_{k=1}^m\sum_{j=1}^n\0_{s_k}\left(b_{t\m_j}f_k\right)=\tl\0_{gh}(a).
	\end{align*}
	\end{proof}

\begin{prop}\label{tl-0-part-act}
	 Let $\0$ be a compatible action of $S$ on $A.$ Then the family $\{\tl\0_g\}_{g\in \cG(S)}$ is a partial action of $\cG(S)$ on $A$.
\end{prop}
\begin{proof}
	This is a consequence of \cref{tl-0_g-bigection,tl-0_g-circ-tl-0_h-le-tl-0_gh}.
\end{proof}

		\subsection{The crossed product   $A\rtimes_\0 S$}
	Given an inverse monoid $S$ and a unital action $\0$ of $S$ on a $K$-algebra $A$, denote $L(A,\0,S)=\bigoplus_{s\in S} 1_sA\dl_s$, where $\delta_s$ is a symbol. Since unital ideals of an associative algebra are always $(L,R)$-associative, it follows from~\cite[Theorem 3.4]{EV} that $L(A,\0,S)$ is an associative $K$-algebra under the multiplication $a\delta_s\cdot b\delta_t=a\0_s(1_{s\m} b)\delta_{st}$. It is also an $A$-bimodule via the homomorphism of algebras $A\to L(A,\0,S)$, $a\mapsto a\dl_1$.
	
	\begin{defn}\label{defn-cross_prod_for_S-mod}
		Denote by $\N$ the ideal of $L(A,\0,S)$ generated by 
		\begin{equation}\label{rel_on_L_gen_rho}
			\{a\delta_s-a\delta_t\in L(A,\0,S)\mid a\in 1_sA\text{ and } s\le t\}.
		\end{equation}
		The quotient algebra $L(A,\0,S)/\N$ will be denoted by $A\rtimes_\0 S$ and called {\it the crossed product}  by $\0 .$ Observe that $A\rtimes_\0 S$ inherits the $A$-bimodule structure from $L(A,\0,S)$ in the natural way.
	\end{defn} 
	
	\noindent Notice that the element $a$ in \eqref{rel_on_L_gen_rho} is contained in  $1_tA$  by \eqref{eq:ideals natural order}.
	 In this case  we also have that $1_{s\m}1_{t\m}=1_{s\m}$ (because $s\m\le t\m$).
	
	\begin{lem}\label{rho_is_trans_closure}
		The set~\cref{rel_on_L_gen_rho} is invariant under the multiplication on the left and on the right in $L(A,\0,S)$, so $\N$ is the $K$-subspace of $L(A,\0,S)$ spanned by~\cref{rel_on_L_gen_rho}.
	\end{lem}
	\begin{proof}
		Denote  by $X$ the set~\cref{rel_on_L_gen_rho}. Given $a\delta_s,a\delta_t\in L(A,\0,S)$, such that $s\le t$, and an arbitrary $b\delta_u\in L(A,\0,S)$, we see that
		\[
		b\delta_u(a\delta_s-a\delta_t)=b\delta_u\cdot a\delta_s-b\delta_u\cdot a\delta_t=b\0_u(1_{u\m}a)\delta_{us}-b\0_u(1_{u\m}a)\delta_{ut}\in X,
		\]
		because $us\le ut$. Similarly,
		\[
		(a\delta_s-a\delta_t)b\delta_u=a\delta_s\cdot b\delta_u-a\delta_t\cdot b\delta_u=a\0_s(1_{s\m}b)\delta_{su}-a\0_t(1_{t\m}b)\delta_{tu},
		\]
		which belongs to $X$ if $a\0_s(1_{s\m}b)=a\0_t(1_{t\m}b)$. This is indeed the case: since $s\le t$, then $\0_s\le\0_t$ and 
		\begin{align*}
			a\0_s(1_{s\m}b)&=a\0_t(1_{s\m}1_{t\m} b)=a\0_t(1_{s\m}1_{t\m})\0_t(1_{t\m}b)\\
			&=a\0_s(1_{s\m})\0_t(1_{t\m}b)=a\0_t(1_{t\m}b).
		\end{align*}
	\end{proof}

Observe that the natural partial order $\le$ on $S$ is compatible with the multiplication on the left or on the right in $S$. Hence, the equivalence relation generated by $\le$ is a congruence on $S$. It is easy to see that this congruence coincides with $\sg$.

\begin{lem}\label{sum-s-in-C-a_s=0}
	Let $a=\sum_{s\in S}a_s\dl_s\in \N$. Then for any $\sg$-class $C$ of $S$ we have $\sum_{s\in C}a_s=0$.
\end{lem}
\begin{proof}
	Since $a\in \N$, by \cref{rho_is_trans_closure} there exists 
	$\{b_{s,t}\}_{s<t}\sst A$ such that $a=\sum_{s<t}(b_{s,t}\dl_s-b_{s,t}\dl_t)$. Then for any $s\in C$ we have
	\begin{align*}
		a_s=\sum_{s<v}b_{s,v}-\sum_{u<s}b_{u,s}.
	\end{align*}
Observe that $u<s\impl u\in C$ and $s<v\impl v\in C$. It follows that
\begin{align*}
	\sum_{s\in C}a_s=\sum_{s\in C}\left(\sum_{s<v}b_{s,v}-\sum_{u<s}b_{u,s}\right)=\sum_{s,v\in C, s<v}b_{s,v}-\sum_{u,s\in C, u<s}b_{u,s}=0.
\end{align*}
\end{proof}

\begin{rem}\label{rem:Adelta_1Embedding}  Lemma~\ref{sum-s-in-C-a_s=0} immediately implies that the homomorphism of $K$-algebras 
$A \to A\rtimes_\0 S,$ given by $a \mapsto a\dl _1 + \N $, is injective. 
\end{rem}

\begin{lem}\label{sum-a_s=0=>a-in-N}
	 Suppose that  $S$ is $E$-unitary and $a=\sum_{s\in S}a_s\dl_s\in L(A,\0,S)$. If $\sum_{s\in C}a_s=0$ for any $\sg$-class $C$ of $S$, then $a\in \N$.
\end{lem}
\begin{proof}
	It is enough to prove that $a_C:=\sum_{s\in C}a_s\dl_s\in \N$ for any  $\sg$-class $C$. The proof will be by induction on $n=|\{s\in C\mid a_s\ne 0\}|$.
	
	The case $n=1$ is impossible, as $a_C=0$ in this case. For $n=2$ we have $a_C=a_{s_1}\dl_{s_1}-a_{s_1}\dl_{s_2}$, where $s_1,s_2\in C$. Since $S$ is $E$-unitary, $s\m_1 s_2\in E(S)$. Then  $u:=s_1s\m_1 s_2\le s_1,s_2$  and
	 $$u u\m = s_1 s\m_1 ( s_2 s\m_2) (s_1 s\m_1)= 
	 s_1 s\m_1 ( s_1 s\m_1) (s_2 s\m_2)= (s_1 s\m_1 ) (s_2 s_2\m).$$ Then using \eqref{eq:thetaOnProductOfIdempotents} we see that  $1_u=1_{u u\m}=1_{s_1 s\m_1}1_{s_2 s\m_2} =1_{s_1}1_{s_2}$, so that $a_{s_1}\in 1_uA$. It follows that $a_C=(a_{s_1}\dl_{s_1}-a_{s_1}\dl_u)+(a_{s_1}\dl_u-a_{s_1}\dl_{s_2})\in \N$.
	
	Assume that $m\ge 2$ and $a_C\in \N$ for all $n\le m$.  Suppose that $n=m+1$  for our element $a$  and write $a_C=\sum_{i=1}^{m+1}a_{s_i}\dl_{s_i}$. Then $a_{s_{m+1}}=1_{s_{m+1}}a_{s_{m+1}}=-\sum_{i=1}^m 1_{s_{m+1}}a_{s_i}$, so that
	\begin{align}
		a_C&=\sum_{i=1}^m (a_{s_i}\dl_{s_i}-1_{s_{m+1}}a_{s_i}\dl_{s_{m+1}})\notag\\
		&=\sum_{i=1}^m (1_{s_{m+1}}a_{s_i}\dl_{s_i}-1_{s_{m+1}}a_{s_i}\dl_{s_{m+1}})+\sum_{i=1}^m (a_{s_i}\dl_{s_i}-1_{s_{m+1}}a_{s_i}\dl_{s_i}).\label{sum-a_s_idl_s_i=two-sums}
	\end{align}
	Observe that each summand of the first sum of \cref{sum-a_s_idl_s_i=two-sums} belongs to $\N$ by the case $n=2$. Now, the second sum of \cref{sum-a_s_idl_s_i=two-sums} equals $\sum_{i=1}^m (1_A-1_{s_{m+1}})a_{s_i}\dl_{s_i}$, where 
	$$\sum_{i=1}^m (1_A-1_{s_{m+1}})a_{s_i}=
	\sum_{i=1}^m a_{s_i} 
	-\sum_{i=1}^m 1_{s_{m+1}}a_{s_i}= \sum_{i=1}^m a_{s_i} +a_{s_{m+1}}=0,$$  so it also belongs to $\N$ by the case $n=m$. Thus, $a_C\in \N$.
\end{proof}

\begin{prop}\label{prop:SurjectivePhi}
	 Assume that  $\0 $ is a compatible  action of an inverse monoid $S$ on the algebra $A.$ Then there is a surjective algebra homomorphism $\Phi:A\rtimes_\0 S\to A\rtimes_{\tl\0} \cG(S)$ sending $a\dl_s+\N$ to $a\dl_{[s]}$, where $[s]\in\cG(S)$ is the $\sg$-class of $s$.
\end{prop}
\begin{proof}
	Let us first show that the mapping $\Phi$ is well-defined. Assume that $\sum a_s\dl_s\in \N$. Then $\sum a_s\dl_{[s]}=0$ because the coefficient of $\dl_{[s]}$ equals $\sum_{t\in[s]}a_t$, which is zero by \cref{sum-s-in-C-a_s=0}. Clearly, $\Phi$ is surjective. Now,
	\begin{align*}
		\Phi(a\dl_s\cdot b\dl_t+\N)&=\Phi(\0_s(\0_{s\m}(a)b)\dl_{st}+\N) =\0_{s}(\0_{s\m}(a)b)\dl_{[st]}\\ &
		=\tl\0_{[s]}(\tl\0_{[s]\m}(a)b)\dl_{[st]}=a\dl_{[s]}\cdot b\dl_{[t]}=
		\Phi(a\dl_s+\N)\Phi( b\dl_t+\N).
	\end{align*} 
\end{proof}

\begin{cor}\label{cor:IsoPhi}
	 Suppose that $\0 $ is a unital action of an  $E$-unitary monoid $S$ on the algebra $A.$ Then $A\rtimes_\0 S\cong A\rtimes_{\tl\0} \cG(S)$ as $K$-algebras.
\end{cor}
\begin{proof}
	 Since any action of $S$ is compatible, Proposition~\ref{prop:SurjectivePhi} is applicable, and we  only need to prove that $\Phi$ is injective. Let $a=\sum a_s\dl_s$ such that $\Phi(a+\N)=0$. Then for any $\sg$-class $C$ we have $\sum_{s\in C}a_s=0$. Thus, $a\in \N$ by \cref{sum-a_s=0=>a-in-N}.
\end{proof}

\begin{rem}\label{rem:A-bimodule mapping}
	The mapping $\Phi:A\rtimes_\0 S\to A\rtimes_{\tl\0} \cG(S)$ from \cref{prop:SurjectivePhi} is also a morphism of $A$-bimodules.
\end{rem}
\begin{proof}
	Indeed, for all $a\in A$ and $b\in 1_sA$ we have
	\begin{align*}
		\Phi(a\cdot(b\dl_s+\N))&=\Phi(a\dl_{1}\cdot b\dl_s+\N)=\Phi(ab\dl_s+\N)=ab\dl_{[s]}\\
		&=a\dl_{[1]}\cdot b\dl_{[s]}=a\cdot\Phi(b\dl_s+\N),\\
		\Phi((b\dl_s+\N)\cdot a)&=\Phi(b\dl_s\cdot a\dl_1+\N)=\Phi(\0_s(\0_{s\m}(b)a)\dl_s+\N)=\0_{[s]}(\0_{[s]\m}(b)a)\dl_{[s]}\\
		&=b\dl_{[s]}\cdot a\dl_{[1]}=\Phi(b\dl_s+\N)\cdot a.
	\end{align*}
\end{proof}

\subsection{$K S$ as skew group algebra for $E$-unitary $S$}

We know from~\cite[Theorem 3.17]{KL} (which is a reformulation of~\cite[Theorem 3.2]{Petrich-Reilly79}) that any $E$-unitary inverse semigroup $S$ is isomorphic to $E(S)\rtimes_\tau \cG(S)$, where $\tau$ is the partial action of $\cG(S)$ on $E(S)$ whose domains are $D_g=\{ss\m\mid s\in g\}$ for all $g\in\cG(S)$ and 
\begin{align}\label{tau_g(s-inv-s)=ss-inv}
	\tau_g(s\m s)=ss\m
\end{align}
for all $s\in g$. The isomorphism $\vf:S\to E(S)\rtimes_\tau \cG(S)$ maps $s\in S$ to 
\begin{align}\label{vf(s)=ss-inv.dl_[s]}
	\vf(s)=ss\m\dl_{[s]},
\end{align}
where $[s]$ is the $\sg$-class of $s$ in $S$ (see \cite[Theorem 3.2]{Petrich-Reilly79}). Clearly, $\vf$ extends by linearity to an isomorphism of $K$-algebras
\begin{equation}\label{eq:KS Iso1}
KS\to K(E(S)\rtimes_\tau \cG(S)).
\end{equation}

 \begin{lem}\label{lem:tilde par action} With the notation above, for every  $g\in \cG(S)$ let $$\tilde{\tau}_g :  K D _{g\m} \to  K D _g $$ be the linear extension of $\tau_g.$ Then the collection 
 $\tilde{\tau}$ of the $K$-linear isomorphisms $\tilde{\tau}_g$ is a partial action of $\cG(S)$ on the commutative semigroup algebra  $K E(S).$
 \end{lem}
 \begin{proof} Observe that for every $g\in \cG(S)$ the semigroup algebra $K D _g $ has local units, i.e. for any 
  $a \in K D _g $ there exists an idempotent $e\in K D_g $
  such that $a = ae.$ Indeed,  write
 $a= \sum _i \lb _i s_is\m_i,$ where $\lb _i \in K,s_i \in g.$
 Then by Lemma~\ref{R=sum-orth-ideals}, the subalgebra of
 $K D _g ,$ generated by the idempotents $s_is\m_i$ has a unity element $e$ which is an idempotent element in $K D _g,$ so that $ae=a,$ as desired. Then for any $g,h \in \cG(S)$ we have that $K D _{g} \cap K D _h = K D_g D_h,$ and the fact that $\tau $ is a partial action of  $\cG(S)$ on $E(S)$ readily implies that $\tilde{\tau} $ is a partial action of $\cG(S)$ on $K E(S).$
 \end{proof}

In view of Lemma~\ref{lem:tilde par action} we may also consider the   obvious $K$-algebra isomorphism 
$$
K(E(S)\rtimes_\tau \cG(S))\cong (KE(S))\rtimes_{\tilde{\tau}} \cG(S).
$$
 The latter isomorphism acts as the identity map on the $K$-basis $\{e\dl_g\mid g\in \cG(S),\ e\in D_g\}$ of both $K(E(S)\rtimes_\tau \cG(S))$ and $(KE(S))\rtimes_{\tl\tau} \cG(S)$. Thus, composing it with \eqref{eq:KS Iso1}, we obtain a $K$-algebra isomorphism 
 $$
 KS\to (KE(S))\rtimes_{\tl\tau} \cG(S),\footnote{One can also obtain this isomorphism from \cref{cor:IsoPhi} by showing that $\tl\tau=\tl\0$ and $KS\cong (KE(S))\rtimes_\0 S$, where $\0$ is the natural action of $S$ on $KE(S)$.}
 $$ 
 sending $s$ to $ss\m\dl_{[s]}$, which is now seen as an element of $(KE(S))\rtimes_{\tl\tau} \cG(S)$. Let us denote this isomorphism by the same letter $\vf$, so that \cref{vf(s)=ss-inv.dl_[s]} holds.

\begin{lem}\label{KS-cong-KE(S)-rtimes-G(S)}
	The $K$-algebra isomorphism $\vf:KS\to (KE(S))\rtimes_{\tl\tau} \cG(S)$ is also an isomorphism of $KE(S)$-bimodules.
\end{lem}
\begin{proof}
	Let $e\in E(S)$ and $s\in S$. Since $[es]=[se]=[s]$, we have 
	\begin{align*}
		e\cdot\vf(s)&\overset{\cref{vf(s)=ss-inv.dl_[s]}}{=}e\dl_{[1]}\cdot ss\m \dl_{[s]}=ess\m \dl_{[s]}=(es)(es)\m\dl_{[s]}\overset{\cref{vf(s)=ss-inv.dl_[s]}}{=}\vf(es),\\
		\vf(s)\cdot e&\overset{\cref{vf(s)=ss-inv.dl_[s]}}{=}ss\m\dl_{[s]}\cdot e\dl_{[1]}=\tl\tau_{[s]}(\tl\tau\m_{[s]}(ss\m)e)\dl_{[s]}\overset{\cref{tau_g(s-inv-s)=ss-inv}}{=}\tl\tau_{[s]}(s\m se)\dl_{[s]}\\
		&=\tl\tau_{[s]}((se)\m (se))\dl_{[s]}\overset{\cref{tau_g(s-inv-s)=ss-inv}}{=}(se)(se)\m\dl_{[s]}\overset{\cref{vf(s)=ss-inv.dl_[s]}}{=}\vf(se).
	\end{align*}
\end{proof}

\cref{KS-cong-KE(S)-rtimes-G(S)} and \cite[Proposition 5.3]{DJ2} imply the following.

\begin{cor}\label{cor:KS pojective over KR(S)} Let $S$ be an $E$-unitary inverse semigroup. Then $K S$ is projective as a left 
$K E(S)$-module and as a right $K E(S)$-module.
\end{cor}

		\section{(Co)homology of an inverse monoid $S$ with values in a $K S$-module}\label{sec-(co)hom}
	Our study of homology and cohomology of the crossed product by an action of an inverse monoid on an algebra involves (co)homology of inverse monoids with values in a module over its semigroup algebra, which we introduce next. It is inspired by the concept of partial  group cohomology  defined in \cite{AAR} and that of  partial  group homology analogously defined in \cite{ADK}.

	

		 Given an inverse monoid $S$, there is a 
		left $K S$-module structure on $KE(S)$, defined by
		\begin{align}\label{left KS-mod KE(S)}
			s\cdot e=ses\m,
		\end{align}
		where $s\in S$ and $e\in E(S)$.	 Similarly, $KE(S)$ is a right $KS$-module by
		\begin{equation}\label{right KS-mod KE(S)}
			e \cdot s =  s\m e s.
		\end{equation} 
	
		This  $K S$-module will be called \textit{trivial}, because in the definition of the (co)homology below it appears in the place  usually occupied by the trivial module in the traditional sense.
	
	 \begin{defn}\label{def:Inv Monoid (co)homol} 
    Let $S$ be  an inverse monoid. Given a left $K S$-module   $V$ and $n\in\bbN$\footnote{Here and below the set of natural numbers $\bbN$ includes zero.} we define the  \textit{$n$-th homology group of $S$ with values in $V$} by setting
    $$H_n(S,V)= \operatorname{Tor}^{K S}_n(K E(S), V). $$
     Analogously, the \textit{$n$-th cohomology group of $S$ with values in $V$} is
     \begin{align*}
			H^n(S,V) =\Ext^n_{KS}(KE(S),V).
		\end{align*}
	\end{defn}

	Given an inverse monoid $S$ and $s\in S$, we write
	\begin{align*}
		\bd(s)&=s\m s\mbox{ and }\br(s)=ss\m.
	\end{align*}
	
	Since both $ \operatorname{Tor}$ and  
	$\operatorname{Ext}$ may be computed using projective resolutions of the first variable, we proceed by constructing one of the trivial left $K S$-module $K E(S).$
	\begin{defn}
		Let $S$ be an inverse monoid and $n\in \bbN $. We define the following sequence of projective (see \cite[Lemma 2.8]{DKS2}) left $K S$-modules:
		\begin{align*}
			P_0&=KS,\\
			P_n&=\bigoplus_{s_1,\dots,s_n\in S}KS\br(s_1\dots s_n),\ n> 0.
		\end{align*}
	\end{defn}
	As in \cite[Remark 2.10]{DKS2}, for $n>0$ one shows that $P_n$ is (isomorphic to) the free $K$-module with free basis
	\begin{align*}
		\{t(s_1,\dots,s_n)\mid t,s_1,\dots,s_n\in S,\ \bd(t)\le\br(s_1\dots s_n)\}.
	\end{align*}
	We shall identify the $n$-tuple $(s_1,\dots,s_n)$ with the element $\br(s_1\dots s_n)(s_1,\dots,s_n)$ of $P_n$, $n>0$. Thus, we may write $t(s_1,\dots,s_n)$ for \textit{arbitrary} $t,s_1,\dots,s_n\in S$, meaning $t\br(s_1\dots s_n)(s_1,\dots,s_n)\in P_n$. This does not lead to a confusion, because, whenever $t(s_1,\dots,s_n)\in P_n$, one has
	\begin{align}\label{eq:t=t range}
		t=t\br(s_1\dots s_n).
	\end{align}
	On the other hand, if we do not impose any condition on $t,s_1,\dots,s_n$ in $t(s_1,\dots,s_n)$, then we may have $t(s_1,\dots,s_n)=v(u_1,\dots,u_n)$ for $t\ne v$. More precisely,
	\begin{align*}
	t(s_1,\dots,s_n)=v(u_1,\dots,u_n)\iff
	\begin{cases}
	(s_1,\dots,s_n)=(u_1,\dots,u_n),\\
	t\br(s_1\dots s_n)=v\br(u_1\dots u_n).
	\end{cases}
	\end{align*}
	
	For uniformity, we also represent $P_0$ as the free $K$-module with free basis
	\begin{align*}
		\{t\emp\mid t\in S\}.
	\end{align*}
	
	\begin{defn}\label{d_n:P_n->P_(n+1)-defn}
		We define the following morphisms of left $KS$-modules $\cb_0:P_0\to KE(S)$ and $\cb_n:P_n\to P_{n-1}$, $n>0$:
		\begin{align}
		\cb_0(s\emp)&=\br(s),\ s\in S,\label{d_0(s())=ss^(-1)}\\
		\cb_1(t(s))&=t(s\emp-\emp),\ t(s)\in P_1,\notag
			\end{align}
			\begin{align}
		\cb_n(t(s_1,\dots,s_n))&=t(s_1(s_2,\dots,s_n)\notag\\
		&\quad+\sum_{i=1}^{n-1}(-1)^i(s_1,\dots,s_is_{i+1},\dots,s_n)\notag\\
		&\quad+(-1)^n(s_1,\dots,s_{n-1})),\label{cb_n(t(s_1...s_n)=...)}
		\end{align}	
		where $n>1$ and $t(s_1,\dots,s_n)\in P_n$.
	\end{defn}

	In order to prove that $\{P_n\}_{n\ge 0}$ is a projective resolution of $KE(S)$, we introduce the next.
	\begin{defn}\label{sigma_n-defn}
		We define the following morphisms of $K$-modules: $\s_{-1}:KE(S)\to P_0$ and $\s_n:P_n\to P_{n+1}$, $n\in \bbN $, where
		\begin{align}
		\s_{-1}(e)&=e\emp,\ e\in E(S),\label{sigma_(-1)(e)=e()}\\
		\s_0(s\emp)&=(s),\ s\in S,\notag\\
		\s_n(t(s_1,\dots,s_n))&=(t,s_1,\dots,s_n),\label{sigma_n(t(s_1...s_n))}
		\end{align}
		where $n>0$ and $t(s_1,\dots,s_n)\in P_n$.
	\end{defn}
	
	\begin{lem}
		The following equalities hold
		\begin{align}
		\cb_0\circ\s_{-1}&=\id_{KE(S)},\label{d_0-sigma_(-1)=id}\\
		\cb_{n+1}\circ\s_n+\s_{n-1}\circ\cb_n&=\id_{P_n},\ n\ge 0.\label{d_(n+1)-sigma_n+sigma_(n-1)-d_n=id}
		\end{align}
	\end{lem}
	\begin{proof}
		Equality \cref{d_0-sigma_(-1)=id} is an easy consequence of \cref{sigma_(-1)(e)=e(),d_0(s())=ss^(-1)}. Now, taking an arbitrary $s\in S$, we have
		\begin{align*}
			(\cb_1\circ\s_0+\s_{-1}\circ\cb_0)(s\emp)&=\cb_1(\br(s)(s))+\s_{-1}(\br(s))\\
			&=\br(s)(s\emp-\emp)+\br(s)\emp\\
			&=s\emp.
		\end{align*}
		Given $n>0$ and $t(s_1,\dots,s_n)\in P_n$, we observe, keeping in mind \eqref{eq:t=t range}, that $\br(ts_1\dots s_n)=\br(t).$  Then by \cref{sigma_n(t(s_1...s_n)),cb_n(t(s_1...s_n)=...)} we have
		\begin{align}
			(\cb_{n+1}\circ\s_n)(t(s_1,\dots,s_n))&=\cb_{n+1}(\br(t)(t,s_1,\dots,s_n))\notag\\
			&=\br(t)(t(s_1,\dots,s_n)-(ts_1,s_2,\dots,s_n)\notag\\
			&\quad+\sum_{i=1}^{n-1}(-1)^{i+1}(t,s_1,\dots,s_is_{i+1},\dots,s_n)\notag\\
			&\quad+(-1)^{n+1}(t,s_1,\dots,s_{n-1}))\notag\\
			&=t(s_1,\dots,s_n)-(ts_1,s_2,\dots,s_n)\notag\\
			&\quad+\sum_{i=1}^{n-1}(-1)^{i+1}(t,s_1,\dots,s_is_{i+1},\dots,s_n)\notag\\
			&\quad+(-1)^{n+1}(t,s_1,\dots,s_{n-1}).\label{cb_(n+1)s_n(t(s_1...s_n))}
		\end{align}
		Furthermore,
		\begin{align}
			(\s_{n-1}\circ\cb_n)(t(s_1,\dots,s_n))&=\s_{n-1}(t(s_1(s_2,\dots,s_n)\notag\\
			&\quad+\sum_{i=1}^{n-1}(-1)^i(s_1,\dots,s_is_{i+1},\dots,s_n)\notag\\
			&\quad+(-1)^n(s_1,\dots,s_{n-1})))\notag\\
			&=(ts_1,s_2,\dots,s_n)\notag\\
			&\quad+\sum_{i=1}^{n-1}(-1)^i(t,s_1,\dots,s_is_{i+1},\dots,s_n)\notag\\
			&\quad+(-1)^n(t,s_1,\dots,s_{n-1})),\label{s_(n-1)cb_n(t(s_1...s_n))}
		\end{align}
		because $ts_1(s_2,\dots,s_n),t(s_1,\dots,s_{n-1})\in P_{n-1}$ (the fact that $t(s_1,\dots,s_is_{i+1},\dots,s_n)\in P_{n-1}$ is obvious). Indeed,
		\begin{align*}
			\bd(ts_1)&=s\m_1\bd(t)s_1\le s\m_1\br(s_1\dots s_n)s_1=\bd(s_1)\br(s_2\dots s_n)\le\br(s_2\dots s_n),\\
			\bd(t)&\le\br(s_1\dots s_n)\le\br(s_1\dots s_{n-1}).
		\end{align*}
		Thus, adding \cref{s_(n-1)cb_n(t(s_1...s_n)),cb_(n+1)s_n(t(s_1...s_n))}, we obtain $t(s_1,\dots,s_n)$. 
	\end{proof}

	\begin{prop}\label{P_n-is-res-of-KE(S)}
		The sequence
		\begin{align}\label{...->P_1->P_0->KE(S)}
		\dots\arr{\cb_2}P_1\arr{\cb_1}P_0\arr{\cb_0} KE(S)\to 0
		\end{align}
		is a projective resolution of $KE(S)$ in the category of left $K S$-modules.
	\end{prop}
	\begin{proof}
		Indeed, \cref{...->P_1->P_0->KE(S)} is exact in $KE(S)$ in view of \cref{d_0-sigma_(-1)=id}. It also follows from \cref{d_(n+1)-sigma_n+sigma_(n-1)-d_n=id} that $\ker\cb_n\sst\im\cb_{n+1}$ for all $n\ge 0$. However, we cannot use \cref{d_(n+1)-sigma_n+sigma_(n-1)-d_n=id} to prove the converse inclusion, because $\s_n(P_n)$ does not generate $P_{n+1}$ as a $K S$-module. So, let us do this by a straightforward computation. Applying $\cb_n\circ\cb_{n+1}$ to a generator $(s_1,\dots,s_{n+1})$ of the $S$-module $P_{n+1}$, we obtain
		\begin{align}
			\cb_n\circ\cb_{n+1}(s_1,\dots,s_{n+1})&=\cb_n(\br(s_1\dots s_{n+1})(s_1(s_2,\dots,s_{n+1})\notag\\
			&\quad+\sum_{i=1}^n(-1)^i(s_1,\dots,s_is_{i+1},\dots,s_{n+1})\notag\\
			&\quad+(-1)^{n+1}(s_1,\dots,s_n)))\notag\\
			&=\br(s_1\dots s_{n+1})s_1\br(s_2\dots s_{n+1})(s_2(s_3,\dots,s_{n+1})\label{s_1(s_2(s_3...s_(n+1)))}\\
			&\quad+\sum_{j=2}^n(-1)^{j+1}(s_2,\dots,s_js_{j+1},\dots,s_{n+1})\label{sum(-1)^n(s_2...s_js_(j+1)...s_(n+1))}\\
			&\quad+(-1)^n(s_2,\dots,s_n))\label{(-1)^n(s_2...s_n)}\\
			&\quad+\sum_{i=1}^n(-1)^i\br(s_1\dots s_{n+1})\cb_n(s_1,\dots,s_is_{i+1},\dots,s_{n+1})\label{sum(-1)^icb_n(s_1...s_is_(i+1)...s_(n+1))}\\
			&\quad+(-1)^{n+1}\br(s_1\dots s_{n+1})\cb_n(s_1,\dots,s_n))).\label{(-1)^(n+1)cb_n(s_1...s_(n+1))}
		\end{align}
		Observe that $\br(s_1\dots s_{n+1})s_1=s_1\br(s_2\dots s_{n+1})$, so $\br(s_1\dots s_{n+1})s_1\br(s_2\dots s_{n+1})=\br(s_1\dots s_{n+1})s_1$ and
		\begin{align*}
			\br(s_1\dots s_{n+1})s_1\br(s_2\dots s_{n+1})s_2&=s_1\br(s_2\dots s_{n+1})s_2\\
			&=s_1s_2\br(s_3\dots s_{n+1})\\
			&=\br(s_1\dots s_{n+1})s_1s_2.
		\end{align*}
		It follows that the term \cref{s_1(s_2(s_3...s_(n+1)))} will be canceled with the first term of the expansion of the 1-st summand (corresponding to $i=1$) of \cref{sum(-1)^icb_n(s_1...s_is_(i+1)...s_(n+1))}; the $j$-th summand of \cref{sum(-1)^n(s_2...s_js_(j+1)...s_(n+1))} will be canceled with the first term of the expansion of the $j$-th summand (corresponding to $i=j$) of \cref{sum(-1)^icb_n(s_1...s_is_(i+1)...s_(n+1))}; and the term \cref{(-1)^n(s_2...s_n)} will be canceled with the first term of the expansion of \cref{(-1)^(n+1)cb_n(s_1...s_(n+1))}. The rest of the terms of \cref{sum(-1)^icb_n(s_1...s_is_(i+1)...s_(n+1)),(-1)^(n+1)cb_n(s_1...s_(n+1))} will be exactly the same as in the classical case, but multiplied by $\br(s_1\dots s_{n+1})$ on the left. Hence, they will also be canceled. Thus, $\cb_n\circ\cb_{n+1}=0$ proving $\im\cb_{n+1}\sst\ker\cb_n$.
	\end{proof}
	
	\begin{defn}\label{C^n_par(G_A)-defn}
		Given a left $K S$-module $V$ and $n\in \bbN$, we define the following $K$-modules
		\begin{align*}
		C^0(S,V)&=V,\\
		C^n(S,V)&=\{\s:S^n\to V\mid \s(s_1,\dots,s_n)\in \br(s_1\dots s_n)V\},\ n>0,
		\end{align*}
		and their morphisms $\dl^n:C^n(S,V)\to C^{n+1}(S,V)$:
		\begin{align}
		(\dl^0 x)(s)&=sx-\br(s)x,\ x\in C^0(S,V),\label{(dl^0m)(g)=...}\\
		(\dl^n\s)(s_1,\dots,s_{n+1})&=s_1\s(s_2,\dots,s_{n+1})\notag\\
		&\quad +\sum_{i=1}^n (-1)^i \s(s_1,\dots,s_is_{i+1},\dots,s_{n+1})\notag\\
		&\quad + (-1)^{n+1} \br(s_1\dots s_{n+1})\s(s_1,\dots,s_n),\ n>0,\ \s\in C^n(S,V).\label{(dl^nf)(g_1...g_n)=...}
		\end{align}
	\end{defn}

	\begin{prop}
		The sequence
		\begin{align}\label{C^0(S_A)->...}
		0\to C^0(S,V)\arr{\dl^0}C^1(S,V)\arr{\dl^1}\dots
		\end{align} 
		is a cochain complex of $K$-modules whose $n$-th cohomology $K$-module is isomorphic to $H^n(S,V)$.
	\end{prop}
	\begin{proof}
		The proof is similar to the proof of \cite[Lemma 2.16]{DKS2} and consists in showing that \cref{C^0(S_A)->...} is the result of application of the functor $\Hom_{KS}(-,V)$ to the projective resolution \cref{...->P_1->P_0->KE(S)}.
	\end{proof}

	\begin{defn}
		Let $V$ be a left $K S$-module and $n\in \bbN$. We introduce the following notations: $Z^n(S,V):=\ker{\dl^n}$, $n\ge 0$, and $B^n(S,V):=\im{\dl^{n-1}}$, $n>0$, where $\dl^n$ is given by \cref{(dl^0m)(g)=...,(dl^nf)(g_1...g_n)=...}.
	\end{defn}
	
	\begin{cor}
		Given  a left   $K S$-module $V$, we have $H^0(S,V)=Z^0(S,V)$ and $H^n(S,V)=Z^n(S,V)/B^n(S,V)$  for all $n\in\bbN$, $n>0$.
	\end{cor}
	
	\begin{exm}
		Let $V$ be a left $K S$-module. We have
		\begin{align*}
			Z^0(S,V)&=\{x \in V\mid \forall s\in S: sx=\br(s)x\},\\
			Z^1(S,V)&=\{\s\in C^1(S,V)\mid\forall s,t\in S:\ s\s(t)-\s(st)+\br(st)\s(s)=0\},\\
			B^1(S,V)&=\{\s\in C^1(S,V)\mid\exists x\in V\ \forall s\in S:\ \s(s)=sx-\br(s)x\},\\
			Z^2(S,V)&=\{\s\in C^2(S,V)\mid\forall s,t,u\in S:\ s\s(t,u)-\s(st,u)+\s(s,tu)-\br(stu)\s(s,t)=0\},\\
			B^2(S,V)&=\{\s\in C^2(S,V)\mid\exists \tau\in C^1(S,V)\ \forall s,t\in S:\ \s(s,t)=s\tau(t)-\tau(st)+\br(st)\tau(s)\}.
		\end{align*}
	\end{exm}

Similarly, for all $n\in\bbN$, we have the following projective right $KS$-modules:
	\begin{align*}
		P'_0&=KS,\\
		P'_n&=\bigoplus_{s_1,\dots,s_n\in S}K\bd(s_n\dots s_1)S,\ n>0.
	\end{align*}
It can be seen that, for all $n>0$, $P'_n$ is isomorphic to the free $K$-module with free basis
\begin{align*}
	\{(s_n,\dots,s_1)t\mid t,s_1,\dots,s_n\in S,\ \br(t)\le\bd(s_n\dots s_1)\}.
\end{align*}
Then, identifying $P'_0$ with the free $K$-module with free basis
\begin{align*}
	\{\emp t\mid t\in S\},
\end{align*}
and an $n$-tuple $(s_n,\dots,s_1)\in S^n$ with the element $(s_n,\dots,s_1)\bd(s_n\dots s_1)\in P'_n ,$
we can define the following morphisms of right $KS$-modules $\cb'_0:P'_0\to KE(S)$ and $\cb'_n:P'_n\to P'_{n-1}$, $n>0$:
\begin{align*}
	\cb'_0(\emp s)&=\bd(s),\ s\in S,\\
	\cb'_1((s)t)&=(\emp s-\emp)t,\ (s)t\in P'_1,\\
	\cb'_n((s_n,\dots,s_1)t)&=((s_n,\dots,s_2)s_1\notag\\
	&\quad+\sum_{i=1}^{n-1}(-1)^i(s_n,\dots,s_{i+1}s_i,\dots,s_1)\notag\\
	&\quad+(-1)^n(s_{n-1},\dots,s_1))t,
\end{align*}	 where $ (s_n,\dots,s_1)t\in P'_n,\ n>1.$
 It follows that
\begin{align}\label{...->P'_1->P'_0->KE(S)}
	\dots\arr{\cb'_2}P'_1\arr{\cb'_1}P'_0\arr{\cb'_0} KE(S)\to 0
\end{align}
	is a projective resolution of $KE(S)$ in the category of right $KS$-modules. 
	
	Let $V$ be a left $KS$-module. Applying the functor $-\ot_{KS}V$ to the sequence \cref{...->P'_1->P'_0->KE(S)} (with $KE(S)$ removed), we obtain the sequence of $K$-modules
\begin{align*}
	\dots\larr{\cb'_2\ot\id}P'_1\ot_{KS}V\larr{\cb'_1\ot\id}P'_0\ot_{KS}V\larr{\cb'_0\ot\id} 0,
\end{align*}
whose $n$-th homology $K$-module is $H_n(S,V)$, $n\in\bbN$. Observe that 
\begin{align*}
	P'_0\ot_{KS}V&=KS\ot_{KS}V\cong V,\\
	P'_n\ot_{KS}V&\cong\bigoplus_{s_1,\dots,s_n\in S}K\bd(s_n\dots s_1)S\ot_{KS}V\cong\bigoplus_{s_1,\dots,s_n\in S}\bd(s_n\dots s_1)V,\ n>0.
\end{align*}
We thus define the following $K$-modules:
\begin{align*}
	C'_0(S,V)&=V,\\
	C'_n(S,V)&=\left\{\sum_{i}(s_{n,i},\dots,s_{1,i})v_i\mid s_{1,i},\dots,s_{n,i}\in S,\ v_i\in\bd(s_{n,i}\dots s_{1,i})V\right\},\ n>0,
\end{align*}
and their morphisms $\dl'_n:C'_n(S,V)\to C'_{n-1}(S,V)$, $n>0$:
\begin{align*}
	\dl'_1((s)v)&=sv-v,\ (s)v\in C'_1(S,V),\\
	\dl'_n((s_n,\dots,s_1)v)&=(s_n,\dots,s_2)s_1v\\
	&\quad+\sum_{i=1}^{n-1}(-1)^i(s_n,\dots,s_{i+1}s_i,\dots,s_1)v\\
	&\quad+(-1)^n(s_{n-1},\dots,s_1)v,\ (s_n,\dots,s_1)v\in C'_n(S,V),\ n>1.
\end{align*}
Then set $Z_n(S,V):=\ker{\dl'_n}$, $n>0$, and $B_n(S,V):=\im{\dl'_{n+1}}$, $n\ge 0$.

\begin{cor}
	Given  a left $KS$-module $V$, we have $H_0(S,V)=V/B_0(S,V)$ and $H_n(S,V)=Z_n(S,V)/B_n(S,V)$, $n\in\bbN$, $n>0$.
\end{cor}

\begin{exm}
	Let $V$ be a left $KS$-module. We have
	\begin{align*}
		B_0(S,V)&=\left\{\sum (s_iv_i-v_i)\mid s_i\in S,\ v_i\in \bd(s_i)V\right\}\\
		&=\left\{\sum (s_iv_i-\bd(s_i)v_i)\mid s_i\in S,\ v_i\in V\right\},\\
		Z_1(S,V)&=\left\{\sum(s_i)v_i\in C'_1(S,V)\mid \sum (s_iv_i-v_i)=0\right\},\\
		B_1(S,V)&=\left\{\sum((s_{2,i})s_{1,i}v_i-(s_{2,i}s_{1,i})v_i+(s_{1,i})v_i)\mid(s_{2,i},s_{1,i})v_i\in C'_2(S,V), \forall i\right\},\\
		Z_2(S,V)&=\left\{\sum(s_{2,i},s_{1,i})v_i\in C'_2(S,V)\mid\sum((s_{2,i})s_{1,i}v_i-(s_{2,i}s_{1,i})v_i+(s_{1,i})v_i)=0\right\},\\
		B_2(S,V)&=\Big\{\sum((s_{3,i},s_{2,i})s_{1,i}v_i-(s_{3,i}s_{2,i},s_{1,i})v_i+(s_{3,i},s_{2,i}s_{1,i})v_i-(s_{2,i},s_{1,i})v_i)\mid\\
		&\quad\quad (s_{3,i},s_{2,i},s_{1,i})v_i\in C'_3(S,V),\forall i\Big\}.
	\end{align*}
\end{exm}

	\section{Homology of $A\rtimes_\0 S $}\label{sec:Homol}

	\subsection{Preparatory results}
 For an algebra $R$ over a commutative ring $K$ let $R^{\rm op}$ be its opposite algebra and denote by $R^e$ the en\-ve\-lo\-ping algebra $R \otimes _K R^{\rm op}.$ It is well-known that an $R$-bimodule $N$ is a left $R^e$-module and also a right $R^e$-module with the actions defined by 

 \begin{equation}\label{left right R^e-mod}
 (a\otimes _K b) \cdot x = a \cdot x \cdot b
 \;\;\;\; \text{and}\;\;\;\; 
 x \cdot (a\otimes _K b) = b \cdot x \cdot a,
\end{equation} $x \in N, a,b \in R.$

 The following easy property will be frequently used.
 
 \begin{lem}\label{lem:dentro fora}\cite[Lemma 2.15]{DJ1} Let  $N_1$ and $N_2$ be bimodules over a $K$-algebra $R.$ Then 
 \begin{equation}\label{dentro}
 x \cdot a \otimes _{R^e} y = x  \otimes _{R^e} a \cdot y,
  \end{equation} 
 
and 
  \begin{equation}\label{fora}
 a \cdot x \otimes _{R^e} y = x  \otimes _{R^e} y \cdot a,
  \end{equation} 
  for any $a\in R, x \in N_1, y\in N_2.$
 \end{lem}

\begin{rem}\label{rem:left R^e on (R')^e}
If $R$ is a  unital $K$-subalgebra  of a unital $K$-algebra $R' ,$   then there is the obvious   homomorphism of unital $K$-algebras 
$R^e \to (R') ^e,$ which sends $a\otimes _K b \in R^e$ to $a\otimes _K b \in (R')^e.$ This endows  $(R')^e$ with the structure of a left $R^e$-module given by 
\begin{equation}\label{eq:left R^e on (R')^e}
 (a\otimes _K b) \cdot (x \ot _K y) = ax \ot _K yb, 
\end{equation}
and with the structure of a right $R^e$-module defined by 
\begin{equation}\label{eq:right R^e on (R')^e}
  (x \ot _K y) \cdot  (a\otimes _K b) = xa \ot _K by, 
\end{equation}
where $a,b \in R$ and $x,y \in R'.$ Of course, this gives  $R$-bimodule structures on $(R') ^e.$ Obviously,  in the case of \eqref{eq:left R^e on (R')^e}  the right $R$-action on
$(R')^{\rm op}$ is given by the multiplication in $R',$ that is  $y\cdot b := yb.$ Thus  $(R')^{\rm op}$ and $R'$ are identical 
as right $R$-modules. Analogously, in the case of \eqref{eq:right R^e on (R')^e} the left $R$-modules  $(R')^{\rm op}$ and $R'$ are identical.

Evidently, we may replace the condition of $R$ being a $K$-subalgebra of $R'$ by the assumption of the existence of a $K$-algebra homomorphism $R \to R'.$ If $R \to R''$
is another  $K$-algebra homomorphism such that  the  $R$-bimodules
  $R'$ and $R''$  are isomorphic, then the left (right) $R^e$-modules 
$(R')^e$ and $(R'')^e$ are also isomorphic.
\end{rem}

As above, let $\0 =\{\0 _s : 1_{s\m}A \to 1_s A \}_{s\in S}  $ be a unital action of the inverse monoid $S$ on the algebra $A.$ Then   
$A$ becomes a left $KS$-module by means of the formula 
\begin{equation}\label{left KS-mod A}
s \cdot a = \0 _s (1_{s\m}a), \;\;\;\;\;\; s\in S, a \in A.
\end{equation}  For details see the proof of Theorem 3.7 in \cite{DK}, in particular, formula (9) in \cite{DK}.

For further use we  display the following obvious equality:

\begin{equation}\label{EmmanuelLemma(i)}
(s s\m) \cdot a  = \0_s (\0_{s\m} (1_s a))= 1_s a,
\end{equation} for any $a\in A.$

For $w \in L(A, \0 , S)$ we shall write   
$\overline{w} = w+ \N \in A\rtimes_\0 S .$

We proceed with the next easy but important for us fact.

\begin{lem}\label{lem:monoid hom} The mapping
\begin{equation}
\Gamma ^{\0} : S \to  A\rtimes_\0 S, \;\; s \mapsto \overline{1_s\delta_s},
\end{equation} is a homomorphism of monoids. 
\end{lem}
\begin{proof} It is enough to show that 
$$
\tilde{\Gamma} ^{\0} : S \to  L(A, \0 , S) , \;\; s \mapsto {1_s\delta_s},$$
 is a homomorphism of monoids. 
 Using \cref{theta on units} we have 
\begin{align*}
	\tilde{\Gamma} ^{\0} (s)  \tilde{\Gamma}^{\0} (t)&=
	 1_s\delta_s \cdot 1_t\delta_t  =  \0 _s (1_{s\m}1_t) \delta_{st} =
	1_{st} \delta_{st} = \tilde{\Gamma} ^{\0} (st),\\
	 \tl\Gamma^{\0}(1)& =1_A\dl_1=1_{L(A,\0,S)}.
\end{align*}

 \end{proof}

Let $M$ be an $A\rtimes_\0 S$-bimodule. Then it easily follows from Lemma~\ref{lem:monoid hom}  that $M$ has a structure of a left $KS$-module via 

\begin{equation}\label{left KS-mod M}
s \cdot x =  (\overline{1_s \delta_s}) \cdot  x \cdot (\overline{1_{s\m} \delta_{s\m}}),
\end{equation} for any $x\in M, s \in S.$ Symmetrically, $M$ is a right $KS$-module by means of 
\begin{equation*}
 x \cdot s =  (\overline{1_{s\m} \delta_{s\m}})  \cdot  x \cdot (\overline{1_s \delta_s}).
\end{equation*}  Furthermore, $M$ will be also considered as an $A$-bimodule with the actions defined by 
\begin{equation}\label{A-bimod M}
a \cdot x =  (\overline{a \delta_1}) \cdot  x  \;\;\; \mbox{and} \;\;\; x \cdot a = x \cdot (\overline{a \delta_1}),
\end{equation} $a\in A, x \in M.$

The following equality will be useful for us.
\begin{lem}\label{lemma:EmmanuelLemma(ii)} For any $s\in S, x \in M$ we have that
\begin{equation}\label{EmmanuelLemma(ii)}  
(ss\m )\cdot x = 1_s  x 1_s.
\end{equation} 
\end{lem} 
\begin{proof} Indeed, 
$$(ss\m )\cdot x \overset{\eqref{left KS-mod M}}{=}  
(\overline{1_s \delta_s})(\overline{1_{s^{-1}} \delta_{s^{-1}}}) \cdot x \cdot (\overline{1_s \delta_s})(\overline{1_{s^{-1}} \delta_{s^{-1}}})
 =  (\overline{1_s \delta_{ss\m}}) \cdot x \cdot 
 (\overline{1_s \delta_{ss\m}}).$$ Since $ss\m \leq 1_S$ we have  
 $1_{s} \delta _{1} - 1_{s} \delta _{ss\m}  \in {\mathcal N}, $
 so that $\overline{1_s \delta_{ss\m}} = \overline{1_s \delta_{1}},$ which implies that 
 $$(ss\m )\cdot x   =  (\overline{1_s \delta_{1}})   \cdot x \cdot 
 (\overline{1_s \delta_{1}}) \overset{\eqref{A-bimod M}}{=}   1_s x 1_s,$$ as desired.
 \end{proof}

\begin{prop}\label{prop:left KS-module A otimes M}
 The tensor product $A \otimes _{A^e} M$ is a left $KS$-module with the action 
\begin{equation}\label{left KS-module A otimes M}
s \cdot (a \otimes _{A^e} x) = s \cdot a \otimes _{A^e} s \cdot x,
\end{equation} $s \in S, a\in A, x \in M.$
\end{prop}
\begin{proof} Since $A$ and $M$ are left $KS$-modules via \cref{left KS-mod A,left KS-mod M}, we only need to check that \eqref{left KS-module A otimes M} is well-defined. For $s\in S$ consider the map 
$$f_s : A \times M \to A \otimes _{A^e} M, \;\;\; 
(a,x) \mapsto s \cdot a \otimes _{A^e} s \cdot x .$$ Then using Lemma \ref{lem:dentro fora} we see that
\begin{align*}
f_s(a\cdot (b \otimes _K c), x) & =  
s \cdot (a \cdot (b \otimes _K c)) \otimes _{A^e} s \cdot x
=  s \cdot (c  a b ) \otimes _{A^e} s \cdot x\\
& =  \0 _s   (1 _{s\m }c  a b ) \otimes _{A^e} s \cdot x
=    \0 _s  (1 _{s\m }c)  \0 _s   (1 _{s\m } a) \0 _s   (1 _{s\m } b)  \otimes _{A^e} s \cdot x\\
&=   \0 _s   (1 _{s\m } a)   \otimes _{A^e} \0 _s   (1 _{s\m } b) \cdot (s \cdot x ) \cdot \0_s   (1 _{s\m }c)\\
&= s   \cdot  a  \otimes _{A^e} 
(\overline{\0 _s   (1 _{s\m } b) \delta_1}) \cdot ((\overline{1_s \delta_s }) \cdot  x \cdot (\overline{1_{s\m} \delta_{s\m}})) \cdot 
\overline{\0_s   (1 _{s\m }c) \delta _1}\\
&= s   \cdot  a   \otimes _{A^e} 
(\overline{\0 _s   (1 _{s\m } b)  \delta_s}) \cdot  x \cdot (\overline{c 1 _{s\m } \delta_{s\m}})\\
&=  s   \cdot  a    \otimes _{A^e} 
  (\overline{1 _{s }  \delta_s}) \cdot (( \overline{ b\delta_1} ) \cdot  x \cdot 
 ( \overline{c\delta _1})) \cdot (\overline{1 _{s\m } \delta_{s\m}})\\
 &= s   \cdot  a   \otimes _{A^e} 
  s \cdot (  b \cdot  x \cdot 
  c)
  = s   \cdot  a   \otimes _{A^e} 
  s \cdot (   (b \otimes _K c) \cdot  x )\\
  &= f_s(a,  (b \otimes _K c) \cdot x).
\end{align*} It follows that $f_s$ is balanced over $A^e$ and
there is a well-defined map
$$ A \otimes _{A^e }M \to A \otimes _{A^e} M, \;\;\; 
a  \otimes _{A^e} x \mapsto s \cdot a \otimes _{A^e} s \cdot x .$$  This yields  that \eqref{left KS-module A otimes M} is well-defined.
\end{proof}

With respect to the  left $KS$-module structure on 
 $A \otimes _{A^e} M$ the following property will be useful, which is an analogue of item (iv) of \cite[Lemma 5.7]{DJ2}.
 \begin{lem}\label{lemma:EmmanuelLemma(iv)} For any $e\in E(S)$,  $a\in A$ and $x \in M$ we have that
  $$ e  \cdot (a \otimes_{A^e} x) = a \otimes_{A^e} e \cdot x = e  \cdot a \otimes_{A^e}  x.$$
  Consequently, 
   $$\omega \cdot (a \otimes_{A^e} x) = \omega \cdot a \otimes_{A^e} x = a \otimes_{A^e} \omega \cdot x,$$ for all $\omega \in K E(S).$ 
 \end{lem}
 \begin{proof} Taking the trivial twist in the proof of item  (iv) of \cite[Lemma 5.7]{DJ2}, it can be easily adapted to the case of an inverse monoid action, using \eqref{left right R^e-mod}, Lemma~\ref{lem:dentro fora}, \eqref{EmmanuelLemma(i)} and \eqref{EmmanuelLemma(ii)}. \end{proof}

\begin{rem}\label{bimod map becomes KS-mod map} It is readily seen that a homomorphism $f : M \to N$ of $A\rtimes_\0 S$-bimodules becomes a mapping of left $KS$-modules in view of \eqref{left KS-mod M}. Then  thanks to  Proposition~\ref{prop:left KS-module A otimes M} it is readily seen that $A \otimes _{A^e} f : A \otimes _{A^e} M \to A \otimes _{A^e} N$ is a homomorphism of left $KS$-modules.
\end{rem}

Proposition~\ref{prop:left KS-module A otimes M} and Remark~\ref{bimod map becomes KS-mod map} allow us to define the following right exact functor
\begin{align}\label{functor-F_1-hom}
  F_1(-):= A \otimes_{A^e}-: (A\rtimes_\0 S)^e\textbf{-Mod} \to KS\textbf{-Mod}.
\end{align}

Recall that $KE(S)$ is a left $KS$-module with the action \cref{left KS-mod KE(S)} and a right $KS$-module with the action \cref{right KS-mod KE(S)}.

The next fact is an analogue of \cite[Lemma 5.7, (iii)]{DJ2} for the case of inverse monoid actions. The proof is similar, but we include it for reader's convenience.

\begin{lem}\label{lemma:EmmanuelLemma(iii)} Let $X$ be a left $KS$-module. Then in $KE(S) \otimes_{KS} X$  we have that
\begin{equation*}
s s\m \otimes_{KS} x = 1 \otimes_{KS} s\m \cdot x = 1 \otimes_{KS} (s s\m) \cdot x,
\end{equation*}  for all $s \in S$ and $x \in X.$
\end{lem}
\begin{proof} Using \eqref{right KS-mod KE(S)} we have that 
	\begin{align*}
		s s\m \otimes_{KS} x &= (1 \cdot s\m) \otimes_{KS} x = 1  \otimes_{KS} s\m \cdot x,\\
		s s\m \otimes_{KS} x &= 1\cdot (s s\m) \otimes_{KS} x=
		1    \otimes_{KS}  (s  s\m) \cdot x, 
	\end{align*}
as desired.
 \end{proof}

 With the right $KS$-action  \eqref{right KS-mod KE(S)}, we consider the right exact functor
\begin{align}\label{functor-F_2-hom}
   F_2(-): = KE(S) \otimes_{KS}-: KS \textbf{-Mod} \to K\textbf{-Mod}. 
\end{align}
We also need the  functor 
\begin{align}\label{functor-F-hom}
   F(-) := (A\rtimes_\0 S) \otimes_{(A\rtimes_\0 S)^e}-: (A\rtimes_\0 S)^e\textbf{-Mod} \to K\textbf{-Mod},
\end{align} 
whose  left-derived functor gives the Hochschild homology of
$A\rtimes_\0 S.$ 

Thanks to Lemma~\ref{lem:monoid hom}, $A\rtimes_\0 S$ is a left $KS$-module via $s \cdot \overline{a \delta _t} = \Gamma ^{\0} (s) \; \overline{ a \delta _t}  = \overline{1_s \delta _s \cdot a\delta _t},$ that is 
\begin{equation}\label{left KS-action on cross}
s \cdot \overline{a \delta _t} 
=  \overline{\0_s (1_{s\m} a ) \delta _{st}}.
\end{equation} In particular, $A\rtimes_\0 S$ is a left $K E(S)$-module with
\begin{equation}\label{left KE(S)-action on cross}
(s s\m) \cdot \overline{a \delta _t}   = 
\overline{1_s  a \delta _t}.
\end{equation} Since any idempotent of $S$ is of the form $s s\m, $ with $s\in S,$ the equality \eqref{left KE(S)-action on cross} determines a left
$K E(S)$-action on  $A\rtimes_\0 S.$

Note that $A\rtimes_\0 S$ may be considered as an $A\rtimes_\0 S$-bimodule, but  the left $K S$-action \eqref{left KS-action on cross} on $A\rtimes_\0 S$ differs from the left $K S$-action \eqref{left KS-mod M}. In all what follows we shall consider \eqref{left KS-action on cross} as a left action of $KS$ on $A\rtimes_\0 S.$

We shall need the following.

\begin{lem} \label{prop:bimodule structure X tensor cross}
	Let $X$ be a right $K S$-module. Then the actions 
	\begin{equation}\label{left crossprod-action}
		\overline{a \delta_s} \cdot (x \otimes_{K E(S)} 
		\overline{c \delta_t}) = x \cdot s^{-1} \otimes_{K E(S)} (\overline{a \delta_s})(\overline{c \delta_t})
	\end{equation}
	and
	\begin{equation} \label{right crossprod-action}
		(x \otimes_{K E(S)} \overline{c \delta_t}) \cdot \overline{a \delta_s} = x \otimes_{K E(S)} (\overline{c \delta_t}) (\overline{a \delta_s}).
	\end{equation}
	give an  $A\rtimes_\0 S$-bimodule structure on  
	$X \otimes_{K E(S)} (A\rtimes_\0 S).$
\end{lem}

\begin{proof} Obviously, the structure of a right $A\rtimes_\0 S$-module is well-defined. As to the left action we check first that 
 $X \otimes_{K E(S)} (A\rtimes_\0 S ) $ is a left $L(A, \0 , S)$-module
 via 
 \begin{equation*}
a \delta_s \cdot (x \otimes_{K E(S)} 
\overline{c \delta_t}) = x \cdot s^{-1} \otimes_{K E(S)} (\overline{a \delta_s})(\overline{c \delta_t}).
        \end{equation*}
For any fixed $a \delta_s\in L(A, \0 , S)$ let
       $$ 
       \tilde{\psi }_{a \delta_s}: X \times (A\rtimes_\0 S) \to X \otimes_{ K E(S)} A\rtimes_\0 S
       $$ 
       be defined by   
        $$ 
        \tilde{\psi}_{a \delta_s}(x, \overline{c\delta_t}) = 
        x \cdot s^{-1} \otimes_{K E(S)}\overline{(a \delta_s)(c \delta_t)}.
         $$
   We check that $ \tilde{\psi }_{a \delta_s}$  is balanced over $K E(S).$ Indeed, for any $ t \in S$ and $c \delta _v \in   L(A, \0 , S)$ we compute 
    \begin{align*}
        \tilde{\psi}_{a \dl_s}(x \cdot (t t\m), \overline{c\dl_v})
        &= (x \cdot  (t t\m) ) \cdot {  s^{-1} } \otimes_{K E(S)} \overline{(a \dl_s)(c \dl_v)} \\
        &= (x \cdot  (t t\m s\m s s\m) \otimes_{K E(S)} \overline{ (a \dl_s)(c \dl_v) } \\  
        &= (x \cdot  s^{-1})\cdot (st t\m s\m) \otimes_{K E(S)} \overline{ (a \dl_s)(c \dl_v)} \\ 
        &= x \cdot  s^{-1}  \otimes_{K E(S)} (st t\m s\m)  \cdot \left(\overline{(a \dl_s)(c \dl_v)}\right) \\ 
        &\overset{\eqref{left KE(S)-action on cross}}{=}  x \cdot {  s^{-1} } \otimes_{K E(S)} \overline{ (1_{st} a \dl_s)(c \dl_v)} \\
        &\overset{\eqref{theta on units}}{=} x \cdot {  s^{-1} } \otimes_{K E(S)} \overline{(a \dl_s)(1_t \dl_1)(c \dl_v)} \\
        &\overset{\eqref{left KE(S)-action on cross}}{=} x \cdot {  s^{-1} } \otimes_{K E(S)} \overline{(a \dl_s)}\left({ (t t\m)} \cdot \overline{(c \dl_v)}\right) \\
        &= \tilde{\psi}_{a \dl_s}(x , { (t t\m)} \cdot  \overline{c\dl_v}).
    \end{align*}
    Hence, we have a well-defined mapping 
    \[
        \psi_{a \delta_s}: X \otimes_{K E(S)}  (A\rtimes_\0 S)  \to X \otimes_{K E(S)} (A\rtimes_\0 S )   
    \]
    determined by 
    \[
        \psi_{a \delta_s}(x \otimes_{K E(S)} \overline{c\delta_t}) = x \cdot s^{-1} \otimes_{K E(S)} \overline{(a \delta_s) (c \delta_t)}
    \]
     for all $s \in S$ and  $a \in 1_s A$.  Then it can be directly verified that 
      $$ 
      \psi_{a \dl_s} (\psi_{b \dl_t}(x \otimes_{K E(S)} \overline{c\dl_v})) = \psi_{(a \dl_s )(b \dl_t)}(x \otimes_{K E(S)} \overline{c\dl_v}) 
      $$
      for all $s, t, v \in S,$ $ a\in 1_sA$, $b\in 1_tA$, $c\in 1_vA$ and $x\in X,$
 so that    we obtain     a left 
 $ L(A, \0 , S)$-module structure on $X \otimes_{K E(S)} (A\rtimes_\0 S),$ by setting
 $$ 
 a \dl_s \cdot (x \otimes_{K E(S)} \overline{c\dl_t}) =  \psi_{a \dl_s}(x \otimes_{K E(S)} \overline{c\dl_t}).
 $$
 
 Observe now that 
 $${\mathcal N} \cdot ( X \otimes_{K E(S)} (A\rtimes_\0 S))=0. $$
 Indeed, let $s,t \in S$ be such that $s \leq t$ and let $a \in 1_sA.$ Then $1_sA \subseteq 1_t A,$ 
 $\overline{a \delta _s}= \overline{a \delta _t} ,$  $s = s s\m t$ and for any $x \in X,$ $c\in 1_vA$ we see that
 \begin{align*}
 	a \dl _s \cdot (x \ot_{K E(S)} \ol{c \dl_v}) 
 	& =  x \cdot s^{-1} \ot_{K E(S)} (\ol{a \dl_s}) (\ol{c \dl_v})\\
 	&=  x \cdot (ss\m t)^{-1} \ot_{K E(S)} (\ol{a \dl_t}) (\ol{c \dl_v})\\
 	&=  x \cdot t^{-1} \ot_{K E(S)} (ss\m )\cdot (\ol{a \dl_t}) (\ol{c \dl_v})\\
 	&\overset{\eqref{left KE(S)-action on cross}}{=}
 	x \cdot t^{-1} \ot_{K E(S)} (\ol{1_{s}  a \dl_t}) (\ol{c \dl_v})\\
 	& = x \cdot t^{-1} \ot_{K E(S)} (\ol{a \dl_t}) (\ol{c \dl_v})\\
 	&=a \delta _t  \cdot (x \otimes_{K E(S)} \overline{c\delta_v}),
 \end{align*}
 since $a \in 1_s A.$
  Consequently, $(a\dl_s-a\dl_t) \cdot (x \ot_{K E(S)} \ol{c \dl_v})=0$, and the left $ L(A, \0 , S)$-action on $X \otimes_{K E(S)} (A\rtimes_\0 S)$ induces a left  $A\rtimes_\0 S$-module structure on  $X \otimes_{K E(S)} (A\rtimes_\0 S)$ determined by
 \eqref{left crossprod-action}. 
 
  Finally, it is readily seen  that the actions \eqref{left crossprod-action} and \eqref{right crossprod-action} commute, providing an $A\rtimes_\0 S$-bimodule structure on 
 $X \otimes_{K E(S)} (A\rtimes_\0 S).$ 
     \end{proof}
     
     It evidently follows from \eqref{left crossprod-action} and \eqref{right crossprod-action} that if $f : X \to X'$ is a homomorphism of right $KS$-modules, then
     \begin{align*}
     	f \otimes_{K E(S)} (A\rtimes_\0 S): X \otimes_{K E(S)} (A\rtimes_\0 S) \to X' \otimes_{K E(S)} (A\rtimes_\0 S)
     \end{align*}
     is a homomorphism of $A\rtimes_\0 S$-bimodules. Therefore, by  Lemma~\ref{prop:bimodule structure X tensor cross} we may consider the functor 
     $$
     - \otimes_{K E(S)} (A\rtimes_\0 S): \textbf{Mod-}K S \to (A\rtimes_\0 S)^e\textbf{-Mod}.
     $$

As in \cite{DJ2}, in order to obtain the homology spectral sequence we shall use a  result by A. Grothendieck 
\cite[Theorem 10.48]{Rotman}, for which we will need to show that the composition $F_2 F_1$ is naturally isomorphic to $F$ and that $F_1(P)$ is right $F_2$-acyclic for any projective $A\rtimes_\0 S $-bimodule $P$ (the latter means any projective object in $(A\rtimes_\0 S)^e\textbf{-Mod} $). For this purpose we shall use the next proposition, which is an inverse semigroup analogue of
\cite[Proposition 5.10]{DJ2}.

\begin{prop} \label{prop:bifunctor iso}
    There is a natural isomorphism of the bifunctors
    \[
        -\otimes_{K S} (A \otimes_{A^e} - ): \rMod K S \times  (A\rtimes_\0 S)^e \Mod \to K\Mod
    \]
    and
    \[
        (- \otimes_{K E(S)} (A\rtimes_\0 S)) \otimes_{(A\rtimes_\0 S)^e}-: \rMod KS \times (A\rtimes_\0 S)^e \Mod \to K\Mod.
    \]  
\end{prop}

\begin{proof}
    Take  a right $K S$-module $X$ and an $(A\rtimes_\0 S)^e$-module $M.$ Fixing  $x \in X ,$ consider the mapping
    \begin{align*}
         A \times M & \to (X \otimes_{K E(S)} (A\rtimes_\0 S))  \otimes_{(A\rtimes_\0 S)^e}M,\\
        (a,m) & \mapsto (x \otimes_{K E(S)} \overline{1_A \delta_1}) \otimes_{(A\rtimes_\0 S)^e} a \cdot m,
    \end{align*}
    which, using Lemma~\ref{lem:dentro fora}, can be directly  verified to be $A^e$-balanced (see also the proof of \cite[Proposition 5.10]{DJ2}). Thus we have the following  well-defined function
    \begin{align*}
        \xi _{x,M}: A \otimes_{A^e} M & \to (X \otimes_{K E(S)} (A\rtimes_\0 S))  \otimes_{(A\rtimes_\0 S)^e}M, \\ 
        a \otimes_{A^e} m &\mapsto (x \otimes_{K E(S)} 
        \overline{1_A \delta_1}) \otimes_{(A\rtimes_\0 S)^e} a \cdot m.
    \end{align*} 
    Varying $x\in X ,$ consider the mapping
    \begin{align*}
        \tilde{\xi}_{(X,M)} : X \times (A\otimes_{A^e}M) & \to (X \otimes_{K E(S)}( A\rtimes_\0 S))  \otimes_{(A\rtimes_\0 S)^e}M, \\ 
        (x,a\otimes_{A^e} y) &\mapsto \xi _{x,M}(a\otimes_{A^e}y)=(x \otimes_{K E(S)} \overline{1_A \delta_1}) \otimes_{(A\rtimes_\0 S)^e} a \cdot y.
    \end{align*}
    In order  to show that  $\tilde{\xi }_{(X,M)}$ is $K S$-balanced, notice first that by \cref{1_(ss-inv)=1_s} we have $1_{s\m}=1_{ s\m s}$, and $s\m  s \leq 1$ implies
    $\overline{1_{s\m } \delta _{s\m s}} = \overline{1_{s\m } \delta _1} .$ Then for $s\in S$ and $x\in X$
     we see that
    \begin{align}
    	x \cdot s \ot_{K E(S)} \ol{1_A \dl_1}
    	&= x \cdot s (s\m s)  \ot_{K E(S)} \ol{1_A \dl_1} 
    	= x \cdot s \ot_{K E(S)} (s^{-1}s) \cdot (\ol{1_A \dl_1})\notag\\ 
    	&\overset{\eqref{left KE(S)-action on cross}}{=} x \cdot s \ot_{K E(S)} \ol{1_{s^{-1}}\dl_1} 
    	= x \cdot s \ot_{K E(S)} \ol{1_{s^{-1}}\dl_{s\m s}}\notag\\ 
    	&= x \cdot s \ot_{K E(S)} (\ol{1_{s^{-1}} \dl_{s^{-1}}})(\ol{1_s \dl_s})\notag \\ 
    	&\overset{\eqref{left crossprod-action}, \eqref{right crossprod-action}}{=}  (\ol{1_{s^{-1}} \dl_{s^{-1}}}) \cdot \big(x \ot_{K E(S)}\ol{1_A \dl_1} \big) \cdot (\ol{1_s \dl_s }). \label{xs-ot-1_A.dl_1}
  \end{align}
  Furthermore, for all $a \in A$ and $y \in M$ we have
  \begin{align}
    	(\ol{1_s \dl_s}) \cdot (a \cdot y) \cdot (\ol{1_{s\m} \dl_{s^{-1}}})&\overset{\eqref{A-bimod M}}{=}(\ol{1_s \dl_s})(\ol{a \dl_1}) \cdot y \cdot (\ol{1_{s\m} \dl_{s^{-1}}}) \notag\\ 
    	&= (\ol{\theta_s(1_{s^{-1}}a)\dl_s}) \cdot y \cdot (\ol{1_{s\m} \dl_{s^{-1}}}) \notag\\
    	&= (\ol{\theta_s(1_{s^{-1}}a) \dl_1}) (\ol{1_{s} \dl_s}) \cdot y \cdot (\ol{1_{s\m} \dl_{s^{-1}}}) \notag\\
    	& \overset{\eqref{left KS-mod M}}{=}(\ol{\theta_s(1_{s^{-1}}a) \dl_1}) ( s \cdot y) 
    	\overset{\eqref{left KS-mod A}}{=}((\ol{s \cdot a)\dl_1}) \cdot (s \cdot y).\label{1_s.dl_s(ay)1_(s-inv).dl_(s-inv)}
    \end{align}
It follows that
    \begin{align*}
    	\tilde{\xi}_{(X,M)}(x \cdot s,a\ot_{A^e}y) 
    	&= (x \cdot s \ot_{K E(S)} \ol{1_A \dl_1}) \ot_{(A\rtimes_\0 S)^e} a \cdot y \\
    	&\overset{\cref{xs-ot-1_A.dl_1}}{=} (\ol{1_{s^{-1}} \dl_{s^{-1}}}) \cdot \big(x \ot_{K E(S)}\ol{1_A \dl_1} \big) \cdot (\ol{1_s \dl_s })\ot_{(A\rtimes_\0 S)^e} a \cdot y \\
    	&\overset{\eqref{dentro},\eqref{fora}}{=}  \big(x \ot_{K E(S)}\ol{1_A \dl_1} \big) \ot_{(A\rtimes_\0 S)^e} (\ol{1_s \dl_s}) \cdot (a \cdot y) \cdot (\ol{1_{s\m} \dl_{s^{-1}}}) \\ 
    	&\overset{\cref{1_s.dl_s(ay)1_(s-inv).dl_(s-inv)}}{=}\big(x \ot_{K E(S)}\ol{1_A \dl_1} \big) \ot_{(A\rtimes_\0 S)^e} ((\ol{s \cdot a)\dl_1}) \cdot (s \cdot y)  \\ 
    	&= \tilde{\xi}_{(X,M)}(x , (s \cdot a)\ot_{A^e}(s \cdot y)) 
    	\overset{\eqref{left KS-module A otimes M}}{=} \tilde{\xi}_{(X,M)}(x ,s \cdot (a\ot_{A^e}y)),
    \end{align*}    
    as desired. Consequently, we obtain the following   well-defined
    mapping
    \begin{align*}
        \xi _{(X,M)} : X \otimes_{K S} (A\otimes_{A^e}M) & \to (X \otimes_{K E(S)} (A\rtimes_\0 S))  \otimes_{(A\rtimes_\0 S)^e}M, \\ 
        x \otimes_{K S} (a\otimes_{A^e} y) &\mapsto (x \otimes_{K E(S)} \overline{1_A \delta_1}) \otimes_{(A\rtimes_\0 S)^e} a \cdot y.
    \end{align*}

    We are going to produce an inverse  of $   \xi _{(X,M)}.$  Fix  $m \in M$ and define the mapping
    \begin{align*}
        \tilde{\eta}_{X,m}: X \times  (A\rtimes_\0 S)  & \to  X \otimes_{K S} (A\otimes_{A^e}M), \\ 
        (x, \overline{a \dl_t})&\mapsto x \otimes_{K S}(1_A \otimes_{A^e} (\overline{a \dl_t}) \cdot m).
    \end{align*}
    Then   $\tilde{\eta}_{X,m}$ is $K E(S)$-balanced. Indeed,
    \begin{align*}
        \tilde{\eta}_{X,m}(x \cdot (s s\m ) , a \dl_t)    &=  x \cdot (s s\m )  \otimes_{K S}(1_A \otimes_{A^e} (\ol{a \dl_t}) \cdot m) \\
        &=  x \otimes_{K S} (s s\m)   \cdot (1_A \otimes_{A^e} (\ol{a \dl_t}) \cdot m) \\
        (\text{by Lemma~\ref{lemma:EmmanuelLemma(iv)}} )&=  x \otimes_{K S} \big( 1_A \otimes_{A^e} s s\m \cdot( (\ol{a \dl_t}) \cdot m) \big)  \\
        &\overset{\eqref{EmmanuelLemma(ii)}}{=}  x \otimes_{K S} \big( 1_A \otimes_{A^e} 1_s ( (\ol{a \dl_t}) \cdot m) 1_s \big)  \\ 
        &= x \ot_{K S} \big(1_s 1_A 1_s \ot_{A^e}  (\ol{a \dl_t}) \cdot m \big) \\
        &= x \ot_{K S} \big(1_A 1_s \ot_{A^e}  (\ol{a \dl_t}) \cdot m \big) \\
        &\overset{\cref{dentro}}{=} x \ot_{K S} \big(1_A \ot_{A^e}  1_s((\ol{a \dl_t}) \cdot m) \big) \\
        &\overset{\cref{A-bimod M}}{=} x \ot_{K S} \big(1_A \ot_{A^e}  \ol{1_s\dl_1}((\ol{a \dl_t}) \cdot m) \big) \\
        &= x \otimes_{K S} \big(  1_A \otimes_{A^e}  (\ol{1_s a \dl_t}) \cdot m \big) \\
        &=\tilde{\eta}_{X,m}(x, \ol{1_s a \dl_t}) \overset{\eqref{left KE(S)-action on cross}}{=}\tilde{\eta}_{X,m}(x, (s s\m) \cdot (\ol{a \dl_t})).
    \end{align*}
    Thus we obtain the following   well-defined mapping
    \begin{align*}
        \eta _{X,m}: X \otimes_{K E(S)} (A\rtimes_\0 S) & \to  X \otimes_{K S} (A\otimes_{A^e}M),\\ 
        x \otimes_{K E(S)} \overline{a \dl_t} &\mapsto x \otimes_{K S}(1_A \otimes_{A^e} (\overline{a \dl_t})\cdot m).
    \end{align*}
    Then we may  consider the  mapping
    \begin{align*}
        \tilde{\eta}_{(X,M)}: (X \otimes_{K E(S)} (A\rtimes_\0 S))  \times M  & \to  X \otimes_{K S} (A\otimes_{A^e}M),\\ 
        (x \otimes_{K E(S)} \overline{a \dl_t}, m) &\mapsto \eta _{X,m}(x \otimes_{K E(S)} \overline{a \dl_t})=x \otimes_{K S}(1_A \otimes_{A^e} (\overline{a \dl_t}) \cdot m).
    \end{align*}
    Let us see that $ \tilde{\eta}_{(X,M)}$ is 
    $(A\rtimes_\0 S)^e$-balanced. On the one hand,
    \begin{align*}
        \tilde{\eta}_{(X,M)}((x \ot_{K E(S)} \ol{a \dl_t})\cdot (\ol{b \dl_s}), m) 
        &\overset{\cref{right crossprod-action}}{=}  \tilde{\eta}_{(X,M)}(x \ot_{K E(S)} (\ol{a \dl_t})(\ol{b \dl_s}), m) \\
        &= x \ot_{K S}(1_A \ot_{A^e} \left((\ol{a \dl_t})(\ol{b \dl_s})\right) \cdot m) \\ 
        &= x \ot_{K S}(1_A \ot_{A^e} (\ol{a \dl_t})\cdot ((\ol{b \dl_s}) \cdot m)) \\
        &=\tilde{\eta}_{(X,M)}(x \ot_{K E(S)} \ol{a \dl_t},  (\ol{b \dl_s })\cdot m).
    \end{align*}
    On the other hand,
    \begin{align*}
        \tilde{\eta}_{(X,M)}((\ol{b \dl_s}) \cdot (x \ot_{K E(S)} \ol{a \dl_t}), m)  & \overset{\cref{left crossprod-action}}{=} \tilde{\eta}_{(X,M)}(x \cdot s\m \ot_{K E(S)} (\ol{b \dl_s})(\ol{a \dl_t}), m) \\
        &= x \cdot s^{-1} \ot_{K S} \bigl( 1_A \ot_{A^e} (\ol{b \dl_s})(\ol{a \dl_t}) \cdot m \bigr) \\ 
        &= x \cdot s^{-1} \ot_{K S} \bigl( 1_A \ot_{A^e}(\ol{b\dl_1})(\ol{1_s \dl_s})(\ol{a \dl_t}) \cdot m \bigr)\\
        &= x \cdot s^{-1} \ot_{K S} \bigl( 1_A \ot_{A^e}(\ol{b\dl_1})\cdot\left((\ol{1_s \dl_s})(\ol{a \dl_t}) \cdot m\right) \bigr)\\
        &\overset{\cref{A-bimod M}}{=}x \cdot s^{-1} \ot_{K S} \bigl( 1_A \ot_{A^e}b\cdot\left((\ol{1_s \dl_s})(\ol{a \dl_t}) \cdot m\right) \bigr)\\
        &\overset{\eqref{dentro}}{=} x \cdot s^{-1} \ot_{K S} \bigl( b \ot_{A^e} (\ol{1_s \dl_s})(\ol{a \dl_t}) \cdot m \bigr) \\ 
        &\overset{\eqref{fora}}{=} x \cdot s^{-1} \ot_{K S} \bigl( 1_A \ot_{A^e} (\ol{1_s \dl_s})(\ol{a \dl_t}) \cdot m \cdot b\bigr) \\ 
        &\overset{\cref{A-bimod M}}{=}x \cdot s^{-1} \ot_{K S} \bigl( 1_A \ot_{A^e} (\ol{1_s \dl_s})(\ol{a \dl_t}) \cdot m \cdot (\ol{b \dl_1})\bigr) \\ 
        &= x  \ot_{K S} s^{-1} \cdot \bigl( 1_A \ot_{A^e} (\ol{1_s \dl_s})(\ol{a \dl_t}) \cdot m \cdot (\ol{b \dl_1})\bigr) \\ 
        &\overset{\eqref{left KS-module A otimes M}}{=} x  \ot_{K S} \bigl( s^{-1} \cdot  1_A \ot_{A^e} s^{-1} \cdot ((\ol{1_s \dl_s})(\ol{a \dl_t}) \cdot m \cdot (\ol{b \dl_1}))\bigr) \\ 
        &\overset{\eqref{left KS-mod M}}{=} x  \ot_{K S} \bigl( s^{-1} \cdot  1_A \ot_{A^e} (\ol{1_{s^{-1}}\dl_{s^{-1}}})(\ol{1_s \dl_s})(\ol{a \dl_t}) \cdot m \cdot (\ol{b \dl_1})(\ol{1_s \dl_s})\bigr)\\
        &= x  \ot_{K S} \bigl( s^{-1} \cdot  1_A \ot_{A^e} (\ol{1_{s^{-1}}\dl_{s^{-1}s}})(\ol{a \dl_t}) \cdot m \cdot (\ol{b \dl_s})\bigr).
\end{align*} 
Since $\overline{1_{s^{-1}}\delta_{s^{-1}s}}= \overline{1_{s^{-1}}\delta_{1}}$ and $s\m\cdot 1_A=1_{s\m}$ by \cref{left KS-mod A}, in view of \cref{A-bimod M} the last equals
 \begin{align*}
        x  \ot_{K S} \Bigl( 1_{s^{-1}} \ot_{A^e} 1_{s^{-1}}\left((\ol{a \dl_t}) \cdot m \cdot (\ol{b \dl_s})\right)\Bigr)
        &\overset{\eqref{dentro}}{=} x \ot_{K S} \Bigl(1_{s^{-1}} \ot_{A^e}  (\ol{a\dl_t}) \cdot m \cdot (\ol{b\dl_s})\Bigr) \\ 
        &\overset{\eqref{fora}}{=} x  \ot_{K S} \Bigl( 1_A \ot_{A^e} \big( (\ol{a\dl_t}) \cdot m \cdot (\ol{b\dl_s})\big)1_{s\m}\Bigr) \\  
        &\overset{\cref{A-bimod M}}{=} x  \ot_{K S} \Bigl( 1_A \ot_{A^e}  (\ol{a \dl_t}) \cdot m \cdot (\ol{b \dl_s})(\ol{1_{s^{-1}} \dl_1})\Bigr) \\  
        &= x  \ot_{K S} \Bigl( 1_A \ot_{A^e} (\ol{a \dl_t}) \cdot m \cdot (\ol{b \dl_s})\Bigr) \\
        &= \tilde{\eta}_{(X,M)}(x \ot_{K E(S)} \ol{a \dl_t}, m \cdot (\ol{b \dl_s }) ).
    \end{align*}
    Consequently, we obtain a well-defined mapping  
    \begin{align*}
        \eta_{(X,M)} :  (X \otimes_{K E(S)} (A\rtimes_\0 S))  \otimes_{(A\rtimes_\0 S)^e} M  & \to  X \otimes_{K S} (A\otimes_{A^e}M), \\ 
        (x \ot_{K E(S)} \ol{a \dl_t}) \ot_{(A\rtimes_\0 S)^e}  m &\mapsto x \ot_{K S}(1_A \ot_{A^e} (\ol{a \dl_t}) \cdot m).
    \end{align*}
     Composing $\xi_X$ and $\eta_X$ we see that
    \begin{align*}
        \eta_{(X,M)}\big( \xi _{(X,M)}(x \otimes_{K S}(a\otimes_{A^e}m))\big) &=  \eta _{(X,M)}\big( (x \otimes_{K E(S)} 1_A \delta_1) \otimes_{(A\rtimes_\0 S)^e} a \cdot m\big) \\ 
        &=x \otimes_{K S}(1_A \otimes_{A^e} a \cdot m) \\
        &\overset{\eqref{dentro}}{=}x \otimes_{K S}(a \otimes_{A^e} m)
    \end{align*}
    and
    \begin{align*}
        \xi _{(X,M)}\big( \eta _{(X,M)}((x \otimes_{K E(S)} \overline{a_t \delta_t}) 
        \otimes_{(A\rtimes_\0 S)^e} m) \big)
        &= \xi _{(X,M)} \big(x \otimes_{K S}(1_A \otimes_{A^e} (\overline{a_t \delta_t}) \cdot m) \big) \\
        &=  (x \otimes_{K E(S)} \overline{1_A \delta_1})
 \otimes_{(A\rtimes_\0 S)^e} (\overline{a_t \delta_t} ) \cdot m \\
        &\overset{\eqref{dentro}}{=}(x \otimes_{K E(S)} \overline{a_t \delta_t})
 \otimes_{(A\rtimes_\0 S)^e} m,
    \end{align*} so that  $\xi _X$ and $\eta _X$ are mutually inverse.
    
    To see that  $\xi $  is natural, let  $f: X \to X'$ be a homomorphism of right $K S$-modules and $\phi : M \to M'$ be a homomorphism of $A\rtimes_\0 S$-bimodules. Write   $\nu_{(f,\phi)}=f \otimes_{K S}(1_A \otimes_{A^e} \phi)$ and $\nu ^{(f,\phi )}=(f \otimes_{K E(S)}1_A\delta_1)
    \otimes_{(A\rtimes_\0 S)^e} \phi$. Consider the  diagram 
    \[\begin{tikzcd}
        {X \otimes_{K S} (A\otimes_{A^e}M)} & {X' \otimes_{K S} (A\otimes_{A^e}M')} \\
        {(X \otimes_{K E(S)} (A\rtimes_\0 S))  \otimes_{(A\rtimes_\0 S)^e}M} & {(X' \otimes_{K E(S)} (A\rtimes_\0 S))  \otimes_{(A\rtimes_\0 S)^e}M'.}
        \arrow["{\nu_{(f,\phi)}}", from=1-1, to=1-2]
        \arrow["{\nu ^{(f,\phi )}}", from=2-1, to=2-2]
        \arrow["{\xi_{(X,M)}}"', from=1-1, to=2-1]
        \arrow["{\xi_{(X',M')}}"', from=1-2, to=2-2]
    \end{tikzcd}\] 
    It is easily seen to be commutative:
    \begin{align*}
        \xi_{(X',M')} \left( \nu_{(f,\phi)}(x \otimes_{K S}(a \otimes_{A^e} m)) \right)
        &=\xi_{(X,M)} \left(f(x) \otimes_{K S}(a \otimes_{A^e} \phi (m))\right) \\
        &= (f(x) \otimes_{K E(S)} \overline{1_A \delta_1}) \otimes_{(A\rtimes_\0 S)^e} a \cdot \phi (m) \\
        &= (f(x) \otimes_{K E(S)} \overline{1_A \delta_1}) \otimes_{(A\rtimes_\0 S)^e} \phi (a \cdot m) \\ 
        &= \nu ^{(f,\phi )}\left( (x \otimes_{K E(S)} \overline{1_A \delta_1}) \otimes_{(A\rtimes_\0 S)^e} a \cdot m \right) \\
        &= \nu ^{(f,\phi )}\left(\xi_{(X,M)}(x \otimes_{K S}(a \otimes_{A^e} m))  \right),
    \end{align*}
implying that $\xi $ is a natural isomorphism.
\end{proof}

We proceed with the next.

\begin{lem} \label{l_AsemidirSisoBsemidirS}
    The mapping 
    \begin{align*}
        \varphi:A\rtimes_\0 S        &\to K E(S) \otimes_{K E(S)} (A\rtimes_\0 S) , \\ 
            \overline{a \delta_t}    &\mapsto 1_{K E(S)} \otimes_{K E(S)} \overline{a \delta_t},
    \end{align*}
    is an isomorphism of $A\rtimes_\0 S$-bimodules.
\end{lem}
\begin{proof}  Evidently, $\varphi $ is a homomorphism of $K$-modules. In order to see that  it is a  $A\rtimes_\0 S $-bimodule mapping, take  arbitrary  $b\in 1_vA$, $a\in 1_sA$ and $c\in 1_tA$. Recall that we consider $K E(S)$ as a right $K S$-module by means of \eqref{right KS-mod KE(S)}. Then
    \begin{align*}
        \varphi((\ol{b \dl_v})(\ol{a \dl_s})(\ol{c \dl_t}))  
        &= 1_{K E(S)} \ot_{K E(S)} (\ol{b \dl_v})(\ol{a \dl_s})(\ol{c \dl_t}) \\  
        & = 1_{K E(S)} \ot_{K E(S)} ( \ol{1_v b \dl_v})(\ol{a \dl_s})  (\ol{c \dl_t})  \\ 
        &\overset{\eqref{left KE(S)-action on cross}}{=} 1_{K E(S)} \ot_{K E(S)} (v v\m)  \cdot ((\ol{b \dl_v})(\ol{a \dl_s}) (\ol{c \dl_t}) ) \\ 
        &= v v\m \ot_{K E(S)} (\ol{b \dl_v})(\ol{a \dl_s})  (\ol{c \dl_t}) \\
        &\overset{\eqref{right crossprod-action}}{=} \big( v v\m\ot_{K E(S)} (\ol{b \dl_v})(\ol{a \dl_s})\big) \cdot \ol{c \dl_t} \\
        &\overset{\eqref{right KS-mod KE(S)}}{=} \big(  1_{K E(S)}\cdot v\m \ot_{K E(S)} (\ol{b \dl_v})(\ol{a \dl_s})\big) \cdot \ol{c \dl_t} \\ 
        &\overset{\eqref{left crossprod-action}}{=} \ol{b \dl_v} \cdot \big(1_{K E(S)} \ot_{K E(S)} \ol{a \dl_s}\big) \cdot \ol{c \dl_t} \\
        &= \ol{b \dl_v} \cdot \varphi (\ol{a \dl_s}) \cdot \ol{c \dl_t}.
    \end{align*}
    Thus, $\varphi$ is a homomorphism of $A\rtimes_\0 S$-bimodules. Since  $A\rtimes_\0 S$ is a left $K E(S)$-module by \eqref{left KE(S)-action on cross}, the mapping 
      \begin{align*}
        {\varphi}': K E(S) \otimes_{K E(S)} (A\rtimes_\0 S) &\to A\rtimes_\0 S, \\ 
        \omega \otimes_{K E(S)} \overline{a\delta_t}   &\mapsto \omega \cdot \overline{a\delta_t},
    \end{align*}
    is well-defined and easily seen to be   an inverse of $\varphi .$ Consequently, $\varphi$ is an isomorphism of $A\rtimes_\0 S $-bimodules.
\end{proof}

Recall that $F_1$, $F_2$ and $F$ are the functors defined in \cref{functor-F_1-hom,functor-F_2-hom,functor-F-hom}, respectively. As a consequence of \cref{prop:bifunctor iso,l_AsemidirSisoBsemidirS}, we obtain the following.
\begin{cor} \label{cor:F2F1isoF}
    The functors $F_2F_1$ and $F$ are naturally isomorphic. 
\end{cor}

\begin{proof} Taking  $X=K E(S)$ as the first argument of the bifunctors from 
   Proposition~\ref{prop:bifunctor iso}  we conclude that the functors
    \[
        K E(S) \otimes_{K S} (A \otimes_{A^e}-): (A\rtimes_\0 S)^e\textbf{-Mod}  \to K\textbf{-Mod}
    \]
    and
    \[
        (K E(S) \otimes_{K E(S)} (A\rtimes_\0 S)) \otimes_{(A\rtimes_\0 S)^e}-:(A\rtimes_\0 S)^e\textbf{-Mod}  \to K\textbf{-Mod}
    \] 
    are naturally isomorphic. On the one hand, $F_2F_1 = K E(S)\otimes_{K S} (A \otimes_{A^e}-),$ and on the other hand Lemma~\ref{l_AsemidirSisoBsemidirS} gives us an $A\rtimes_\0 S$-bimodule isomorphism  
    $$
    K E(S) \otimes_{K E(S)} (A\rtimes_\0 S) \cong A\rtimes_\0 S.
    $$ 
    The latter   implies the natural isomorphism  
    $$
    F \cong (K E(S) \otimes_{K E(S)} (A\rtimes_\0 S)) \otimes_{(A\rtimes_\0 S)^e}- 
    $$ 
    of functors, proving  our statement.
\end{proof}

  We also need the following auxiliary fact on flat modules, which is known, but we did not find a  bibliographic reference for it and we include a proof for reader's convenience. 

\begin{lem}\label{lem:FlatTransitivity} (Transitivity of the flatness)
 Let $\phi:R\to R'$ be a  homomorphism of unital rings such that $_RR'$ (resp. $R'_R$) is flat, and let $_{R'}M$ (resp. $M_{R'}$) be a  flat left (resp. right) module. Then $_RM$ (resp. $M_R$) is also flat.
\end{lem}
\begin{proof}
We will prove the statement for left modules. Consider an exact sequence of right $R$-modules 
$0\to K\to N$. Since $_RR'$ is flat then $0\to K \otimes_R R'\to N\otimes _R R'$ is an exact sequence in  the category of right $R'$-modules. The sequence   $0\to (K \otimes_R R')\otimes_{R'} M\to (N\otimes _R R')\otimes_{R'} M$ is also exact, because of the flatness of  $_{R'}M$. Then by the associativity of the tensor product and the canonical isomorphism $R'\otimes_{R'} M\cong M$, we have that $0\to K \otimes_R  M\to N\otimes _R M$ is exact.
\end{proof}


We shall also use the following fact.

\begin{lem}\label{lem:for0-homology} Let    $\0$ be a  unital action of an   inverse monoid $S$ on  an algebra $A$ over a commutative ring $K.$ Let  $M$ be an $A\rtimes_\0 S$-bimodule and denote by 
$[A, M]$  the $K$-submodule of $M$ generated by 
$$\{a x-x a \; \mid \; a\in A, x\in M\}.$$   Then $[A, M]$ is a left $KS$-submodule of $M$ and the
$K$-module mapping
 $$\psi: M/[A, M] \to   A \ot _{A^e} M, $$ given by
 $x+[A, M]   \mapsto  1_A  \ot  x,$ $x\in M,$ is a well-defined isomorphism of left $KS$-modules.
 \end{lem}
 \begin{proof} First note that it is readily seen, using \eqref{dentro} and \eqref{fora}, that the $K$-module homomorphism  $ M \to   A \ot _{A^e} M, $ determined by $ x \mapsto  1_A  \ot  x,$
  $x \in M,$ vanishes on $[A,M].$ Thus $\psi$ is a well-defined homomorphism of $K$-modules. Then keeping in mind \eqref{left right R^e-mod}, it is easily verified that the mapping  
  $$A \ot _{A^e} M \to M/[A, M], \;\; a\ot x \mapsto ax+
  [A, M],$$ is a well-defined inverse of $\psi. $
It follows from 
	 \eqref{left KS-mod A},  \eqref{left KS-mod M} and \eqref{A-bimod M} that

	 \begin{equation*}
		s\cdot (a \cdot x)= (s\cdot a) (s \cdot x) \;\;\; \text {and} \;\;\; 
		s\cdot (x \cdot a)= (s\cdot x) (s \cdot a) 
	 \end{equation*} for all $s \in S, a \in A$ and $x \in M.$ This implies that 
  $[A, M]$ is a left $K S$-submodule
	of $M,$ so that $ M/[A, M]$ is a left $K S$-module. Finally, for any $s\in S$ and $x \in M$ we see using \eqref{left KS-mod A},  \eqref{left KS-mod M} and \eqref{left KS-module A otimes M} that
	\begin{align*}  
& s\cdot \psi (x + [A,M]) = s\cdot (1_A \ot x)=
s\cdot 1_A \ot s\cdot x= 1_s \ot s\cdot x\\
&=1_A 1_s \ot (\overline{1_s \delta _s}) \cdot x \cdot (\overline{1_{s\m} \delta _{s\m}}) =
 1_A  \ot (\overline{1_s \delta _1} \overline{ 1_s \delta _s}) \cdot x \cdot (\overline{1_{s\m} \delta _{s\m}})
\\ &= 1_A  \ot (\overline{ 1_s \delta _s}) \cdot x \cdot (\overline{1_{s\m} \delta _{s\m}}) =  1_A \ot s\cdot x = \psi (s \cdot (x + [A,M])),
	\end{align*} showing that $\psi $ is a left $KS$-module isomorphism.
	\end{proof}
	

\subsection{Homology of $A\rtimes_\0 S $ over a field $K$}
To proceed with an arbitrary inverse monoid $S$  \underline{we  assume during this subsection} that $K$ is a field.

By \cite[Proposition 5.14]{DJ2} any commutative algebra over a field
 ge\-ne\-ra\-ted by idempotents is von Neumann regular. Thus, $K E(S)$ is von Neumann regular. Consequently, any module over $ K E(S)$ is flat (see  \cite[Corollary 1.13)]{Goodearl}). Therefore, we have the next:

\begin{lem}\label{lem:exactfunctor} The functor 
$- \otimes_{K E(S)} (A\rtimes_\0 S) $ is exact.
\end{lem}

Now we are ready to give the following result.

\begin{thrm} \label{teo:HLH}
   Let $\0$ be a unital action of an   inverse monoid $S$ on an algebra  $A$  over a field $K$ and $M$ an $A\rtimes_\0 S$-bimodule. Then   there exists a first quadrant homology spectral sequence
    \[
        E^2_{p,q} = H_p(S, (L_q F_1)M  ) \Rightarrow H_{p+q}(A\rtimes_\0 S, M). 
    \]
\end{thrm}
\begin{proof} We are going to use \cite[Theorem 10.48]{Rotman} for the functors
 \[
       (A\rtimes_\0 S)^e\textbf{-Mod}  
       \overset{F_1}{\longrightarrow} K S\textbf{-Mod}  \overset{F_2}{\longrightarrow} K\textbf{-Mod}.
    \]
 Note that the functors  $F_1$ and $F_2$ are right exact and that  
 \begin{align*}
 	L_\bullet F_2 (-) = \operatorname{Tor}^{K S}_\bullet (K E(S),-)\text{ and }L_{\bullet}F(-)=H_{\bullet}(A\rtimes_\0 S, -).
 \end{align*}
  Since by Corollary~\ref{cor:F2F1isoF} the functors $F_2 F_1$ and $ F$ are naturally isomorphic,  it  remains to show that for any projective object $P$ in $(A\rtimes_\0 S)^e$\textbf{-Mod} the left $K S$-module $F_1(P)$ is left $F_2$-acyclic. Equivalently,
    \[
        \operatorname{Tor}^{K S}_n(K E(S),F_1(P))=0, \, \forall n>0.
    \]
    So let  $P$ be a projective $(A\rtimes_\0 S)^e$-module. Then $-\otimes_{(A\rtimes_\0 S)^e} P$ is an exact functor. In addition, the functor  $- \otimes_{K E(S)} (A\rtimes_\0 S)$ is also exact thanks to Lemma~\ref{lem:exactfunctor}. Consequently, 
    $(- \otimes_{K E(S)} (A\rtimes_\0 S)) \otimes_{(A\rtimes_\0 S)^e}P$ is an exact functor. Taking $M=P$ as the second argument of the bifunctors from Proposition~\ref{prop:bifunctor iso},  we have that $- \otimes_{K S} (A \otimes_{A^e} P)$ is naturally isomorphic to $(- \otimes_{K E(S)} (A\rtimes_\0 S)) \otimes_{(A\rtimes_\0 S)^e}P.$ Then  
    \begin{align*}
        \operatorname{Tor}^{K S}_n(K E(S),F_1(P))
        &=L_n(- \otimes_{K S} (A \otimes_{A^e} P))(K E(S)) \\ 
        &\cong  L_n((- \otimes_{K E(S)} (A\rtimes_\0 S) \otimes_{(A\rtimes_\0 S)^e}P)(K E(S)) =0,
    \end{align*} for all $n>0,$ as desired. Thus we can apply \cite[Theorem 10.48]{Rotman}, which gives us the following first quadrant homology spectral sequence 
    \[
        E^2_{p,q} = \operatorname{Tor}^{K S}_p(K E(S), (L_q F_1)M  ) \Rightarrow H_{p+q}(A\rtimes_\0 S, M). 
    \] Finally, observe that by Definition~\ref{def:Inv Monoid (co)homol} we have that
    $$ \operatorname{Tor}^{K S}_p(K E(S), (L_q F_1)M  ) =
    H_p(S, (L_q F_1)M  ),
    $$  completing our proof.
    \end{proof}


\begin{cor}\label{cor:flatHHH}
	Let $\0$ be a unital action of an  inverse monoid $S$ on  a $K$-algebra $A$ over a field $K$ and $M$ an $A\rtimes_\0 S$-bimodule. Assume that 
	$(A\rtimes_\0 S)^e $ is flat as a left $A^e$-module. Then   there exists  a first quadrant homology spectral sequence
	\begin{align*}
	   E^2_{p,q} = H_p(S, H_q (A, M)) \Rightarrow H_{p+q}(A\rtimes_\0 S, M). 
   \end{align*}
	   \end{cor}
	   \begin{proof}  
	   By the transitivity of the flatness (Lemma~\ref{lem:FlatTransitivity}) every flat left 
	   $(A\rtimes_\0 S)^e $-module is flat as a left $A^e$-module. Consequently, any flat resolution of $M$ in  $(A\rtimes_\0 S)^e\textbf{-Mod}$ is a flat resolution of $M$ in  $A^e\textbf{-Mod}.$    Since
	   $\operatorname{Tor}^{A^e}_q (A,M)$ can be computed using flat resolutions of $M$ in  $A^e\textbf{-Mod}$  (see, for example,  \cite[Theorem 7.5]{Rotman}), it follows that  
   $(L_q F_1)M =   \operatorname{Tor}^{A^e}_q (A,M) = H_q (A,M)$  for all $q,$ and our statement is obtained by applying
   Theorem~\ref{teo:HLH}. 
		\end{proof}

   Observe that it is well-known that if  $N$ is an $A$-bimodule, then   the  groups 
$H_n(A,N) := \operatorname{Tor}_n^{A^e} (A,N),$ $n\geq 0,$  can be identified with the homology groups of the Hochschild complex  (see, for example,  \cite[Section 1.1]{Loday}). In particular, there are isomorphisms between the $K$-spaces
$H_0(A,N)$ and    $N/[A, N].$   In our case, under the assumptions of Corollary~\ref{cor:flatHHH},  the identification between   
$H_0(A,M)$  and    $M/[A, M]$  turns out to be an isomorphism 
of  $KS$-modules:  

\begin{lem}\label{lem:0-homology} Let $\0$ be a unital action of an  inverse monoid $S$ on  a $K$-algebra $A$ over a field $K$ and $M$ an $A\rtimes_\0 S$-bimodule. Assume that 
	$(A\rtimes_\0 S)^e $ is flat as a left $A^e$-module. Then there is a  left $KS$-module isomorphism
	$$H_0(A,M)  \cong M/[A, M].  $$
\end{lem}
\begin{proof} Thanks to Lemma~\ref{lem:for0-homology} it is enough to establish a left $KS$-module isomorphism between $H_0(A,M) $ and $ A \ot _{A^e} M.$   An identification of these $K$-spaces is well-known and we only need to observe that it is a $KS$-isomorphism. 

Let 
\begin{equation}\label{eq:flatResolution}
 \ldots \to P_1 \to P_0 \to M\to 0
 \end{equation}  be a flat resolution of $M$ in  $(A\rtimes_\0 S)^e\textbf{-Mod}.$ Applying to the exact sequence $ P_1 \to P_0 \to M\to 0$ the right exact functor $F_1,$ we obtain the exact sequence of left $KS$-modules and  $KS$-homomorphisms  
$$ A \ot _{A^e}  P_1 \to A \ot _{A^e} P_0 \to A \ot _{A^e} M\to 0.$$ As in the proof of Corollary~\ref{cor:flatHHH},   the complex \eqref{eq:flatResolution} is also  a flat resolution of $M$ in $A^e\textbf{-Mod}$ and, consequently, 
 $$H_0(A,M) = 
( A \ot _{A^e} P_0) / {\rm im}\, (A \ot _{A^e}  P_1 \to A \ot _{A^e} P_0).$$  This yields  the desired isomorphism between the left $KS$-modules $H_0(A,M)$ and  $ A \ot _{A^e} M.$ 
\end{proof}

\begin{cor}\label{rem:separable}
    Let $\0$ be a unital action of an  inverse monoid $S$ on  a $K$-algebra $A$ over a field $K$ and $M$ an $A\rtimes_\0 S$-bimodule. Assume that $A$ is separable over $K.$ Then   there is  an isomorphism
    \[
         H_n(S, M/[A, M]) \cong H_{n}(A\rtimes_\0 S, M). 
    \] 
\end{cor}
\begin{proof} 
 Since $A$ is separable over the field $K,$ the $K$-algebra $A^e$  is semisimple, so that any (left or right)  $A^e$-module is semisimple and hence flat.  In particular, so is the left   $A^e$-module $(A\rtimes_\0 S)^e ,$ and the conditions of \cref{cor:flatHHH} are satisfied. Then, since  $A$ is flat  as a right $A^e$-module,  the spectral sequence collapses on the $p$-axis and by 
	\cite[Proposition 10.21]{Rotman} we obtain an isomorphism
	$$ H_n(S, H_0(A,M)) \cong H_{n}(A\rtimes_\0 S, M).  $$
	Finally,  thanks to Lemma~\ref{lem:0-homology},
	the left $KS$-modules   $H_0(A,M)$ and $  M/[A, M]$ are isomorphic. 
\end{proof}



\subsection{Homology of $A\rtimes_\0 S $ with $E$-unitary $S$}
 In the $E$-unitary case we are able to drop the restriction on $K$ to be a field, replacing it by the assumption that $A$ is flat over the commutative ring $K.$  Moreover, we shall see below   that if $A$ is flat over $K,$ then  $(A\rtimes_\0 S)^e $ is flat as a left (right) $A^e$-module for any $E$-unitary inverse semigroup $S.$  

\begin{lem}\label{lem:acyclic}
	Let $\0$ be a unital action of an  $E$-unitary inverse monoid $S$ on a $K$-algebra  $A$  over a commutative ring $K,$ 
	 $M$ be a left $K S$-module, and let $P_\bullet \to K E(S)$ be a projective resolution  of $K E(S)$ in $\rMod KS.$  Then, $H_n(P_\bullet \otimes_{K E(S)} M) = 0$ for all $n \geq 1$.
\end{lem}

\begin{proof} 
   By Corollary~\ref{cor:KS pojective over KR(S)},   $K S$ is projective  as a right $K E(S)$-module.  Then by the transitivity of projectivity (see \cite[Proposition 1.4]{DeMeyerIngraham}) any projective right $KS$-module is  projective as a right $K E(S)$-module.  Consequently,
	  $P_{\bullet} \to K E(S)$ is also a projective resolution of $K E(S)$ in \textbf{Mod-}$K E(S)$. Therefore,
	\[
		H_n(P_{\bullet}\otimes_{K E(S)} M)= \operatorname{Tor}_n^{K E(S)}(K E(S), M)= \left\{\begin{matrix}
			0 & \text{ if } n \geq 1 \\ 
			M & \text{ if } n=0.
			\end{matrix}\right.
	\]
\end{proof}

  \begin{prop} \label{p_F1sendPtoF2acylic} Let $\0$ be a unital action of an  $E$-unitary inverse monoid $S$ on a $K$-algebra  $A$  over a commutative ring $K.$ Then
     $F_1$ sends projective $A\rtimes_\0 S$-bimodules to left $F_2$-acyclic modules.
 \end{prop}
 \begin{proof} Let $P$ be an arbitrary projective 
  $(A\rtimes_\0 S)^e$-module. We need to show that 
     \[
         (L_nF_2)(A \otimes_{A^e}P) = 0, \, \, \forall n \geq 1.
     \] Using the fact that $\operatorname{Tor}(-,-)$ is symmetric, we see that
     \begin{align*}
         (L_nF_2)(A \otimes_{A^e}P)
         &= L_n(K E(S) \otimes_{K S} -)(A \otimes_{A^e}P) \\
         &\cong\operatorname{Tor}^{K S}_n (K E(S),A \otimes_{A^e}P)\\
         &\cong L_n(- \otimes_{K S} (A \otimes_{A^e}P))(K E(S)) \\
        (\text{by Proposition \ref{prop:bifunctor iso}}) &\cong L_n((- \otimes_{K E(S)} (A\rtimes_\0 S)) \otimes_{(A\rtimes_\0 S)^e}P)(K E(S)).
     \end{align*}
     Let $Q_{\bullet} \to K E(S)$ be a projective resolution of $K E(S)$ in \textbf{Mod-}$K S$. Then,
     \begin{align*}
         (L_nF_2)(A \otimes_{A^e}P) &\cong L_n((- \otimes_{K E(S)} (A\rtimes_\0 S)) \otimes_{(A\rtimes_\0 S)^e}P)(K E(S))\\ &= H_n \left(  (Q_{\bullet}\otimes_{K E(S)} (A\rtimes_\0 S)) \otimes_{(A\rtimes_\0 S)^e}P \right).
     \end{align*}
    By Lemma~\ref{lem:acyclic} we have that $ H_n(Q_{\bullet}\otimes_{K E(S)}( A\rtimes_\0 S))= 0$ for all $n>0.$ 
      This yields that the complex $Q_{\bullet}\otimes_{K E(S)}  (A\rtimes_\0 S)$ is exact at the $n$-th term for all $n \geq 1$. It follows that 
      the complex $(Q_{\bullet}\otimes_{K E(S)} (A\rtimes_\0 S)) \otimes_{(A\rtimes_\0 S)^e}P$ is also exact at the $n$-th term for all $n \geq 1$ since $P$ is projective as an $(A\rtimes_\0 S)^e$-module, and thus
     \[
         H_n \left(  (Q_{\bullet}\otimes_{K E(S)}(A\rtimes_\0 S)) \otimes_{(A\rtimes_\0 S)^e}P \right) = 0,\ \forall n \geq 1,
     \]
     completing our proof. 
 \end{proof}

 In view of  Remark~\ref{rem:Adelta_1Embedding} we may consider $A$ as a $K$-subalgebra of $A\rtimes_\0 S$ by means of the embedding $A\to  A\rtimes_\0 S,$ given by $a \mapsto a \dl _1.$ Then  $(A\rtimes_\0 S)^e$ is a left $A^e$-module with the action given in \eqref{eq:left R^e on (R')^e} and it is also a right
  $A^e$-module with the action defined in \eqref{eq:right R^e on (R')^e}. Similarly, 
 $(A\rtimes_{\tl\0} {\mathcal G}(S))^e$ is also a left and right  $A^e$-module via \eqref{eq:left R^e on (R')^e} and \eqref{eq:right R^e on (R')^e}, respectively.

\begin{lem}\label{lem:FlatnessInE-unitaryCase}
  Suppose that the $K$-algebra $A$ is flat over the commutative ring $K$ and $\0$ is a  unital action of an  $E$-unitary inverse monoid $S$ on  $A.$ Then  $(A\rtimes_\0 S)^e $ is flat as a left (right) $A^e$-module.
\end{lem}
\begin{proof}  We know  from Proposition~\ref{cor:IsoPhi} and Remark~\ref {rem:A-bimodule mapping} that
      $A\rtimes_\0 S $ and $  A\rtimes_{\tl\0} {\mathcal G}(S)$ are isomorphic not only as  $K$-algebras, but also as $A$-bimodules. By Remark~\ref{rem:left R^e on (R')^e} this implies that  $(A\rtimes_\0 S)^e $ and 
      $(A\rtimes_{\tl\0} {\mathcal G}(S))^e$ are isomorphic as left (right)  $A^e$-modules. Thus, it is enough to prove that $(A\rtimes_{\tl\0} {\mathcal G}(S))^e$ is flat as a left (right) $A^e$-module.
   
   We shall show first that  $\D _g$ is flat as a right (left) $A$-module  for any $g \in {\mathcal G}(S).$
   Indeed,  let $g\in \cG(S)$ and $(\cF(g),\sst)$ be the directed set of all finite subsets of the $\sg$-class $g$. For any $T\in \cF(g)$ consider the finite sum $\cI(T):= \sum_{t\in T}1_tA$ of unital ideals. Then $\cI(T)$ is also a unital ideal in $A$ by \cref{R=sum-orth-ideals}.  It follows that $\cI(T)$ is of the form $Ae$, where $e \in \D_ g$ is an idempotent, which is central in $A.$ Then $\cI(T)$ is a direct summand of $A$, and, consequently, it is projective as a left (right)  $A$-module. Therefore, $\cI(T)$ is flat  as a left (right)  $A$-module.  Consider the direct system $\{\cI(T),\vf^T_U\}$ of left (right) $A$-modules, where $\vf^T_U:\cI(T)\to\cI(U)$ is the inclusion map for all $T\sst U$. Observe by \cite[Corollary 5.31]{Rotman} that
     	\begin{align*}
     		\D_g=\sum_{s\in g}1_sA=\bigcup_{T\in\cF(g)}\cI(T)
     	\end{align*}
     	is the direct limit of $\{\cI(T),\vf^T_U\}$. Thus, by \cite[Corollary 5.34]{Rotman},  $\D_ g$  is  flat  as a left (right) $A$-module.
     	
     	Take $t,s \in S.$ With a slight abuse of notation we denote by $ (1_sA)^{\rm op}$ the subset $1_sA $ of  $ A^{\rm op}.$  Since $1_tA$  is a direct summand of $A,$  $ (1_sA)^{\rm op}$ is a direct summand of $ A^{\rm op}$ and tensor product respects direct sums, it follows that  
     	$1_tA \otimes _K  (1_sA)^{\rm op}$ is a direct summand of  $A^e = A \ot _K A^{\rm op}.$  Hence $1_tA \otimes _K  (1_sA)^{\rm op}$ is a flat  left (right) $A^e$-module, being projective over $A^e.$ 
     
   We know already that $\D_g$ is flat as a right $A$-module for every $g \in \cG(S).$ Using  the fact that $A$ is flat over $K$ and the transitivity of flat modules (Lemma~\ref{lem:FlatTransitivity}), we conclude that $\D_g$ is flat over $K.$   	Then for any $h\in  \cG(S)$ the inclusion mapping 
     	$ (\D _h)^{\rm op} \to A^{\rm op}$ gives rise to the exact sequence 
     	\begin{equation}\label{eq:mono1}
0 \to \D _g \otimes _K  (\D _h)^{\rm op} \to \D _g \ot _K A^{\rm op}. 
\end{equation} Here, similarly as above, $ (\D _h)^{\rm op}$ stands for the subset  $\D _h$ of  $A^{\rm op}.$ Obviously, the mapping $\D_g \ot _K (\D _h)^{\rm op} \to \D _g \ot _K A^{\rm op} $ is a monomorphism of non-necessarily unital algebras,  
      which allows us to identify the algebra $\D_g \ot _K (\D _h)^{\rm op} $ with its image in $\D _g \ot _K A^{\rm op}.$
      Now, since $A$ is flat over the commutative ring $K, $ then $A^{\rm op}$ is also flat over $K,$ and from the inclusion map $\D _g \to A$ we obtain the exact sequence of algebras
\begin{equation}\label{eq:mono2}
0 \to \D _g \ot _K A^{\rm op} \to A \ot _K A^{\rm op} = A^e.
\end{equation} Combining \eqref{eq:mono1} and \eqref{eq:mono2} we obtain a monomorphism of algebras  (considered as non-necessarily unital algebras)
$$0 \to \D _g \otimes _K  (\D _h)^{\rm op} \to A \ot _K A^{\rm op}, $$ and identify $\D _g \otimes _K  (\D_h)^{\rm op}$ with its image in $A\ot _K  A^{\rm op}, $  so that   $\D_g \otimes _K  (\D _h)^{\rm op}$ becomes an ideal in $A^e.$ Moreover, we may write
 $$\D_g \otimes _K  (\D _h)^{\rm op}=  \sum_{t\in g, s\in h}1_tA \otimes _K (1_s A)^{\rm op}.$$

     Let $g, h\in \cG(S)$ and $\cF(g,h)$ be the  set of all subsets $T$ of the Cartesian product $g \times h$ of the form  $T=T_g \times T_h,$ where $T_g$ is a  finite subset of the  $\sg$-class $g$ and $T_h$ is a  finite subset of the  $\sg$-class $h.$ For 
     $T=T_g \times T_h \in \cF(g,h)$ and $U= U_g \times  U_h \in \cF(g,h)$ it is clear that  
     $T \sst U$ if and only if 
      $T_g \sst U_g$ and  $T_h \sst U_h.$ Evidently, $(\cF(g,h),\sst)$ is a directed partially ordered set.

     Given any  $T=T_g \times T_h\in \cF(g,h)$, consider the finite sum 
     $$\cI(T):= \sum_{t\in T_g, s\in T_h}1_tA \otimes _K (1_s A)^{\rm op}$$ of unital ideals of $A^e.$ Then using again  \cref{R=sum-orth-ideals} we see that $\cI(T)$ is also a unital ideal in $A^e .$  Hence, 
      $\cI(T) = A^e a^e$, where $a^e \in A^e$ is an idempotent, which is central in $A^e.$ Thus $\cI(T)$ is a direct summand of $A^e$, and, consequently, it is  flat as a left (right)  $A^e$-module, being  projective over $A^e.$
        Now we look at the direct system 
       $\{\cI(T),\psi^{T}_{U}\}$ of left (right) $A^e$-modules, where $\psi^T_U:\cI(T)\to\cI(U)$ is the inclusion map 
        for all $T\sst U$. Referring again to \cite[Corollary 5.31]{Rotman} we have that
     	\begin{align*}
     		\D_g \otimes _K (\D _h)^{\rm op}=\sum_{t\in g, s \in h}1_tA \otimes _K (1_sA)^{\rm op} =\bigcup_{T\in\cF(g,h)}\cI(T)
     	\end{align*}
     	is the direct limit of $\{\cI(T),\psi^T_U\}$. Then, $ \D_g \otimes _K (\D _h)^{\rm op} \subseteq A \otimes _K A^{\rm op}$  is  flat  as a left (right)  $A^e$-module, being a direct limit of flat  left (right)  $A^e$-modules   (see \cite[Corollary 5.34]{Rotman}).  
     	
     Given $g, h \in \cG(S),$ denote by $(\D_h \dl_h)^{\rm op}$ the subset $\D_h\dl_h$ of $(A\rtimes_{\tl\0} {\mathcal G}(S))^{\rm op}$ and consider the mapping 	
     $$
     \psi :  \D_g \otimes _K (\D _{h\m})^{\rm op} \to \D _g \delta _g \otimes _ K (\D _h \delta _h)^{\rm op},
     $$ 
     given by 
     $$a \ot b \mapsto a \delta _g \ot {\tl\0}_h (b) \dl_h, $$
     where $a \in  \D_g\text{ and }b \in \D _{h\m}.$ Then $\psi $ is an isomorphism of left $A^e$-modules. Indeed, clearly, it is a $K$-isomorphism. Moreover, for $a \in  \D_g, b \in \D _{h\m}$ and $c \ot d \in A^e$ we have that 
      \begin{align*}
      	\psi ((c \ot d)\cdot (a \ot b)) 
      	&= \psi (ca \ot bd) 
      	= ca \dl _g \ot  \tl\0_h (bd) \dl_h =  ca \dl _g \ot \tl\0_h(\tl\0\m_h(\tl\0_h(b))d)\dl_h\\
      	 &=c\dl_1 \cdot a \dl _g \ot \tl\0_h (b) \dl _h \cdot d \dl_1  = 	(c \ot d)\cdot \psi (a \ot b),
      \end{align*} showing that $\psi $  is a mapping of left $A^e$-modules. 
      Similarly, for $g, h \in \cG(S),$ define 	
      $$
      \eta :  \D_{g\m} \otimes _K (\D_h)^{\rm op} \to \D_g \delta_g \otimes_K (\D_h \delta _h)^{\rm op}
      $$ 
      by 
      $$
      a \ot b \mapsto {\tl\0}_g(a) \delta_g \ot b \dl_h, 
      $$
      where $a \in  \D_{g\m}$ and $b \in \D _h.$ Then $\eta$ is an isomorphism of right $A^e$-modules. For, $\eta$ is obviously a $K$-isomorphism and for all $a \in \D_{g\m}, b \in \D _h$ and $c \ot d \in A^e$ one has
      \begin{align*}
      	\eta((a \ot b) \cdot (c \ot d)) 
      	      	&= \eta (ac \ot db)={\tl\0}_g(ac) \delta_g \ot db \dl_h=\tl\0_g(\tl\0\m_g(\tl\0_g(a))c) \delta_g \ot db \dl_h\\  
      	      	&=\tl\0_g(a)\dl _g \cdot c\dl_1 \ot d \dl_1 \cdot b\dl_h= \eta(a \ot b)\cdot (c \ot d).
      \end{align*} 
       This yields that 
      $\D _g \delta _g \otimes _ K (\D _h \delta _h)^{\rm op}$ is also flat as a left (right) $A^e$-module.
      
      Now, since  
      $  A\rtimes_{\tl\0} {\mathcal G}(S) = \bigoplus_{g \in {\mathcal G}(S)} \D _g \dl _g$ and $\ot _K$ respects direct sums,
      we obtain the direct sum  
 $$( A\rtimes_{\tl\0} {\mathcal G}(S))^{e} =
 \bigoplus _{g,h \in {\mathcal G}(S)} \D _g \delta _g \otimes _ K (\D _h \delta _h)^{\rm op} $$ of $K$-modules. Evidently, it is a direct sum of left (right) $A^e$-modules, which implies that  $( A\rtimes_{\tl\0} {\mathcal G}(S))^{e}$ is flat as a left (right) $A^e$-module. 
 \end{proof}

\begin{thrm}\label{teo:E-UnitaryHomolSpectralSeq}
   Suppose that the $K$-algebra $A$ is flat over the commutative ring $K.$ Let $\0$ be a  unital action of an  $E$-unitary inverse monoid $S$ on  $A$ and $M$ an $A\rtimes_\0 S$-bimodule. Then   there exists  a first quadrant homology spectral sequence
    \begin{equation}\label{eq:E-UnitaryHomolSpectralSeq}
        E^2_{p,q} = H_p(S, H_q (A, M)) \Rightarrow H_{p+q}(A\rtimes_\0 S, M). 
    \end{equation}
    \end{thrm}
    \begin{proof} 
    Since  the functor  $F_2$ is right exact and   by Proposition~\ref{p_F1sendPtoF2acylic} the functor 
     $F_1$ sends projective $A\rtimes_\0 S$-bimodules to left $F_2$-acyclic modules,  we are in conditions to apply \cite[Theorem 10.48]{Rotman} to the functors
 \[
       (A\rtimes_\0 S)^e\textbf{-Mod}  
       \overset{F_1}{\longrightarrow} K S\textbf{-Mod}  \overset{F_2}{\longrightarrow} K\textbf{-Mod},
    \]  
    which gives us  a first quadrant Grothendieck spectral sequence
    \begin{equation}\label{eq:SpectralSeqFromRotman}
        E^2_{p,q} = (L_p F_2) (L_q F_1) M \Rightarrow L_{p+q} ( F_2 F_1 ) M. 
    \end{equation} Thanks to the natural isomorphism
   of the functors $F_2 F_1$ and $ F$  provided  by Corollary~\ref{cor:F2F1isoF} and keeping in mind that
    \begin{align*}
 	L_\bullet F_2 (-) = \operatorname{Tor}^{K S}_\bullet (K E(S),-)\text{ and }L_{\bullet}F(-)=H_{\bullet}(A\rtimes_\0 S, -),
 \end{align*} we  obtain from \eqref {eq:SpectralSeqFromRotman} the  first quadrant  spectral sequence 
    \[
        E^2_{p,q} = \operatorname{Tor}^{K S}_p(K E(S), (L_q F_1)M  ) \Rightarrow H_{p+q}(A\rtimes_\0 S, M). 
    \] Since by definition of homology of $S$ (see Definition~\ref{def:Inv Monoid (co)homol}), we have that 
    $$ \operatorname{Tor}^{K S}_p(K E(S), (L_q F_1)M  ) =
    H_p(S, (L_q F_1)M  ),
    $$  the spectral sequence takes the form
    \begin{align}\label{eq:IntermediateSpectralSeq}
        E^2_{p,q} = H_p(S, (L_q F_1)M  ) \Rightarrow H_{p+q}(A\rtimes_\0 S, M). 
    \end{align}
         
            By  Lemma~\ref{lem:FlatnessInE-unitaryCase} we have that $(A\rtimes_\0 S)^e $ is flat as a left $A^e$-module. Then by the transitivity of the flatness (Lemma~\ref{lem:FlatTransitivity}) every flat left 
	   $(A\rtimes_\0 S)^e $-module is flat as a left $A^e$-module. Consequently, any flat resolution of $M$ in  $(A\rtimes_\0 S)^e\textbf{-Mod}$ is a flat resolution of $M$ in  $A^e\textbf{-Mod}.$    As it was mentioned already,
	   $\operatorname{Tor}^{A^e}_q (A,M)$ can be computed using flat resolutions of $M$ in  $A^e\textbf{-Mod}.$ It follows that  
   $(L_q F_1)M =   \operatorname{Tor}^{A^e}_q (A,M) = H_q (A,M)$  for all $q,$ and the spectral sequence \eqref{eq:IntermediateSpectralSeq} takes the form 
    \eqref{eq:E-UnitaryHomolSpectralSeq}, completing our proof. 
   \end{proof}

        Since an algebra over a field $K$ is obviously flat over $K,$ we immediately obtain the next:
    
   \begin{cor}\label{cor:E-unitaryOverField}
   Let $\0$ be a unital action of an  $E$-unitary inverse monoid $S$ on an algebra $A$ over a field $K$ and $M$ be an $A\rtimes_\0 S$-bimodule. Then   there exists  a first quadrant homology spectral sequence of the form 
   \eqref{eq:E-UnitaryHomolSpectralSeq}.
   \end{cor}

    \section{Cohomology of $A\rtimes_\0 S$}\label{sec-cohom}

\subsection{Preliminary results} As in the case of homology, we need some preparation. Let, as above,  $M$ be an $A \rtimes_\0 S$-bimodule.

\begin{lem}\label{Gamma-hom-S-to-End(Hom)}
	There is a homomorphism $\Gamma : S \to \End_K (\Hom_{A^e}(A, M))$, $s\mapsto\Gamma _s$,
	given by  
	\begin{align*}
		\Gamma _s(f)(a) = s \cdot f(s\m \cdot a),
	\end{align*}
	for all $f \in \Hom_{A^e} (A, M)$ and $a\in A$, where $A$ is considered as a left $KS$-module via \cref{left KS-mod A} and $M$ is a left $KS$-module via \cref{left KS-mod M}.
\end{lem}
\begin{proof}
	Let us first show that $\Gamma _s(f) \in \Hom_{A^e} (A, M)$. For all $a,b,c\in A$ we have
	\begin{align*}
		\Gamma _s (f) (bac) &= s \cdot f(s\m \cdot (bac))
		= s \cdot f (s\m  b \cdot s\m  a \cdot s\m  c)
		= s \cdot (s\m  b \cdot f(s\m  a) \cdot s\m  c)\\
		&\overset{\cref{left KS-mod M}}{=}\ol{1_s\dl_s}\cdot (s\m  b \cdot f(s\m  a) \cdot s\m  c)\cdot \ol{1_{s\m} \dl_{s\m}}\\
		&\overset{\cref{left KS-mod A}}{=}\ol{1_s\dl_s}\cdot (\0_{s\m} (1_sb) \cdot f(s\m  a) \cdot \0_{s\m} (1_sc))\cdot \ol{1_{s\m} \dl_{s\m}}\\
		&\overset{\cref{A-bimod M}}{=}\ol{1_s\dl_s}\cdot \ol{\0_{s\m}(1_sb)\dl_1} \cdot f(s\m  a) \cdot \ol{\0_{s\m} (1_sc)\dl_1}\cdot \ol{1_{s\m} \dl_{s\m}}\\
		&=\ol{\0_s\circ\0_{s\m}(1_sb)\dl_s} \cdot f(s\m  a) \cdot \ol{\0_{s\m} (1_sc)\dl_{s\m}}\\
		&=\ol{1_sb\dl_s} \cdot f(s\m  a) \cdot \ol{\0_{s\m} (1_sc)\dl_{s\m}}=\ol{b\dl_1}\cdot\left(\ol{1_s\dl_s} \cdot f(s\m  a) \cdot \ol{1_{s\m}\dl_{s\m}}\right)\cdot\ol{c\dl_1}\\
		&\overset{\cref{left KS-mod M}}{=}\ol{b\dl_1}\cdot\left(s \cdot f(s\m  a) \right)\cdot\ol{c\dl_1}\overset{\cref{A-bimod M}}{=}b\cdot\left(s \cdot f(s\m  a) \right)\cdot c=b\cdot\Gamma _s(a)\cdot c.
	\end{align*}
	Now we prove that $\Gamma _s\circ\Gamma _t = \Gamma _{st}$. For all $f \in \Hom_{A^e} (A, M)$ and $a\in A$ we have
	\begin{align*}
		\Gamma _s(\Gamma _t(f))(a) = s \cdot \Gamma _t(f)(s\m\cdot  a) &= s \cdot t \cdot f(t\m \cdot s\m \cdot a)= (st) \cdot f( (t\m s\m)\cdot  a)\\
		&=(st) \cdot f( (st)\m\cdot  a)= \Gamma _{st}(f)(a).
	\end{align*}	
\end{proof}


Thus, $\Gamma $ endows $\Hom_{A^e}(A, M)$ with a left $KS$-module structure via the action
\begin{align}\label{(sf)(a)=sf(s-inv-a)}
	(s\la f)(a)=s\la f(s\m\la a).
\end{align}


Hence, we can define the following left exact functors:
\begin{align*}
	T_1&:=\Hom_{A^e}(A,-): (A \rtimes_\0 S)^e\Mod\to KS\Mod,\\
	T_2&:=\Hom_{KS}(KE(S),-):KS\Mod\to K\Mod.
\end{align*}
We shall also need  the functor which determines the Hochschild cohomology of $A \rtimes_\0 S$ with coefficients in an $A \rtimes_\0 S$-bimodule:
\begin{align*}
	T:=\Hom_{(A \rtimes_\0 S)^e}(A \rtimes_\0 S,-) : (A \rtimes_\0 S)^e\Mod\to K\Mod.
\end{align*}

Let $X$ be a left $KS$-module. Then $X$ can be considered as a right $KE(S)$-module via
\begin{align}\label{x-cdot-b=b-cdot-x}
	x\cdot b:=b\cdot x,
\end{align}	
for all $x\in X$ and $b\in KE(S)$, since $KE(S)$ is commutative.
\begin{lem}
	Let $X$ be a left $KS$-module. Then $X\ot_{KE(S)}(A \rtimes_\0 S)$ is an $A \rtimes_\0 S$-bimodule with the actions \cref{right crossprod-action} and
	\begin{align}\label{a.dl_s(x-ot-b.dl_t)}
		\ol{a\dl_s}\la(x\ot_{KE(S)}\ol{b\dl_t})=s\la x\ot_{KE(S)}\ol{a\dl_s}\cdot\ol{b\dl_t}.
	\end{align}
\end{lem}
\begin{proof}
	The proof is similar to the proof of \cref{prop:bimodule structure X tensor cross}.
\end{proof}

\begin{lem}\label{construction-of-gm}
	There is a $K$-linear map
	\begin{align*}
		\gm:\Hom_{KS}(X,\Hom_{A^e}(A,M))\to\Hom_{(A \rtimes_\0 S)^e}(X\ot_{KE(S)} (A \rtimes_\0 S) ,M)
	\end{align*}
	such that, for all $f\in\Hom_{KS}(X,\Hom_{A^e}(A,M))$,
	\begin{align}\label{gm(f)(x-ot-a.dl_s)=f_s(1_A)a.dl_s}
		\gm(f)(x\ot_{KE(S)}\ol{a\dl_s})=f_x(1_A)\cdot \ol{a\dl_s},
	\end{align}
	where $f_x:=f(x)\in\Hom_{A^e}(A,M)$.
\end{lem}
\begin{proof}
	Let us see that, for a fixed 
	 $f\in\Hom_{KS}(X,\Hom_{A^e}(A,M))$, the map 
	\begin{align*}
		\gm^f: X\times (A \rtimes_\0 S) \to M,\ (x,\ol{a\dl_s})\mapsto f_x(1_A)\cdot \ol{a\dl_s},
	\end{align*}
	is balanced over $KE(S)$. Note that $f_{s\la x}=s\la f_x$ for all $s\in S$ and $x\in X$, since $f$ is a morphism of left $KS$-modules. Let $e\in E(S)$. Then
	\begin{align*}
		f_{x\cdot e}(1_A)&\overset{\cref{x-cdot-b=b-cdot-x}}{=}f_{e\cdot x}(1_A)=(e\cdot f_x)(1_A)\overset{\cref{(sf)(a)=sf(s-inv-a)}}{=}e\cdot f_x(e\cdot 1_A)\overset{\cref{EmmanuelLemma(i)}}{=}e\cdot f_x(1_e)
		\overset{\cref{EmmanuelLemma(ii)}}{=}1_ef_x(1_e)1_e\\
		&=f_x(1_e1_e1_e)=f_x(1_A1_e)=f_x(1_A)1_e\overset{\cref{A-bimod M}}{=}f_x(1_A)\cdot\ol{1_e\dl_1}.
	\end{align*}
	It follows that
	\begin{align*}
		\gm^f(x\cdot e,\ol{a\dl_s})&=f_{x\cdot e}(1_A)\cdot \ol{a\dl_s}=f_x(1_A)\cdot\ol{1_e\dl_1}\cdot \ol{a\dl_s}=f_x(1_A)\cdot\ol{1_ea\dl_s}\\
		& \overset{\cref{left KE(S)-action on cross}}{=}f_x(1_A)\cdot(e\la\ol{a\dl_s})=\gm^f(x,e\la\ol{a\dl_s}).
	\end{align*}
	Thus, $\gm^f$ is balanced over $KE(S)$, so that $\gm(f)$ given by \cref{gm(f)(x-ot-a.dl_s)=f_s(1_A)a.dl_s} is well-defined.
	
	Let us show that  $\gm(f)$ is an $A \rtimes_\0 S$-bimodule mapping. 
	We have
	\begin{align*}
		\gm(f)((x\ot_{KE(S)}\ol{a\dl_s})\cdot\ol{b\dl_t})&\overset{\cref{right crossprod-action}}{=}\gm(f)(x\ot_{KE(S)}(\ol{a\dl_s}\cdot\ol{b\dl_t}))\\
		&\overset{\cref{gm(f)(x-ot-a.dl_s)=f_s(1_A)a.dl_s}}{=}f_x(1_A)\cdot(\ol{a\dl_s}\cdot\ol{b\dl_t})=(f_x(1_A)\cdot \ol{a\dl_s})\cdot\ol{b\dl_t}\\
		&\overset{\cref{gm(f)(x-ot-a.dl_s)=f_s(1_A)a.dl_s}}{=}\gm(f)(x\ot_{KE(S)}\ol{a\dl_s})\cdot\ol{b\dl_t}.
	\end{align*}
	Furthermore, 
	\begin{align*}
		\gm(f)(\ol{b\dl_t}\la(x\ot_{KE(S)}\ol{a\dl_s}))&\overset{\cref{a.dl_s(x-ot-b.dl_t)}}{=}\gm(f)(t\cdot x\ot_{KE(S)}(\ol{b\dl_t}\cdot\ol{a\dl_s}))\\
		&\overset{\cref{gm(f)(x-ot-a.dl_s)=f_s(1_A)a.dl_s}}{=}f_{t\la x}(1_A)\cdot (\ol{b\dl_t}\cdot\ol{a\dl_s}).
	\end{align*}
	Since
	\begin{align*}
		f_{t\la x}(1_A)&=(t\la f_x)(1_A)\overset{\cref{(sf)(a)=sf(s-inv-a)}}{=}t\la f_x(t\m \la 1_A)\overset{\cref{left KS-mod A}}{=}t\la f_x(\0_{t\m}(1_t1_A))\\
		&\overset{\cref{left KS-mod M}}{=}\ol{1_t\dl_t}\cdot f_x(1_{t\m})\cdot\ol{1_{t\m}\dl_{t\m}},
	\end{align*}
	it follows that
	\begin{align*}
		\gm(f)(\ol{b\dl_t}\la(x\ot_{KE(S)}\ol{a\dl_s}))&=\ol{1_t\dl_t}\cdot f_x(1_{t\m})\cdot\ol{1_{t\m}\dl_{t\m}}\cdot (\ol{b\dl_t}\cdot\ol{a\dl_s})\\
		&=\ol{1_t\dl_t}\cdot f_x(1_{t\m}1_A)\cdot(\ol{1_{t\m}\dl_{t\m}}\cdot \ol{b\dl_t})\cdot\ol{a\dl_s}.
	\end{align*}
	Now,
	\begin{align*}
		\ol{1_t\dl_t}\cdot f_x(1_{t\m}1_A)\cdot(\ol{1_{t\m}\dl_{t\m}}\cdot \ol{b\dl_t})
		&=\ol{1_t\dl_t}\cdot 1_{t\m}\cdot f_x(1_A)\cdot \ol{\0_{t\m}(b)\dl_{t\m t}}\\
		&\overset{\cref{A-bimod M}}{=}(\ol{1_t\dl_t}\cdot \ol{1_{t\m}\dl_1})\cdot f_x(1_A)\cdot \ol{\0_{t\m}(b)\dl_1}\\
		&\overset{\cref{A-bimod M}}{=} \ol{\0_t(1_{t\m})\dl_t}\cdot (f_x(1_A)\cdot \0_{t\m}(b))\\
		&=\ol{1_t\dl_t}\cdot f_x(\0_{t\m}(b))=\ol{1_t\dl_t}\cdot \0_{t\m}(b)\cdot f_x(1_A)\\
		&\overset{\cref{A-bimod M}}{=}(\ol{1_t\dl_t}\cdot \ol{\0_{t\m}(b)\dl_1})\cdot f_x(1_A)\\
		&=\ol{(\0_t(\0_{t\m}(b))\dl_t}\cdot f_x(1_A)=\ol{b\dl_t}\cdot f_x(1_A).
	\end{align*}
	Thus,
	\begin{align*}
		\gm(f)(\ol{b\dl_t}\la(x\ot_{KE(S)}\ol{a\dl_s}))&=\ol{b\dl_t}\cdot f_x(1_A)\cdot\ol{a\dl_s}\overset{\cref{gm(f)(x-ot-a.dl_s)=f_s(1_A)a.dl_s}}{=}\ol{b\dl_t}\cdot\gm(f)(x\ot_{KE(S)}\ol{a\dl_s}),
	\end{align*}
	completing the proof that  $\gm(f)\in \Hom_{(A \rtimes_\0 S)^e}(X\ot_{KE(S)} (A \rtimes_\0 S) ,M).$
\end{proof}

\begin{lem}
	There is a $K$-linear map
	\begin{align*}
		\eta:\Hom_{(A \rtimes_\0 S)^e}(X\ot_{KE(S)}(A \rtimes_\0 S),M)\to\Hom_{KS}(X,\Hom_{A^e}(A,M))
	\end{align*}
	such that, for all $f\in\Hom_{(A \rtimes_\0 S)^e}(X\ot_{KE(S)}(A \rtimes_\0 S),M)$,
	\begin{align}\label{eta(f)_x(a)=f(x-ot-a.dl_1)}
		\eta(f)_x(a)=f(x\ot_{KE(S)}\ol{a\dl_1}).
	\end{align}
\end{lem}
\begin{proof}
	Let us show that $\eta(f)_x$ is an $A^e$-module map. For all $a,b,c\in A$ we have
	\begin{align*}
		\eta(f)_x(bac)&\overset{\cref{eta(f)_x(a)=f(x-ot-a.dl_1)}}{=}f(x\ot_{KE(S)}\ol{bac\dl_1})=f(x\ot_{KE(S)}\ol{b\dl_1}\cdot\ol{a\dl_1}\cdot\ol{c\dl_1})\\
		&\overset{\cref{right crossprod-action}}{=}f((1\cdot x\ot_{KE(S)}\ol{b\dl_1}\cdot\ol{a\dl_1})\cdot\ol{c\dl_1})\overset{\cref{a.dl_s(x-ot-b.dl_t)}}{=}f(\ol{b\dl_1}\cdot(x\ot_{KE(S)}\ol{a\dl_1})\cdot\ol{c\dl_1})\\
		&=\ol{b\dl_1}\cdot f(x\ot_{KE(S)}\ol{a\dl_1})\cdot\ol{c\dl_1}\overset{\cref{A-bimod M}}{=}b\cdot f(x\ot_{KE(S)}\ol{a\dl_1})\cdot c\\
		&\overset{\cref{eta(f)_x(a)=f(x-ot-a.dl_1)}}{=}b\cdot \eta(f)_x(a)\cdot c.
	\end{align*}
	Thus, $\eta(f)_x\in\Hom_{A^e}(A,M)$.
	
	We now prove that $\eta(f)$ is a homomorphism of left $KS$-modules.
	For all $s\in S$, $x\in X$ and $a\in A$ we have
	\begin{align*}
		(s\cdot \eta(f)_x)(a)&\overset{\cref{(sf)(a)=sf(s-inv-a)}}{=}s\la \eta(f)_x(s\m\la a)\overset{\cref{eta(f)_x(a)=f(x-ot-a.dl_1)}}{=}s\la f(x\ot_{KE(S)}\ol{(s\m\la a)\dl_1})\\
		&\overset{\cref{left KS-mod A}}{=}s\la f(x\ot_{KE(S)}\ol{\0_{s\m}(1_sa)\dl_1})\\
		&\overset{\cref{left KS-mod M}}{=}\ol{1_s\dl_s}\la f(x\ot_{KE(S)}\ol{\0_{s\m}(1_sa)\dl_1})\cdot\ol{1_{s\m}\dl_{s\m}}\\
		&=f(\ol{1_s\dl_s}\la(x\ot_{KE(S)}\ol{\0_{s\m}(1_sa)\dl_1})\cdot\ol{1_{s\m}\dl_{s\m}})\\
		&\overset{\cref{right crossprod-action}}{=}f(\ol{1_s\dl_s}\la(x\ot_{KE(S)}\ol{\0_{s\m}(1_sa)\dl_1}\cdot\ol{1_{s\m}\dl_{s\m}}))\\
		&\overset{\cref{a.dl_s(x-ot-b.dl_t)}}{=}f(s\la x\ot_{KE(S)}\ol{1_s\dl_s}\cdot\ol{\0_{s\m}(1_sa)\dl_1}\cdot\ol{1_{s\m}\dl_{s\m}})\\
		&=f(s\la x\ot_{KE(S)}\ol{1_sa\dl_s}\cdot\ol{1_{s\m}\dl_{s\m}})=f(s\la x\ot_{KE(S)}\ol{1_sa\dl_{ss\m}})\\
		&=f(s\la x\ot_{KE(S)}\ol{1_sa\dl_1})\overset{\cref{left KE(S)-action on cross}}{=}f(s\la x\ot_{KE(S)}ss\m\la \ol{a\dl_1})\\
		&=f(s\la x\cdot ss\m\ot_{KE(S)}\ol{a\dl_1})\overset{\cref{x-cdot-b=b-cdot-x}}{=}f(ss\m(s\la x)\ot_{KE(S)}\ol{a\dl_1})\\
		&=f((ss\m s)\la x\ot_{KE(S)}\ol{a\dl_1})=f(s\la x\ot_{KE(S)}\ol{a\dl_1})	\overset{\cref{eta(f)_x(a)=f(x-ot-a.dl_1)}}{=}\eta(f)_{s\cdot x}(a).
	\end{align*}
	Thus, $\eta(f)\in\Hom_{KS}(X,\Hom_{A^e}(A,M))$, and so $\eta$ is well-defined.
\end{proof}

\begin{lem}
	The maps $\gm$ and $\eta$ are inverse to each other.
\end{lem}
\begin{proof}
	Let us show that $\gm\circ\eta=\id$. For all $f\in \Hom_{(A \rtimes_\0 S)^e}(X\ot_{KE(S)}(A \rtimes_\0 S),M)$ we have
	\begin{align*}
		\gm(\eta(f))(x\ot_{KE(S)}\ol{a\dl_s})&\overset{\cref{gm(f)(x-ot-a.dl_s)=f_s(1_A)a.dl_s}}{=}\eta(f)_x(1_A)\cdot \ol{a\dl_s}\overset{\cref{eta(f)_x(a)=f(x-ot-a.dl_1)}}{=}f(x\ot_{KE(S)}\ol{1_A\dl_1})\cdot \ol{a\dl_s}\\
		&=f((x\ot_{KE(S)}\ol{1_A\dl_1})\cdot \ol{a\dl_s})\overset{\cref{right crossprod-action}}{=}f(x\ot_{KE(S)}\ol{1_A\dl_1}\cdot \ol{a\dl_s})\\
		&=f(x\ot_{KE(S)}\ol{a\dl_s}),
	\end{align*}
	whence $\gm(\eta(f))=f$, as desired.
	
	We now show that $\eta\circ\gm=\id$. Take $g\in\Hom_{KS}(X,\Hom_{A^e}(A,M))$. For all $x\in X$ and $a\in A$ we have
	\begin{align*}
		\eta(\gm(g))_x(a)&\overset{\cref{eta(f)_x(a)=f(x-ot-a.dl_1)}}{=}\gm(g)(x\ot_{KE(S)}\ol{a\dl_1})\overset{\cref{gm(f)(x-ot-a.dl_s)=f_s(1_A)a.dl_s}}{=}g_x(1_A)\cdot \ol{a\dl_1}\\
		&\overset{\cref{A-bimod M}}{=}g_x(1_A)\cdot a=g_x(1_A\cdot a)=g_x(a),
	\end{align*}
	whence $\eta(\gm(g))=g$, as desired.
\end{proof}

\begin{prop}\label{nat-iso-Homs}
	\begin{enumerate}
		\item\label{contravar-Homs-nat-iso} For arbitrary $A \rtimes_\0 S$-bimodule $M$ the functors 
		\begin{align*}
			\Hom_{KS}(-,\Hom_{A^e}(A,M)):KS\Mod\to K\Mod
		\end{align*}
		and
		\begin{align*}
			\Hom_{(A \rtimes_\0 S)^e}(-\ot_{KE(S)}(A \rtimes_\0 S),M):KS\Mod\to K\Mod
		\end{align*}
		are naturally isomorphic.
		
		\item\label{covar-Homs-nat-iso} For arbitrary left $KS$-module $X$ the functors
		\begin{align*}
			\Hom_{KS}(X,\Hom_{A^e}(A,-)):(A \rtimes_\0 S)^e\Mod\to K\Mod
		\end{align*}
		and
		\begin{align*}
			\Hom_{(A \rtimes_\0 S)^e}(X\ot_{KE(S)}(A \rtimes_\0 S),-):(A \rtimes_\0 S)^e\Mod\to K\Mod
		\end{align*}
		are naturally isomorphic.
	\end{enumerate}
\end{prop}
\begin{proof} Write $\Lb = A \rtimes_\0 S.$

	\cref{contravar-Homs-nat-iso}. Take an arbitrary morphism $\vf:X\to X'$ of left $KS$-modules and consider the diagram
	\begin{align}\label{diagr-contravar-Homs}
		\begin{tikzcd}[column sep=3.5cm,ampersand replacement=\&]
			{\Hom_{KS}(X,\Hom_{A^e}(A,M))} \& {\Hom_{KS}(X',\Hom_{A^e}(A,M))} \\
			{\Hom_{\Lb^e}(X\ot_{KE(S)} \Lb ,M)} \& {\Hom_{\Lb^e}(X'\ot_{KE(S)}\Lb ,M),}
			\arrow[swap,"{\Hom_{KS}(\vf,\Hom_{A^e}(A,M))}", from=1-2, to=1-1]
			\arrow["{\Hom_{\Lb^e}(\vf\ot_{KE(S)}\Lb,M)}", from=2-2, to=2-1]
			\arrow[swap,"\gm", from=1-1, to=2-1]
			\arrow["\gm'", from=1-2, to=2-2]
		\end{tikzcd}
	\end{align}
	where $\gm'$ is defined as $\gm$ in \cref{construction-of-gm}, but with $X$ replaced by $X'$. For any $f\in \Hom_{KS}(X',\Hom_{A^e}(A,M))$ let $g$ be its image under $\Hom_{KS}(\vf,\Hom_{A^e}(A,M))$. Then 
	\begin{align*}
		\gm(g)(x\ot_{KE(S)}\ol{a\dl_s})\overset{\cref{gm(f)(x-ot-a.dl_s)=f_s(1_A)a.dl_s}}{=}g_x(1_A)\cdot\ol{a\dl_s}=f_{\vf(x)}(1_A)\cdot\ol{a\dl_s}.
	\end{align*}
	Since $\Hom_{\Lb^e}(\vf\ot_{KE(S)}\Lb,M)(\gm'(f))(x\ot_{KE(S)}\ol{a\dl_s})$ equals
	\begin{align*}
		\gm'(f)(\vf(x)\ot_{KE(S)}\ol{a\dl_s})=f_{\vf(x)}(1_A)\cdot\ol{a\dl_s},
	\end{align*}
	diagram \cref{diagr-contravar-Homs} commutes, which proves \cref{contravar-Homs-nat-iso}.
	
	\cref{covar-Homs-nat-iso}. Take an arbitrary morphism $\psi:M\to M'$ of $\Lb$-bimodules and consider the diagram
	\begin{align}\label{diagr-covar-Homs}
		\begin{tikzcd}[column sep=3.5cm,ampersand replacement=\&]
			{\Hom_{KS}(X,\Hom_{A^e}(A,M))} \& {\Hom_{KS}(X,\Hom_{A^e}(A,M'))} \\
			{\Hom_{\Lb^e}(X\ot_{KE(S)}\Lb,M)} \& {\Hom_{\Lb^e}(X\ot_{KE(S)}\Lb,M'),}
			\arrow["{\Hom_{KS}(X,\Hom_{A^e}(A,\psi))}", from=1-1, to=1-2]
			\arrow[swap,"{\Hom_{\Lb^e}(X\ot_{KE(S)}\Lb,\psi)}", from=2-1, to=2-2]
			\arrow[swap,"\gm", from=1-1, to=2-1]
			\arrow["\gm'", from=1-2, to=2-2]
		\end{tikzcd}
	\end{align}
	where $\gm'$ is as in \cref{construction-of-gm}, but with $M$ replaced by $M'$. Given a morphism $f\in \Hom_{KS}(X,\Hom_{A^e}(A,M))$, let $g=\Hom_{KS}(X,\Hom_{A^e}(A,\psi))(f)$. Then 
	\begin{align*}
		\gm'(g)(x\ot_{KE(S)}\ol{a\dl_s})\overset{\cref{gm(f)(x-ot-a.dl_s)=f_s(1_A)a.dl_s}}{=}g_x(1_A)\cdot\ol{a\dl_s}=\psi(f_x(1_A))\cdot\ol{a\dl_s}.
	\end{align*}
	Since $\Hom_{\Lb^e}(X\ot_{KE(S)}\Lb,\psi)(\gm(f))(x\ot_{KE(S)}\ol{a\dl_s})$ equals
	
	\begin{align*}
		\psi(\gm(f)(x \ot_{KE(S)}\ol{a\dl_s}))
\overset{\cref{gm(f)(x-ot-a.dl_s)=f_s(1_A)a.dl_s}}{=}  \psi(f_x(1_A)\cdot\ol{a\dl_s})=\psi(f_x(1_A))\cdot\ol{a\dl_s},
	\end{align*} 
		diagram \cref{diagr-covar-Homs} commutes, which proves \cref{covar-Homs-nat-iso}.
\end{proof}

\begin{cor} \label{T_2T_1-iso-T}
	The functors $T_2T_1$ and $T$ are naturally isomorphic.
\end{cor}
\begin{proof}  As above, denote 
$\Lb = A \rtimes_\0 S.$ Applying \cref{nat-iso-Homs}\cref{covar-Homs-nat-iso} with $X=KE(S)$ we see that $T_2T_1$ is naturally isomorphic to 
	\begin{align*}
		\Hom_{\Lb^e}(KE(S)\ot_{KE(S)}\Lb,-):\Lb^e\Mod\to K\Mod.
	\end{align*}
	Let $\vf:\Lb\to KE(S)\ot_{KE(S)}\Lb$ be the isomorphism of $\Lb$-bimodules from \cref{l_AsemidirSisoBsemidirS}. For any morphism of $\Lb$-bimodules $\psi:M\to M'$ consider the following diagram
	\begin{align}\label{diagr-Hom(Lb_M)-and-Hom(B-ot-Lb_M)}
		\begin{tikzcd}[column sep=3.5cm,ampersand replacement=\&]
			{\Hom_{\Lb^e}(\Lb,M)} \& {\Hom_{\Lb^e}(\Lb,M')} \\
			{\Hom_{\Lb^e}(KE(S)\ot_{KE(S)}\Lb,M)} \& {\Hom_{\Lb^e}(KE(S)\ot_{KE(S)}\Lb,M').}
			\arrow["{\Hom_{\Lb^e}(\Lb,\psi)}", from=1-1, to=1-2]
			\arrow[swap,"{\Hom_{\Lb^e}(KE(S)\ot_{KE(S)}\Lb,\psi)}", from=2-1, to=2-2]
			\arrow["{\Hom_{\Lb^e}(\vf,M)}", from=2-1, to=1-1]
			\arrow[swap,"{\Hom_{\Lb^e}(\vf,M')}", from=2-2, to=1-2]
		\end{tikzcd}
	\end{align}
	Given $f\in \Hom_{\Lb^e}(KE(S)\ot_{KE(S)}\Lb,M)$, we have
	\begin{align*}
		(\Hom_{\Lb^e}(\Lb,\psi)\circ\Hom_{\Lb^e}(\vf,M))(f)&=\psi\circ (f\circ\vf),\\
		(\Hom_{\Lb^e}(\vf,M')\circ\Hom_{\Lb^e}(KE(S)\ot_{KE(S)}\Lb,\psi))(f)&=(\psi\circ f)\circ\vf,
	\end{align*}
	so \cref{diagr-Hom(Lb_M)-and-Hom(B-ot-Lb_M)} establishes a natural isomorphism between $\Hom_{\Lb^e}(KE(S)\ot_{KE(S)}\Lb,-)$ and $\Hom_{\Lb^e}(\Lb,-)=T$. Thus, $T_2T_1\cong T$.
\end{proof}

In an analogy with the homology case we give the next:

\begin{lem}\label{lem:for0-cohomology} Let    $\0$ be a  unital action of an   inverse monoid $S$ on  an algebra $A$ over a commutative ring $K.$ Let  $M$ be an $A\rtimes_\0 S$-bimodule and write 
$$M^A = \{x\in M \; \mid \; a\cdot x =x\cdot a  \;\;\; \text{for all} \;\;\; a\in A\}.$$   Then $M^A$ is a left $KS$-submodule of $M$ and the
$K$-module mapping
 $${\psi}': \Hom_{A^e}(A,M) \to   M^A, \;\;\; f   \mapsto  f(1_A),$$  is an isomorphism of left $KS$-modules.
 \end{lem}
 \begin{proof} It is well-known and  readily seen  that ${\psi}'$ is a well-defined $K$-module isomorphism. Moreover, for any $x\in M^A, s \in S$ and $a \in A,$ in view of  \eqref{left KS-mod A},  \eqref{left KS-mod M} and \eqref{A-bimod M}, we see that
\begin{align*}
a\cdot (s\cdot x) &= 
 (\overline{a\delta _1} \overline{ 1_s \delta _s}) \cdot x \cdot (\overline{ 1_{s\m} \delta _{s\m}})=
(\overline{ 1_s \delta _s  \theta _{s\m}(a1_{s})\delta _1}) \cdot x \cdot (\overline{ 1_{s\m} \delta _{s\m}})\\
&= (\overline{ 1_s \delta _s}) \cdot 
(\theta _{s\m}(a1_{s}) \cdot x) \cdot 
(\overline{ 1_{s\m} \delta _{s\m}})=
(\overline{ 1_s \delta _s}) \cdot 
(  x\cdot \theta _{s\m}(a1_{s})) \cdot 
(\overline{ 1_{s\m} \delta _{s\m}})\\
&=(\overline{ 1_s \delta _s}) \cdot 
  x\cdot (\overline{\theta _{s\m}(a1_{s}) 
 1_{s\m} \delta _{s\m}})
 =(\overline{ 1_s \delta _s}) \cdot 
  x\cdot (\overline{ 1_{s\m} \delta _{s\m} a\delta_1})= (s\cdot x)\cdot a,
\end{align*} showing that $M^A$ is a $KS$-submodule of $M.$ Finally, for any $f \in  \Hom_{A^e}(A,M)$
and $s\in S,$ using \eqref{left KS-mod M}, \eqref{A-bimod M} and \eqref{(sf)(a)=sf(s-inv-a)} we have that 
\begin{align*}
{\psi}'(s\cdot f) &=  (s\cdot f)(1_A) = 
s\cdot f(s\m \cdot 1_A)= 
(\overline{ 1_s \delta _s} ) \cdot f(1_{s\m}) \cdot (\overline{ 1_{s\m} \delta _{s\m}})\\
&= (\overline{ 1_s \delta _s} ) \cdot (f(1_A)\cdot 1_{s\m}) \cdot (\overline{ 1_{s\m} \delta _{s\m}})
= (\overline{ 1_s \delta _s} ) \cdot f(1_A) \cdot (\overline{ 1_{s\m} \delta _{s\m}})\\
&= s \cdot f(1_A) = s\cdot {\psi}' (f),
\end{align*} proving that ${\psi}' $ is a left $KS$-module isomorphism.
	\end{proof}

\subsection{Cohomology of $A\rtimes_\0 S $ over a field $K$}  In this subsection we assume that $K$ is a field.


 Recall from  Definition~\ref{def:Inv Monoid (co)homol} that, given a left $KS$-module $N$, the cohomology of $S$ with values in $N$ is defined by $H^n(S,N)=\Ext^n_{KS}(KE(S),N).$

\begin{thrm}\label{E^pq_2=>H^(p+q)}
	Let $\0$ be a unital action of an inverse monoid $S$ on an algebra $A$  over a field $K.$
	Then for any $A\rtimes_\0 S$-bimodule $M$ there exists a third quadrant cohomology spectral sequence
	\begin{align*}
		E^{p,q}_2=H^p(S,(R^qT_1)M)\impl H^{p+q}(A\rtimes_\0 S,M).
	\end{align*} 
\end{thrm}
\begin{proof}We claim  that the functor $T_1$ sends injective 
	$(A \rtimes_\0 S)^e$-modules to right $T_2$-acyclic $KS$-modules. Indeed, let $Q$ be an injective $(A \rtimes_\0 S)^e$-module and $P_\bullet\to KE(S)$ be a projective resolution of $KE(S)$ in $KS\Mod$. In particular, the functor 
	$\Hom  _{(A \rtimes_\0 S)^e}(-, Q)$ is exact. 
	 Recall that, since $K$ is a filed, the $K$-algebra  $KE(S)$ is  von Neumann regular, and therefore $A \rtimes_\0 S$ is  a flat $KE(S)$-module. Hence, the functor 
	$-\ot_{KE(S)} (A \rtimes_\0 S)$ is also exact. Then 
	\begin{align*}
		(R^nT_2)(T_1(Q))&=\Ext^n_{KS}(KE(S),\Hom_{A^e}(A,Q))=H^n(\Hom_{KS}(P_\bullet,\Hom_{A^e}(A,Q))).
	\end{align*}
	By \cref{nat-iso-Homs}\cref{covar-Homs-nat-iso} the latter equals $H^n(\Hom_{(A \rtimes_\0 S)^e}(P_\bullet\ot_{KE(S)}(A \rtimes_\0 S),Q))$, which is $0$ for all $n\ge 1$, because $\Hom_{(A \rtimes_\0 S)^e}(-,Q)$ and $-\ot_{KE(S)}(A \rtimes_\0 S)$ are exact. This proves our claim.
	
	Next, since  the functor $T_2$ is left exact, in view of the natural isomorphism  $T_2T_1\cong T,$ provided  by \cref{T_2T_1-iso-T},  and
	\begin{align*}
		(R^nT)M=\Ext^n_{(A\rtimes_\0 S)^e}(A\rtimes_\0 S,M)=H^n(A\rtimes_\0 S,M),
	\end{align*}
 we obtain from \cite[Theorem 10.47]{Rotman}  the third quadrant Grothendieck spectral sequence
    \begin{equation}\label{eq:CohomSpectralSeqFromRotman}
        E_2^{p,q} = (R^p T_2) (R^q T_1) M \Rightarrow H^{p+q}(A\rtimes_\0 S,M). 
        \end{equation} 
	 Since by definition of cohomology of $S$, we have that 
    $$(R^p T_2) (R^q T_1) M = \operatorname{Ext}_{K S}^p (K E(S), (R^q F_1)M  ) =
    H^p(S, (R^q F_1)M  ),
    $$  the spectral sequence \eqref{eq:CohomSpectralSeqFromRotman} takes the desired form.
\end{proof}

\begin{cor}\label{cor:FlatCondition}
	Under the assumptions of \cref{E^pq_2=>H^(p+q)},  assume that $(A\rtimes_\0 S)^e$ is flat as a right $A^e$-module. Then for any $A\rtimes_\0 S$-bimodule $M$ there exists a third quadrant cohomology spectral sequence
	\begin{align*}
		E^{p,q}_2=H^p(S,H^q(A,M))\impl H^{p+q}(A\rtimes_\0 S,M).
	\end{align*} 
\end{cor}
\begin{proof}
	Since  $(A\rtimes_\0 S)^e$ is a flat right $A^e$-module, it follows from \cite[Corollary 3.6A]{Lam} that any injective left $(A\rtimes_\0 S)^e$-module is an injective left $A^e$-module. Thus, any injective  resolution of $M$ in $(A\rtimes_\0 S)^e\Mod$ is an injective  resolution of $M$ in $A^e\Mod$. Consequently, the right derived functor of $T_1$ computes the Hochschild cohomology of $A$ with values in $M$, i.e.
	\begin{align*}
		(R^qT_1)M=\Ext^q_{A^e}(A,M)=H^q(A,M),
	\end{align*}
	and the result follows from \cref{E^pq_2=>H^(p+q)}.
\end{proof}

\begin{lem}\label{lem:0-cohomology} Let $\0$ be a unital action of an  inverse monoid $S$ on  a $K$-algebra $A$ over a field $K$ and $M$ an $A\rtimes_\0 S$-bimodule. Assume that 
	$(A\rtimes_\0 S)^e $ is flat as a right $A^e$-module. Then there is a  left $KS$-module isomorphism
	$$H^0(A,M)  \cong M^A.  $$
\end{lem}
\begin{proof} In view of Lemma~\ref{lem:for0-cohomology} it is enough to establish a left $KS$-module isomorphism between $H^0(A,M) $ and $\Hom_{A^e}(A,M).$   An isomorphism between these $K$-spaces is well-known and we will easily  show that  it is a  $KS$-isomorphism.

Let 
\begin{equation}\label{eq:InjResolution}
0\to M \to Q_0 \stackrel{d^0}{\to} Q_1 \stackrel{d^1}{\to}  \ldots
 \end{equation} be 
an injective resolution of $M$ in $(A\rtimes_\0 S)^e\Mod.$ Applying to the exact sequence 
$0\to M \to Q_0 \stackrel{d^0}{\to} Q_1$  the left exact functor  $T_1,$ we obtain the exact sequence  of left $KS$-modules and  $KS$-homomorphisms
$$0\to \Hom_{A^e}(A, M) \to   \Hom_{A^e}(A, Q_0) \stackrel{\Hom_{A^e}(A,d^0)}{\longrightarrow }  \Hom_{A^e}(A,  Q_1) .$$
   Then, as in the proof Corollary~\ref{cor:FlatCondition}, we see that the complex \eqref{eq:InjResolution} is also  an injective resolution of $M$ in $A^e\textbf{-Mod}$ and consequently  
 $$H^0(A,M) = {\rm ker}\, (\Hom_{A^e}(A,d^0)).$$
 This gives   the desired isomorphism between the left $KS$-modules $H^0(A,M)$ and  $\Hom_{A^e}(A,M).$ 
\end{proof}

\begin{cor}\label{rem:separableCohom}
    Under the conditions of \cref{E^pq_2=>H^(p+q)}, assume that $A$ is separable over $K.$ Then  for any $A\rtimes_\0 S$-bimodule $M$ there is  an isomorphism
    \[
         H^n(S, M^A) \cong H^{n}(A\rtimes_\0 S, M), 
    \] where $M^A$ is the $K$-submodule of invariants of $M,$ as defined in Lemma~\ref{lem:for0-cohomology}.
\end{cor}
\begin{proof}
	Since  $A$ is separable,   $A^e$ is semisimple and any left or right  $A^e$-module is projective.  In particular,  $(A\rtimes_\0 S)^e$ is flat as a right $A^e$-module, and we may apply Corollary~\ref{cor:FlatCondition}. Moreover, $A$ is projective as a left  $A^e$-module and, consequently,  the spectral sequence collapses on the $p$-axis, so that  by 
	\cite[Proposition 10.26]{Rotman} we obtain an isomorphism
	$$ H^n(S, H^0(A,M)) \cong H^{n}(A\rtimes_\0 S, M).  $$ Finally, by Lemma~\ref{lem:0-cohomology},  
	the left $K S$-modules  $H^0(A,M)$ and $M^A$ are isomorphic.
\end{proof}


\subsection{Cohomology of $A\rtimes_\0 S $ with $E$-unitary $S$}
 As we did in the case of homology, dealing with an   $E$-unitary $S$ we will abandon  the restriction on $K$ to be a field, replacing it by the flatness of $A$  over $K$ when needed. 

\begin{prop}\label{prop:T1sendInjtoT2acylic}
   Let $\0$ be a unital action of an $E$-unitary inverse monoid $S$ on an algebra $A$  over a commutative ring $K.$ 
   Then  $T_1$ sends injective $(A\rtimes_\0 S)^e$-modules to right $T_2$-acyclic modules.
\end{prop}

\begin{proof} Let $Q$ be an arbitrary injective $(A\rtimes_\0 S)^e$-module.
    We need to show that
    \[
        (R^n T_2)(\operatorname{Hom}_{A^e}(A,Q)) = 0 ,
    \] for all $n > 0.$
     Notice that
    \begin{align*}
        (R^n T_2)(\operatorname{Hom}_{A^e}(A,Q))
        &= R^n(\operatorname{Hom}_{K S}(K E(S),-))(\operatorname{Hom}_{A^e}(A,Q)) \\
        &\cong R^n(\operatorname{Hom}_{K S}(-,\operatorname{Hom}_{A^e}(A,Q)))(K E(S)) \\
       (\text{by Proposition \ref{nat-iso-Homs} (i)}) &\cong R^n(\operatorname{Hom}_{(A\rtimes_\0 S)^e}(- \otimes_{K E(S)} (A\rtimes_\0 S), Q))(K E(S)).
    \end{align*}
    Take a projective resolution  $P_{\bullet} \to K E(S)$  of $K E(S)$ in $K S$-\textbf{Mod}. Then,
    \[
        (R^n T_2)(\operatorname{Hom}_{A^e}(A,Q)) \cong H^n \left( \operatorname{Hom}_{(A\rtimes_\0 S)^e}(P_\bullet \otimes_{K E(S)} (A\rtimes_\0 S), Q)\right).
    \]
    Observe  that   Lemma~\ref{lem:acyclic} states that $H_n(P_{\bullet}\otimes_{K E(S)} (A\rtimes_\0 S))=0$ for all $n\geq 1,$ which means that the complex  $P_{\bullet}\otimes_{K E(S)} (A\rtimes_\0 S)$ is exact at the $n$-th term for all $n\ge 1$. Then the chain complex $\operatorname{Hom}_{(A\rtimes_\0 S)^e}(P_\bullet \otimes_{K E(S)} (A\rtimes_\0 S), Q)$ is also exact at the $n$-th term for all $n \geq 1,$ because $Q$ is injective as an $(A\rtimes_\0 S)^e$-module. Therefore, 
    \[
        H^n \left( \operatorname{Hom}_{(A\rtimes_\0 S)^e}(P_\bullet \otimes_{K E(S)} (A\rtimes_\0 S), Q)\right) = 0,\ \forall n \geq 1,
    \] 
    completing our proof.  \end{proof}

\begin{thrm} \label{teo:EunitaryCohomSpectralSeq}
 Let $\0$ be a unital action of an $E$-unitary inverse monoid $S$ on an algebra $A$  over a commutative ring $K.$ Assume that $A$ is flat over $K.$ Then, for any $(A\rtimes_\0 S)^e$-module $M$ there exists a third quadrant  cohomology spectral sequence $E_r$ such that
    \begin{equation}\label{eq:EunitaryCohomSpectralSeq}
        E_2^{p,q} = H^p (S,H^q(A,M) ) \Rightarrow H^{p+q}(A\rtimes_\0 S, M). 
    \end{equation}
\end{thrm}

\begin{proof}  By Proposition~\ref{prop:T1sendInjtoT2acylic} the functor $T_1$ sends injective $(A\rtimes_\0 S)^e$-modules to right $T_2$-acyclic modules. Then, since   $T_2$ is a  left exact functor and by Corollary~\ref{T_2T_1-iso-T} there is a natural isomorphism  $T_2 T_1 \cong T,$ we are in conditions to apply   \cite[Theorem 10.47]{Rotman}  to obtain the  spectral sequence
\[
        E_2^{p,q} = (R^p T_2)(R^q T_1)M \Rightarrow (R^{p+q}T) M. 
    \]
 Obviously, $$(R^{p+q}T)(-)=H^{p+q}(A\rtimes_\0 S, -)$$ and $$R^p T_2 (-) = \operatorname{Ext}_{K S}^p(K E(S),-)= H^p(S,-),$$
      so that the spectral sequence takes the form
     \begin{equation}\label{eq:AlmostFinalCohomolSpectralSeq}
        E_2^{p,q} = H^p_{K S}(KS, (R^q T_1)M ) \Rightarrow H^{p+q}(A\rtimes_\0 S, M).
	 \end{equation} 

		Thanks to  Lemma~\ref{lem:FlatnessInE-unitaryCase} the right $A^e$-module $(A\rtimes_\0 S)^e$ is flat and, as in the  proof of Corollary~\ref{cor:FlatCondition},
		we conclude that
		\[ 
       R^q T_1(-)=H^q(A,-), 
    \] so that the spectral sequence  
	\eqref{eq:AlmostFinalCohomolSpectralSeq} takes the announced form \eqref{eq:EunitaryCohomSpectralSeq}.
  \end{proof}

  If $K$ is a field, then the flatness of $A$ over $K$ is guaranteed, so that we immediately obtain the next:
  
  \begin{cor}
   Let $\0$ be a unital action of an $E$-unitary inverse monoid $S$ on an algebra $A$  over a field $K.$  Then, for any $(A\rtimes_\0 S)^e$-module $M$ there exists a third quadrant  cohomology spectral sequence $E_r$  of the form \eqref{eq:EunitaryCohomSpectralSeq}.
   \end{cor}




	\section{Application to Steinberg algebras}\label{sec:Steinberg}
	
	 We shall consider unital Steinberg algebras, for which we recall some notions and facts.

\underline{In all what follows} $K $ will be a field and $X$ a topological space, on which some assumptions will be imposed soon. A function $f:X \to K $ is said to be {\it locally constant} if, for every $x$ in $X$, there exists a neighborhood $V$ of $x$, such that $f$ is
constant on $V$. Considering the discrete topology on $K,$ it is readily seen that a function $f: X \to K$ is continuous if and only if it is locally constant.  The {\it support} of $f$ is defined to be the set
  $$
  {\rm supp} (f) = \{x\in X: f(x)\neq 0\}.
  $$
  Observe that the support of a locally constant function $f$ is always closed.
  

  Let $S$ be a discrete inverse semigroup. A topological action  of $S$ on $X$ is  a semigroup homomorphism $\0:S\to\cI(X)$, $s\mapsto\0_s$, such that $\dom\0_s$ and $\ran\0_s$ are open subsets of $X$ and $\0_s$ is a homeomorphism for all $s\in S$.  Write 
$$X_{s\m }:= \dom\0_s \;\;\ \text{and} \;\;\; X_s := \ran\0_s$$ for all $s \in S.$ 
 As in the case of actions on algebras,   $\0_e$ is the identity map
$X_{e } \to X_{e },$ for each idempotent $e\in S,$ and the equality $\0 _s \circ \0\m_s =   \0_{s s\m},$ readily implies  that
 $$X_{s s\m} = X_s,$$
for all $s\in S.$ Hence, $\0 _{s s\m}$ is the identity map $X_s  \to  X_s.$ 
  
  \underline{Assume for the rest of the paper} that, unless otherwise stated,    $X$ is a compact Hausdorff space whose topology admits a basis formed by
  compact open~\footnote{That  is, sets which are simultaneously compact and open.} subsets. Since $S$ is Hausdorff, a subset of $X$ is closed if and only if it is compact.
    Given any compact open set $E\subseteq X$, it is easy to see that its characteristic function $1_E$ is
locally constant.  Moreover, it is not difficult to prove  that every locally constant function $f:X\to K$ is a linear combination of the form
  \begin{equation}\label{eq:LinearCombination}
  f = \sum_{i=1}^m c_i\,1_{E_i},
  \end{equation}
  where $E_i$ are pairwise disjoint compact open subsets and  $c_i \in K $. We will  denote by $\LX$ the set of all locally constant  $K $-valued  functions on $X$.
With point-wise multiplication, $\LX$ is a commutative $K $-algebra  with unity $1_X.$

Let $\theta $ be a topological action of $S$ on $X$ such that  the subset $X_s$  is closed in $X$  (besides being open) for every $s\in S$. For each $s$ in $S$, we may  consider the $K  $-algebra $\LXs$, which we will identify with
the set formed by all $f$ in $\LX$ vanishing on $X\setminus X_s$.  Under this identification $\LXs $ becomes an
ideal in $\LX ,$ generated by the idempotent element $1_{X_s}.$


Using the action $\theta ,$ we  define an action $\hat{\0}$ of $S$ on the algebra $\LX$ as follows:\\
the mapping 
  \begin{equation}\label{eq:InducedAlgAction}
  \hat{\theta}_g:\LXsm \to\LXs,
  \end{equation}
  such that
  $$ 
  \hat{\0}_s(f) = f\circ\theta_{s\m} \in\LXsm,
  $$
   is an isomorphism of algebras, and the collection formed by all ideals $\LXs $, together with the collection of all $\hat{\0}_s$, is  easily seen to be a unital   action  of $S$ on $\LX .$

Recall that a groupoid is a small category in which each morphism is an isomorphism. The set of unit elements of a  groupoid $\sG$ will be denoted by ${\sG}^{(0)}$ and identified with the set of  objects of $\sG.$ Every morphism $\gamma $ of $\sG$ has a source (domain) ${\bf s}(\gamma $) and a range (codomain)
 ${\bf r}(\gamma ) ,$ which are identified with the units ${\gm}\m \gm $ and  $\gm {\gm}\m ,$ respectively.   This defines the source map ${\bf s} : \sG \to \sG^{(0)},$ $\gm \mapsto {\gm}\m \gm ,$
and the range map ${\bf r} : \sG \to \sG^{(0)},$ $\gm \mapsto {\gm} {\gm}\m .$ The set of composable pairs is denoted by 
$\sG^{(2)},$ i.e. 
$$ \sG^{(2)} = \{ (\alpha , \beta ) \in \sG \times \sG \;  \mid \; 
{\bf s} (\alpha) = {\bf r}(\beta) \}.
$$ A subset $U$ of $\sG$ is called a \textit{bisection} if the restrictions ${\bf s}|_U$ and ${\bf r}|_U$ are injective.

A {\it topological groupoid} is a groupoid $\sG$ whose underlying set is equipped with a (non-necessarily Hausdorff) topology making
 the product and inversion continuous, where the set of composable pairs is given the 
 topology induced from the product topology of $\sG\times \sG.$
 
 An {\it \'etale}  groupoid
 is a topological groupoid $\sG$, whose unit space $\sG^{(0)}$ is locally compact and Hausdorff in the relative topology, and such that the range map ${\bf r} : \sG \to \sG^{(0)}$ is a
 local homeomorphism. The latter means that for any $\gm \in \sG$ there are an open subset $U$ of $\sG$ containing $\gm $ 
  and an open subset $V$ of $\sG^{(0)}$ containing ${\bf r}(\gm)$, such that ${\bf r}(U)=V ,$   and ${\bf r}|_U$
 is a homeomorphism onto $V$.   Since ${\bf d}(\gm ) ={\bf r}({\gm}\m),$ for every $\gm  \in \sG,$ it follows  that ${\bf  d}$  is also a local
 homeomorphism. Like any local homeomorphism, ${\bf  d}$ and ${\bf  r}$ are open
 maps. An important property of an {\'e}tale groupoid $\sG $ is that its unit space is open \cite[Proposition 3.2]{Exel2008}. Thus, $\sG^{(0)}$ is an open bisection of $\sG.$ It is easily seen that for any open bisection $U$ of $\sG $ the subsets ${\bf  d}(U)$ and ${\bf  r}(U)$ are open in $\sG$ and the restrictions  ${\bf  d}|_U$ and ${\bf  r}|_U$ are homeomorphisms.
 
 Given open bisections $U$ and $V$ of an {\'e}tale groupoids
 $\sG$ the (possibly empty)
set-wise product   
  $$UV =\{uv \; \mid \; u \in U, v\in V\}$$
 and the set-wise inverse
 $$U\m = \{u\m \; \mid \; u\in U\},$$ 
 are also open bisections,  such that with these operations the set of all open bisections ${\mathcal S} (\sG )$ of $\sG$ is an inverse monoid, whose idempotents are the open subsets of $\sG ^{(0)}$ and the unity is  $\sG^{(0)}$ (see \cite[Proposition 3.8]{Exel2008}).

  Considering an  {\'e}tale groupoid $\sG ,$  for each $U\in  {\mathcal S} (\sG )$ define  the mapping 
\begin{equation}\label{eq:action}
 {\0}_U : {\bf d}(U) \to {\bf r}(U) ,
 \end{equation} by setting $$  {\0}_U(x) = {\bf r}(\gm ), $$ where $x\in {\bf d}(U)$ and $\gamma $ is the unique element $U$ with ${\bf d} (\gm )=x.$ Then the collection of the mappings $\0 _U,$ where $U$ runs over the open bisections of  $\sG ,$ constitutes a topological action $\0$ of the inverse semigroup ${\mathcal S} (\sG )$ on the unit space $\sG ^{(0)}$  (see \cite[Proposition 5.3]{Exel2008}).
 
 An {\'e}tale groupoid $\sG $ is called {\it ample} if the topology of  $\sG $ admits a basis formed by compact open bisections. It is well-known and can be directly verified  that the compact open bisections of an ample groupoid $\sG$ form an inverse subsemigroup of  ${\mathcal S} (\sG ),$ which is usually denoted by ${\mathcal S}^{a}(\sG ).$
 
  Given  an ample groupoid  $\sG ,$ let
  $A_K(\sG)$  be the $K$-subspace of $K^{\sG}$ spanned by the characteristic functions of compact
 open bisections  of $\sG.$  Every element of  $A_K(\sG)$ is a linear combination of characteristic functions
 of pairwise disjoint compact bisections (see \cite[Proposition 4.3]{St2010}).   Evidently, every compact open subset of $\sG ^{(0)}$ is a compact open bisection of $\sG ,$ so that its characteristic function is an element of $A_K(\sG).$   For $f, g\in A_K(\sG)$ their product is defined  by the convolution formula
  $$(f \ast g )(\gm ) =
  \sum _{\af \beta = \gamma} f(\af ) g (\bt), $$ for each $\gm \in \sG.$ The summation above is taken over all $\af, \bt \in \sG$ such that $\af \bt = \gm.$ It is easily seen that the above sum is finite, $f \ast g$ is an element of   $A_K(\sG)$ and  $$1_U \ast 1_V = 1_{UV} , \;\;\; U,V \in {\mathcal S} (\sG ).$$  With this multiplication $A_K(\sG)$ becomes an algebra over $K,$ called the {\it Steinberg algebra} of the ample groupoid $\sG .$ 
  
  Since we work with unital algebras, \underline{we assume that} $\sG$ is an ample groupoid, 
  whose unit space $X=\sG ^{(0)}$ is compact, so that $A_K(\sG)$ is a unital algebra with unity 
  element $1_{X}.$
  
  If $\sG$ is Hausdorff then each compact open bisection of $\sG$ is closed in $\sG$ and, consequently, its characteristic 
  function is continuous. This yields that each element of  $A_K(\sG)$ is a continuous compactly 
  supported  function, and   $A_K(\sG)$ can be seen as the set of all continuous compactly 
  supported functions of the form $\sG \to K.$ However, if $\sG$ is not Hausdorff then it is easy to see that
 $\sG$ contains a compact open bisection $U$ which is not closed in  $\sG ,$ so that $1_U$ is not continuous.
 Nevertheless, if $U$ is a compact open subset of $X,$ then $1_U$ is continuous  as a function 
 $X \to K,$ as well as  as a function $\sG \to K.$ 
 Note also that in view of  \eqref{eq:LinearCombination} we may consider $\LX$  as a subalgebra  of $A_K(\sG).$

 Since  ${\bf s} (U)$ and
  ${\bf r} (U)$ are compact open subsets of $X$ (and also of $\sG$),  the topological action 
  \eqref{eq:action} of ${\mathcal S}^{a}(\sG )$ on $X$ gives rise 
  to a  unital action $\hat{\0}$ of  ${\mathcal S}^{a}(\sG )$  on $\LX $ given by \eqref{eq:InducedAlgAction}.
Given  $U \in {\mathcal S}^{a}(\sG )$
 and $\varphi \in {\mathcal L}({\br}(U)) = {\mathcal L}(X_U) ,$ 
 define $\varphi \Delta _U \in  A_K(\sG)$ by setting
 $$\varphi \Delta _U = \begin{cases}
	\varphi \circ {\bf r}|_U \; \text{on}\; U,\\
	\;\;\; 0 \; \;\;\; \;\;\; \text{on}\; \sG \setminus U .
	\end{cases} $$ Then for any $\gm \in \sG$ we have that
$$\varphi \Delta _U (\gm ) = [\gm \in U] \varphi ({\bf r}(\gm) ),$$
where the brackets indicate {\it boolean value}, i.e. $[\bullet]$ is $1$ if $\bullet$ is true and  $0$ otherwise.
Then by \cite[Theorem 5.2]{BeuterGon2018} (see also \cite[Theorem 5.2.4]{Demeneghi2019} and 
\cite[Proposition 5.4]{Exel2008}) the mapping 
\begin{equation}\label{eq:IsoPsi}
	\Psi : \LX \rtimes _{\hat{\0}} {\mathcal S}^{a}(\sG ) \to A_K(\sG),
\end{equation} given by 
\begin{equation*}
	\overline{\varphi \delta _U} \mapsto \varphi \Delta _U,
\end{equation*} where $U \in {\mathcal S}^{a}(\sG ), \varphi \in {\mathcal L}({\br}(U)),$ is an isomorphism of $K$-algebras. 

It is interesting to observe the following fact.
\begin{rem}\label{lem:AbimoduleSteinbergIso} The mapping \eqref{eq:IsoPsi} is an isomorphism of $ \LX$-bimodules.
\end{rem}
\begin{proof} Observe that for any $f \in \LX,$  $\psi \in A_K(\sG)$ and $\gm \in \sG$ we have that
	$$(f \ast \psi )(\gm ) = \sum _{\af \bt = \gm} f(\af) \psi (\bt) = f({\bf r}(\gm)) \psi (\gm). $$
Hence, for any $\gm \in \sG,$ $U \in  {\mathcal S}^{a}(\sG ), f \in \LX$ and $\varphi \in {\mathcal L} ({\br}(U))$ we see that
\begin{align*}
	\Psi (f \cdot \overline{\varphi \delta _U}) (\gm) &=
	\Psi ( \overline{f \varphi \delta _U}) (\gm) =
 ( (f \varphi) \Delta_{U})(\gm )= [\gm \in U] (f \varphi) ({\br}(\gm))\\
 &= [\gm \in U] f ({\br}(\gm)) \varphi  ({\br}(\gm))
 =  f ({\br}(\gm)) (\varphi \Delta _U)(\gm)\\
& =  (f \ast  \varphi \Delta _U)(\gm) = (f  \ast \Psi(\overline{\varphi \delta _U})) (\gm),\\
\end{align*} showing that $\Psi$ is a homomorphism of left $\LX$-modules. As to the 
 right action
 of $\LX ,$ recalling that the identity element of the monoid ${\mathcal S}^{a}(\sG )$ is the compact open bisection $X= \sG ^{(0)}$ we compute, 
  on the one  hand,
\begin{align*}
	\Psi ( \overline{\varphi \delta _U} \cdot f) (\gm)
	&= \Psi ( \overline{\varphi \delta _U f \delta _X}) (\gm)
	=  \Psi ( \overline{\hat{\0}_U(\hat{\0}_{U\m}(\varphi ) f) \delta _{U}}) (\gm)\\
	&=  \Psi ( \overline{\hat{\0}_U((\varphi \circ \0 _U ) f) \delta _{U}}) (\gm)
	=  \Psi ( \overline{((\varphi \circ \0 _U ) f) \circ {\0}_{U\m}\delta _{U}}) (\gm)\\
	&=  (((\varphi \circ \0 _U ) f) \circ {\0}_{U\m}\Delta _{U}) (\gm)
	= [\gm \in U] (((\varphi \circ \0 _U ) f) \circ {\0}_{U\m} ({\br} (\gm))\\
	&= [\gm \in U] ((\varphi \circ \0 _U ) f)  ({\bf s} (\gm))
	= [\gm \in U] (\varphi \circ \0 _U )({\bf s} (\gm)) f  ({\bf s} (\gm))\\
	&= [\gm \in U] \varphi  ({\bf r} (\gm)) f  ({\bf s} (\gm)).
\end{align*}
On the other hand,
\begin{align*}
	(\Psi ( \overline{\varphi \delta _U} ) \ast   f) (\gm)
	&= ((\varphi \Delta _U ) \ast   f) (\gm) 
	= \sum _{\af \bt = \gm}\varphi \Delta _U (\af)    f (\bt)\\
	& = \varphi \Delta _U (\gm)    f ({\bf s}(\gm))
	= [\gm \in U] \varphi ({\bf r} (\gm))    f ({\bf s}(\gm)),
\end{align*} proving that $\Psi$ is also a right $\LX $-module mapping.  \end{proof} 

Our homology result for Steinberg algebras is as follows.
\begin{thrm}\label{teo:SteinbergHomology}
	Let $\sG $ be an ample groupoid with compact unit space $\sG ^{(0)}$ and $M$ an $ A_K(\sG)$-bimodule.  Then there is an isomorphism of homology groups
\begin{equation}\label{eq:SteinbergHomology}
	H_{n}(A_K(\sG), M) \cong  H_n({\mathcal S}^{a}(\sG ), M/[{\mathcal L}( \sG ^{(0)}), M]),
	\end{equation} where the left $K {\mathcal S}^{a}(\sG )$-module structure on $M$ is given by 
	\begin{equation}\label{eq:SliceLeftAction}
	U \cdot x =  (1_{{\bf r} (U)} \Delta_U) \cdot  x \cdot (1_{{\bf s} (U)} \Delta_{U\m}), 
	\end{equation} for any  $U \in {\mathcal S}^{a}(\sG )$ 
	and
	 $ x \in M.$
	 \end{thrm}
	 \begin{proof} Notice that the left action of 
		$K {\mathcal S}^{a}(\sG )$ on $M$ comes from formula  \eqref{left KS-mod M} and the isomorphism $\Psi $ 
		in \eqref{eq:IsoPsi}.
		 Note also that, as it can be seen  
		  in  Lemma~\ref{lem:for0-homology},  the left action of $K {\mathcal S}^{a}(\sG )$ on $M$ gives a left $K {\mathcal S}^{a}(\sG )$-module structure on $M/[\LX, M].$
		
		Recall that the $K$-algebra $\LX $ is generated by the characteristic functions $1_{U},$ where $U$ are compact open subsets of $X.$
	 It follows   that  $\LX \otimes _K  \LX$ is a commutative $K$-algebra generated by idempotents. 
	  By \cite[Proposition 5.14]{DJ2} any unital commutative algebra $A$ over a field  is von Neumann regular, if $A$ is generated by idempotents. Thus, $\LX \otimes _K  \LX$  is von Neumann regular. Note that $\LX = \LX^{\rm op} .$ Consequently,  any left  module over  $\LX ^e $ is flat (see  \cite[Corollary 1.13)]{Goodearl}). In particular, the algebra $(\LX \rtimes _{\hat{\0}} {\mathcal S}^{a}(\sG ) )^e$ is flat as  a left  $\LX ^e $-module.  
 Using the  $K$-algebra isomorphism $\Psi $ from \eqref{eq:IsoPsi}, we  transform the   $ A_K(\sG)$-bimodule $M$ into a $\LX \rtimes _{\hat{\0}} {\mathcal S}^{a}(\sG ) $-bimodule. Then  by Corollary~\ref{cor:flatHHH} there is a 
		first quadrant homology spectral sequence
	 \begin{equation}\label{eq:SteinbergHomolSpectralSeq}
		 E^2_{p,q} = H_p({\mathcal S}^{a}(\sG ), 
		 H_q (\LX, M)) \Rightarrow H_{p+q}(\LX \rtimes _{\hat{\0}} {\mathcal S}^{a}(\sG ) , M). 
	 \end{equation} In view of the fact that $\LX ^e$ is von Neumann regular, the left   $\LX ^e$-module $M$ is flat, which implies that   $H_q (\LX, M)=0$ for all $q>0.$ This yields  that the spectral sequence \eqref{eq:SteinbergHomolSpectralSeq} collapses on the $p$-axis, and by \cite[Proposition 10.21]{Rotman} we obtain an isomorphism of $K$-vector spaces
	 $$ H_n({\mathcal S}^{a}(\sG ), H_0(\LX ,M)) \cong H_{n}(\LX \rtimes_{\hat{\0}} {\mathcal S}^{a}(\sG ), M).$$ Then, since  by Lemma~\ref{lem:0-homology} we may write
	 $H_0(\LX,M) =  M/[\LX, M],$ we get an isomorphism
	 \begin{equation}\label{eq:StenbergHomolIso}
	  H_n({\mathcal S}^{a}(\sG ), M/[\LX, M]) \cong H_{n}(\LX \rtimes_{\hat{\0}} {\mathcal S}^{a}(\sG ), M).
	  \end{equation}
	 
	 Observe that $\Psi $ maps the copy of $\LX $ in $\LX \rtimes_{\hat{\0}} {\mathcal S}^{a}(\sG )$ onto the copy of $\LX$ in $A_K(\sG)$ and, consequently, the $\LX $-bimodule structure on $M$ coming from its $\LX \rtimes_{\hat{\0}} {\mathcal S}^{a}(\sG )$-bimodule structure is the same as that coming from the  $A_K(\sG)$-bimodule structure. Then it can be readily seen that there is an isomorphism of $K$-vector spaces
	 \begin{equation*}\label{eq:SkewAndSteinbergHomolIso}
H_{n}(\LX \rtimes_{\hat{\0}} {\mathcal S}^{a}(\sG ), M) \cong H_n (A_K(\sG),M),
	 \end{equation*} so that the announced isomorphism 
	 \eqref{eq:SteinbergHomology} follows from \eqref{eq:StenbergHomolIso}.
   \end{proof}
		
		For  cohomology we give the next:
		
		\begin{thrm}\label{teo:SteinbergCoHomology}
	Let $\sG $ be an ample groupoid with compact unit space $\sG ^{(0)}$ and $M$ an $ A_K(\sG)$-bimodule.  Then there is a third quadrant cohomology spectral sequence 
	\begin{align}\label{eq:SteinbergCohomolSpectralSeq}
		E^{p,q}_2=H^p({\mathcal S}^{a}(\sG ),H^q({\mathcal L}( \sG ^{(0)}) ,M))\impl H^{p+q}( A_K(\sG),M),
	\end{align} where the left $K {\mathcal S}^{a}(\sG )$-action on $M$ 
	 is  given  in \eqref{eq:SliceLeftAction}.
	 \end{thrm}
			\begin{proof} We saw in the proof of Theorem~\ref{teo:SteinbergHomology} that the $K$-algebra $\LX ^e$ is von Neumann regular.  Then   
			$(\LX \rtimes _{\hat{\0}} {\mathcal S}^{a}(\sG ))^e$ is flat as a right $\LX ^e$-module and, considering  the   $ A_K(\sG)$-bimodule $M$  as 
			a $\LX \rtimes _{\hat{\0}} {\mathcal S}^{a}(\sG ) $-bimodule via the $K$-algebra isomorphism $\Psi $ from \eqref{eq:IsoPsi}, we obtain 
			by Corollary~\ref{cor:FlatCondition}
			a third quadrant cohomology spectral sequence 
				\begin{align}\label{eq:SkewSteinbergCohomolSpectralSeq}
		E^{p,q}_2=H^p({\mathcal S}^{a}(\sG ),H^q(\LX ,M))\impl H^{p+q}(\LX \rtimes _{\hat{\0}} {\mathcal S}^{a}(\sG ),M).
	\end{align} Then one can readily  see the isomorphism 
	$$H^{p+q}(\LX \rtimes _{\hat{\0}} {\mathcal S}^{a}(\sG ),M) \cong H^{p+q}( A_K(\sG),M),$$ which transforms  \eqref{eq:SkewSteinbergCohomolSpectralSeq} into \eqref{eq:SteinbergCohomolSpectralSeq}. \end{proof}

\section*{Acknowledgements}
This work was supported by Fundación Séneca (22004/PI/22) and the Spanish Government Grant PID2020-113206GBIO0 funded by MCIN/ AEl/ 10.13039/ 501100011033.
 The first-named author was partially supported by 
Funda\c c\~ao de Amparo \`a Pesquisa do Estado de S\~ao Paulo (Fapesp), process n°:  2020/16594-0, and by  Conselho Nacional de Desenvolvimento Cient\'{\i}fico e Tecnol{\'o}gico (CNPq), process n°: 312683/2021-9. The second-named author
was partially supported by CMUP, member of LASI, which is financed by national funds through FCT --- Funda\c{c}\~ao para a Ci\^encia e a Tecnologia, I.P., under the project with reference UIDB/00144/2020, and by the Funda\c{c}\~ao para a Ci\^encia e a Tecnologia (Portuguese Foundation for Science and Technology) through the project PTDC/MAT-PUR/31174/2017. 
 Both the first-named and the second-named authors would like to thank the Department of Mathematics of the University of Murcia for its warm hospitality during 
their visits.  The first-named author would also like to thank the Department of Mathematics of the University of Porto for its cordial hospitality during his visit.

    \bibliography{bibl-pact}{}
	\bibliographystyle{acm}
	
\end{document}